\documentclass[a4paper,11pt]{article}
\usepackage[utf8]{inputenc}
\usepackage{tikz}
\usepackage{amsthm}
\usepackage{amsmath,stmaryrd, mathrsfs}
\usepackage{amsfonts}
\usepackage{amssymb}
\usepackage{dsfont}
\usepackage{tikz-cd}
\usepackage{subcaption}
\usepackage{float}
\usepackage{url}
\usepackage{hyperref}
\usepackage{booktabs}
\usepackage{import}
\usepackage{enumitem}
\usepackage[normalem]{ulem}
\usepackage{soul}
\usepackage{upgreek}
\usepackage{wrapfig}
\usepackage{pgfplots}

\usepackage{geometry}
\newtheorem{theorem}{Theorem}[section]
\newtheorem{lemma}[theorem]{Lemma}
\newtheorem{proposition}[theorem]{Proposition}
\newtheorem{corollary}[theorem]{Corollary}
\newtheorem{remark}[theorem]{Remark}

\newtheorem{assumption}{Assumption}
\newtheorem{example}{Example}

\def\proj{\mathbf{\Pi}}
\def\embd{\pi} 
\def\embdcomb{\pi} 
\def\tdlim{N/L\to \rho}
\def\weakconv{\stackrel{d}{\to}}

\newcommand{\NN}{\mathbf{N}}

\newcommand{\mG}{\mathcal{G}}
\newcommand{\mL}{\mathcal{L}}

\newcommand{\mO}{\mathcal{O}}

\newcommand{\mM}{\mathcal{M}}
\newcommand{\mT}{\mathcal{T}}

\newcommand{\mD}{\mathcal{D}}

\newcommand{\mR}{\mathcal{R}}

\newcommand{\mA}{\mathcal{A}}

\newcommand{\mZ}{\mathcal{Z}}
\newcommand{\mK}{\mathcal{K}}

\newcommand*{\ud}{\mathrm{\,d}}
\newcommand{\supp}{\operatorname{supp} } 

\title{Modulated Poisson--Dirichlet diffusions arising from inclusion processes
with a slow phase}
\author{Simon Gabriel\footnote{Email:\href{mailto:simon.gabriel@uni-muenster.de}{simon.gabriel@uni-muenster.de}, University of M\"unster, DE}
}
\date{}

\begin{document}

\maketitle


\begin{abstract}
We study mean--field inclusion processes with an
additional slow phase, in which particle interactions occur at a vanishing rate
proportional to the inverse system size. 
In the thermodynamic limit, such systems exhibit condensation
at high particle density, forming clusters of diverging size.
Our main result provides convergence in law of inclusion processes to a novel two-component
infinite-dimensional stochastic diffusion, describing the co-evolution of the
solid condensed and microscopic fluid phase.
 In particular, we establish non--trivial mass
exchange between the two phases. 

The resulting scaling limit
extends the Poisson-Dirichlet diffusion
(Ethier and Kurtz, 1981),
introducing an additional control process that modulates its parameters. 
Our result builds on classical estimates of generator differences, which
in this setting yield non-vanishing deterministic error bounds.
We provide the missing probabilistic ingredient by showing
instantaneous condensation, with particle clusters concentrating  on a
vanishing volume fraction immediately.  
We further establish the well-posedness of the limiting dynamics generally
 as Feller processes on a compact state space.
\end{abstract}



\noindent\hspace{.1cm}{\small\textit{{Keywords.}} Poisson--Dirichlet, Inclusion Process,
Condensation, Infinitely--Many--Neutral--Alleles.}

\noindent\hspace{.1cm}{\small\textit{{MSC classification.}} Primary: 60K35,
60J25, Secondary:
82C22, 82C26.}


\setcounter{tocdepth}{2}
\begingroup
  \hypersetup{hidelinks}
  \tableofcontents
\endgroup

\section{Introduction}

The classical inclusion process with mean--field
interactions (without a slow phase) on a graph of
size $L$ is a conserved particle
system of $N$ indistinguishable particles with
generator 
\begin{equation}\label{eq_classic_IP}
\begin{aligned}
\mathfrak{L}_{L, N} f( \eta) =
\sum_{i,j =1}^{L}  \eta_{i}\big( \eta_{j} +  \tfrac{
\Theta}
{L} \big) [ f ( \eta^{i,j}) - f(
\eta)]\,,
\end{aligned}
\end{equation}
acting on functions on the state space $\Omega_{L,N}: = \{ \eta \in
\mathbf{N}_{0}^{L}\,:\, \sum_{i =1}^{L} \eta_{i} =N\}$. Here, $ \Theta >0 $ and  $\eta^{i,j}$
denotes the configuration $ \eta + e_{i} - e_{j} $, where one particle jumped
from $ i $ to $j$.
The generator \eqref{eq_classic_IP} represents two competing dynamics. On 
one hand, the rate $ \eta_{i} \eta_{j} $ encourages particles to cluster
together: Particles jump at unit rate and are attracted by one another at the
same rate. This leads to local accumulations, or clusters, of particles
attracting other particles at a large rate. On the other hand, a dissipative effect
causes every particle to jump to a uniformly
chosen position $j=1,\ldots, L$ at rate $ \tfrac{\Theta}{L}$.
The inclusion process, introduced in \cite{GKR07} as the dual of an energy transport
model, can also be understood in the context of population genetics. There it
describes a Moran-type population \cite{Mor58} of $N$ individuals and $L$ types, where $ \tfrac{\Theta}{L}$
represents the mutation rate, and resampling occurs at rate one.

The interest lies in studying the scaling behaviour of the system
\eqref{eq_classic_IP} in the \emph{thermodynamic limit} $N,L \to
\infty$ such that the density $\tdlim >0$ remains conserved, often abbreviated by simply
writing $ \tdlim$.
The attractive rates in
\eqref{eq_classic_IP} lead to particles clustering into fragments of diverging
size, a phenomenon known as \emph{condensation} \cite{EH05,CG14, EW14, God19}.
Note that the number of sites hosting such diverging chunks must be a
vanishing fraction of the volume, making the phenomenon difficult to analyse.
Instead, a new reference frame and statistics are necessary to observe the
condensate.
One is then concerned about the nature of these diverging clusters: \emph{How many
are there? What can we say about their sizes?}

Condensation in the inclusion process has been studied previously in
\cite{GRV11,GRV13, BDG17,
KS21, Kim21}, for a fixed lattice and diverging number of particles. 
In this case, there is only a single diverging cluster of particles.
The thermodynamic limit $\tdlim$ of \eqref{eq_classic_IP} was studied in
\cite{JCG19, CGG22}, where
the limiting stationary statistics of the condensate were
 determined for a general class of graphs. 
There the cluster sizes at stationarity were
found to be of order $\mO(N)$, with limiting sizes having Poisson--Dirichlet statistics, cf. \eqref{eq_def_PD}.
In particular, the number of diverging particle clusters extends
 non--trivially over infinitely many sites.
For the slowed down inclusion process, with generator $
L^{-1} \mathfrak{L}_{L,N}$,  a system of 
mean--field
equations has been derived 
and  propagation of chaos established \cite{Grosskinsky2019}. 
The evolution of a tagged particle,  driven by the aforementioned
mean-field limit,
 was described in \cite{kout24} in terms of an inhomogeneous dynamic.
See also the works \cite{Schlichting2019,lam2024variational}, 
which investigate more general  evolution equations for particle densities in exchange-driven growth models.
Lastly, the dynamic scaling limit of the inclusion process
\eqref{eq_classic_IP} was
recently identified as the Poisson--Dirichlet diffusion \cite{CGG24}. 
\\

Originally introduced in population genetics as the
\emph{infinitely-many-neutral-alleles model} \cite{EK81},
 the \emph{Poisson–Dirichlet (PD) diffusion} $ \overline{X}$
is a Feller diffusion $ \overline{X}=( \overline{X} (t))_{t \geqslant 0}$ on the Kingman simplex 
\begin{equation}\label{e_kingman}
\begin{aligned}
\overline{\nabla}
:=
\bigg\{x \in [0,1]^{\NN} \,:\, x_{1} \geqslant x_{2} \geqslant \cdots \ \text{and} \
\| x\|_{1}:= \sum_{i = 1}^{ \infty} x_{i} \leqslant 1 \bigg\}\,,
\end{aligned}
\end{equation}
which is equipped with the product topology. 
The stochastic process is 
 characterised by its infinitesimal (pre--)generator
\begin{equation}\label{e_gen_clPD} 
\begin{aligned}
\mA_{\gamma , \theta }f (x)
=
\sum_{i,j =1}^{ \infty} x_{i} (\gamma \, \delta_{i,j} - x_{j} )
\partial_{x_i}\partial_{x_j} f ( x)
- \theta \sum_{i =1}^{\infty} x_{i} \partial_{x_i}f (x)\,,\qquad f \in
\mD_{X}\,,
\end{aligned}
\end{equation}
for some $ \theta >0 $ and the choice $ \gamma =1 $. Here $ \delta_{i,j}$ represents
the Kronecker--\(\delta\), 
\begin{equation}\label{eq_def_DX}
\begin{aligned}
\mD_{X}
:=
\text{subalgebra of $ C ( \overline{\nabla})$ generated by } 1,
\varphi_{2}, \varphi_{3},\ldots\,,
\end{aligned}
\end{equation}
where $ \varphi_{m}(x) := \sum_{i=1}^{\infty} x_{i}^{m}$, and $\mA_{1, \theta}f$
is evaluated on $ \nabla:= \overline{\nabla} \cap \big\{ \sum_{i =1}^{\infty}
x_{i}=1\big\}$ and extended to $ C ( \overline{\nabla}) $ by continuity.

In the seminal work \cite{EK81} Ethier and Kurtz showed that the Poisson--Dirichlet
distribution $\mathrm{PD}( \theta)$, see \eqref{eq_def_PD} below, is the unique invariant distribution for $
\overline{X}$
and that the process appears naturally 
as the scaling limit of discrete Wright--Fisher models.
Since its introduction, the diffusion model has been subject to extensive
research and generalisation,
see for example \cite{EG87,EK87,Sc91,Ethier1992,Pe09,CBERS17}. 
As noted above, 
 the PD diffusion also describes 
the dynamic scaling limit 
of the inclusion process $ (\eta(t))_{t \geqslant 0} = (\eta^{(L,N)}(t) )_{t
\geqslant 0}$, generated by
\eqref{eq_classic_IP} \cite{CGG24}:
\begin{equation}\label{eq_conv_classIP}
\begin{aligned}
\Big(\frac{1}{N} \hat{\eta} (t)\Big)_{t \geqslant 0} \weakconv
\overline{X}\,, \qquad \text{ in } D ( [0, \infty), \overline{\nabla})\,,
\ \text{ as } \ \frac{N}{L} \to \rho\,.
\end{aligned}
\end{equation}
 Here $ \hat{\eta} \in \Omega_{L,N}$ denotes the
configuration $ \eta$ with entries shuffled such that they appear in decreasing
order $ \hat{\eta}_{1} \geqslant \hat{\eta}_{2} \geqslant \ldots \geqslant 
\hat{\eta}_{L}$.
\\

The purpose of this work is to study $ \mA_{ \gamma(\cdot), \theta(\cdot)}$ when both $\gamma=
\gamma(Y (t)) \in [0,1]$
and $\theta= \theta(Y (t)) \geqslant 0$ are functions of a deterministic control process $Y= (Y
(t))_{t \geqslant 0} $ taking values in a compact state space $ S_{Y}$.
We call the jointly generated process $(X,Y)=(X (t),
Y(t))_{t \geqslant 0}$ the \emph{modulated Poisson--Dirichlet
diffusion}, with compact (and metrisable) state space 
\begin{equation}\label{e_statespace}
\begin{aligned}
S := 
\bigg\{ (x,y) \in \overline{\nabla}\times S_{Y}\,:\,
\sum_{i=1}^{\infty} x_{i}
\leqslant  \gamma( y)
 \bigg\}\,,
\end{aligned}
\end{equation} 
which is equipped with the product topology. 
The article's main result is the derivation of $(X,Y)$ as the natural, and
to a certain extent universal, scaling limit of inclusion processes
with a non--trivial fluid phase.
In Proposition~\ref{prop_L_generator}, we also
 provide sufficient criteria for well--posedness of the
joint process $(X,Y)$ on $S$ as a Feller diffusion.\\



The inclusion process with a slow phase is a
generalisation
of the classical inclusion process \eqref{eq_classic_IP}.
To this end, we introduce a threshold $ A \in
\mathbf{N}$ which acts as the boundary of a slow phase in the following
sense: Particles below the threshold $ A$ jump
and attract other particles at a slower (vanishing) rate $ \sim \tfrac{1}{L}$, while
particles above the threshold jump and attract particles at unit rate, as in
\eqref{eq_classic_IP}.
This manipulation of rates traps a fraction of particles in the slow phase
(below $A$), creating a heterogeneous environment.

\begin{figure}[H]
\centering
\captionsetup{width=.8\linewidth}
\begin{subfigure}{0.4\textwidth}
\centering
    {
	\begin{tikzpicture}[scale = 0.6]
\draw (2.2,1.55) node[anchor=north]  {\textcolor{purple}{{\small $A$}}};
\draw (3,0) node[anchor=north]  {{{\scriptsize $1$}}};
\draw (10.2,0) node[anchor=north]  {{{\scriptsize $L$}}};

\draw[very thick, purple] (2.5,1.2) -- (10.5,1.2);

\draw[ thick] (9.8,.15) circle (0.15);

\foreach \x in {3, 3.4, 4.2, 4.6, 5, 5.4, 5.8, 6.2, 6.6, 7, 7.4, 7.8, 8.2,
8.6, 9, 9.4, 10.2} {
    \foreach \y in {0.15, 0.45} {
        \draw[ thick] (\x,\y) circle (0.15);
    }
}

\foreach \x in {3, 3.4, 4.2, 4.6, 5, 5.4, 5.8, 6.2, 6.6, 7, 7.8, 8.2,
8.6, 9, 9.4, 10.2} {
    \foreach \y in { .75} {
        \draw[ thick] (\x,\y) circle (0.15);
    }
}

\foreach \x in {3, 3.4, 4.2, 4.6, 5.4, 5.8, 6.2, 6.6, 7, 7.8, 8.2,
8.6, 9, 9.4, 10.2} {
    \foreach \y in {1.05} {
        \draw[ thick] (\x,\y) circle (0.15);
    }
}

\foreach \x in {3, 3.4, 4.2, 4.6, 5.8, 6.2, 6.6, 7, 7.8, 8.2,
8.6, 9, 9.4, 10.2} {
    \foreach \y in {1.35} {
        \draw[ thick] (\x,\y) circle (0.15);
    }
}

\foreach \x in {3, 3.4, 3.8, 4.2, 4.6, 5, 5.4, 5.8, 6.2, 6.6, 7, 7.4, 7.8, 8.2,
8.6, 9, 9.4, 10.2} {
	\draw[thick] (\x,0) -- (\x,-0.15) {};	
}

\foreach \y in {1.65} {
     \draw[ thick] (3,\y) circle (0.15);
}

\foreach \y in {1.35, 1.65,1.95,2.25,2.55,2.85} {
     \draw[ thick] (4.6,\y) circle (0.15);
}

\foreach \y in {1.35, 1.65,1.95} {
     \draw[ thick] (6.2,\y) circle (0.15);
}

\foreach \y in {1.35, 1.65,1.95,2.25} {
     \draw[ thick] (9,\y) circle (0.15);
}

\foreach \y in {1.35, 1.65,1.95,2.25,2.55} {
     \draw[ thick] (7.8,\y) circle (0.15);
}

\draw[very thick] (2.5,0) -- (10.5,0);

\end{tikzpicture}
    }
\end{subfigure} 
    \caption{
Sample particle configuration $ \eta$ with a slow phase.
}
\label{fig_slowphase}
\end{figure}
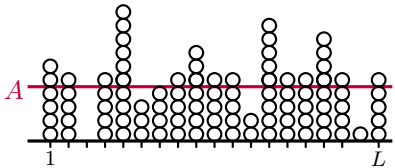

\noindent For instance, let us consider rates of the form 
\begin{equation}\label{eq_IP_example}
\begin{aligned}
\mathfrak{L}_{L, N} f( \eta) =
\sum_{i,j =1}^{L}  \big( (\eta_{i}-A)\mathds{1}_{ \eta_{i} >A} +  \tfrac{
\Theta}
{L}\, 
\mathds{1}_{ 0< \eta_{i} \leqslant A}\big) \big( (\eta_{j}-A)\mathds{1}_{ \eta_{j} >A}  +  \tfrac{
\Theta}
{L}
 \big) [ f ( \eta^{i,j}) - f(
\eta)]\,,
\end{aligned}
\end{equation}
which will act as a leading example
throughout the paper \cite{CGG22}.
If $ A=0$, there is no slow phase to accommodate particles
and we recover the classical inclusion process \eqref{eq_classic_IP}.
For $ A \geqslant 1$,
 sites with fewer than $ A$ particles act as
bottlenecks, restricting the flow of particles between the 
slow  and the fast 
phase (above $A$). 
This leads to slower equilibration, restricted by the exchange of mass between the
two phases.
Consequently, the evolution of particle clusters  will
depend on the available mass in the fast phase. 

The threshold of the slow phase in the stochastic particle system
\eqref{eq_IP_example} can be thought of 
as an activation energy barrier.
For energy quanta (particles) to contribute to another phase with increased
mobility, this threshold has to be overcome.
Let us also mention that the hydrodynamic behaviour of various particle systems with slow boundary reservoirs
has been subject of a series of works, see for example 
\cite{baldasso2017exclusion,bonorino2020hydrodynamics,franceschini2022symmetric}.
Alternatively, from a population genetics perspective, this model (like the classical
Poisson--Dirichlet diffusion) can be seen as a way to capture genetic
diversity. 
The threshold behavior in terms of $A$ 
may then represent the minimum concentration to facilitate reactions or
growth. 
In this direction, we also mention the works
\cite{greven2022spatial,den2022spatially} on seed--bank models
with dormancy.\\

We conclude the introduction by providing a brief description of the main result,
which is presented in detail in the next section.
Theorem~\ref{thm_main} below establishes, under an
embedding $ \embdcomb = \embdcomb^{(L,N)} : \Omega_{L,N} \to \overline{\nabla}
\times S_{Y}$, convergence of inclusion processes with a slow phase:
\begin{equation}\label{e_mainresult}
\begin{aligned}
(\embdcomb ( \eta(t)) \big)_{t \geqslant 0} \weakconv (X,Y)\,, \quad \text{as }
\tfrac{N}{L} \to \rho\,,
\end{aligned}
\end{equation}
where the limit
$(X,Y)$ is given in terms of the modulated PD diffusion.
$ X $ represents the evolution of
cluster sizes in the fast phase (above $A$) and $Y$ the evolution of particle occupations
in the slow phase (below $A$).
Recall that $Y$ modulates the drift
$ \theta$ and relative size $ \gamma$ of the macroscopic (fast) phase $X$. 
In particular, the sequence of inclusion processes exhibits condensation
whenever $ \gamma>0 $.
The convergence \eqref{e_mainresult} will be
proven for a general class of inclusion dynamics,
see Assumption~\ref{ass}. In particular, we allow for varying $\Theta $ and
perturbation of transition rates.
Our result offers new insights into heterogeneous inclusion-type models beyond
the static, equilibrium setting.

\section{Setup and main result}

For the well--posedness of the process $(X,Y)$, we restrict ourselves to the
case where $Y= (Y (t))_{t \geqslant 0}$ is 
 the solution of a system of differential
equations of the form
\begin{equation}\label{eq_ode}
\begin{aligned}
\frac{\ud}{\ud t} Y_{k}(t) = b_{k}(Y(t))\,, \quad k =0, \ldots, A-1\,,
\end{aligned}
\end{equation}
$ A \in \NN $, on the
simplex  
\begin{equation}\label{e_def_K}
\begin{aligned}
 S_{Y}:=
\bigg\{y \in [0, 1]^A\,:\, \sum_{k=0}^{A-1} y_{k} \leqslant 1 \bigg\}\,.
\end{aligned}
\end{equation}
Furthermore, we assume that the $ b_{k} \in C(S_{Y})$ are
Lipschitz and satisfy
\begin{equation}\label{e_cond_b1}
\begin{aligned}
b_{k}(y) \geqslant 0 \,, \text{ for all $y$ such that } y_{k}=0\,,
\end{aligned}
\end{equation}
\begin{equation}\label{e_cond_b2}
\begin{aligned}
\text{and } \quad \sum_{k =0}^{A-1} b_{k}(y) \leqslant  0\,,  \text{ for all } y \text{ such that }
\, \sum_{k =0}^{A-1} y_{k} = 1\,,
\end{aligned}
\end{equation}
ensuring that the process $Y $ does not exit $S_{Y}$.
Note that \eqref{eq_ode} has a unique solution by the Picard–Lindel\"of theorem.

In view of \eqref{eq_IP_example}, $Y_{k}(t)$ will represent the relative ratio
of sites that are occupied by $k$ particles:
\begin{equation}\label{eq_Y_motiv}
\begin{aligned}
Y_{k}(t) = \lim_{\tdlim} \frac{1}{L} \#_{k} \eta(t)\,, \qquad \text{where }\quad
\#_{k} \eta:=\sum_{i =1}^{L}
\mathds{1}_{\eta_{i}= k}\,, 
\end{aligned}
\end{equation}
which holds all the necessary information we need from the slow phase.
While being helpful to motivate its construction, the exact interpretation of $Y$ does not play a role in the
definition of the modulated process $(X,Y)$ at this stage.
Instead, it will be convenient to write the system of ODEs \eqref{eq_ode} in terms of
the generator
\begin{equation}\label{eq_gen_G}
\begin{aligned}
\mG h(y) = 
 \sum_{k =0 }^{A-1} b_{k}(y) \partial_{y_{k}} h(y)\,,
\quad \text{ for all } h \in \mD_{Y}\,,
\end{aligned}
\end{equation}
acting on the  polynomials
\begin{equation}\label{eq_def_DY}
\begin{aligned}
 \mD_{Y}:=\mathrm{span}\Big\{ y \mapsto y^{ \bf n} = \prod_{k =0}^{A-1}
y^{\bf n_{i}} \,, \ {\bf n} \in
\mathbf{N}^{A}_{0}\Big\}\,.
\end{aligned}
\end{equation}

\subsection{The modulated Poisson--Dirichlet diffusion}
 
Our first result concerns the well--posedness of the Feller process $(X,Y)$ on
$S$ \eqref{e_statespace}, when
the infinitesimal description $ \mA_{\gamma (y),\theta (y)}$ of $X (t)$ depends
on the state $Y(t)=y $.
The natural candidate for the combined generator is given in terms of
\begin{equation}\label{def_L_gen}
\begin{aligned}
\mL f(x,y) := \mA_{ \gamma(y), \theta(y)}f (\cdot, y )(x) + \mG f(x, \cdot)(y)\,,
\end{aligned}
\end{equation}
acting on functions $f$ in
\begin{equation}\label{e_dom_L} 
\begin{aligned}
\mD := 
\mathrm{span}\big\{
f \cdot g\,:\, f \in \mD_{X}\,, \ g \in \mD_{Y}
\big\}
\subset C (S)\,,
\end{aligned}
\end{equation}
where $ \mD_{X}$ and $ \mD_{Y}$ were introduced in \eqref{eq_def_DX} and
\eqref{eq_def_DY}.
Throughout the paper, we adopt the convention that $ \mA_{\gamma(y), \theta(y)}f $ in  \eqref{def_L_gen}
is evaluated on
\begin{equation}\label{eq_boundaryS}
\begin{aligned}
\partial S := \Big\{
(x,y) \in S\, : \, 
\sum_{i =1}^{\infty} x_{i} = \gamma (y)
\Big\}
\end{aligned}
\end{equation}
and extended to $S$ by continuity. This agrees with the convention for $ 
\mA_{1, \theta}$, employed in \cite{EK81}, see the comment below
\eqref{eq_def_DX}.

\begin{assumption}\label{ass_def_L}
\ 
\begin{enumerate}[label=(\alph*)]

\item 
The level set 
\begin{equation}\label{eq_SyStar}
\begin{aligned}
S_{Y}^{*}:= \{y \in S_{Y}\,:\, \gamma(y)=0\} = S_{Y}\setminus \supp \gamma
\end{aligned}
\end{equation}
is absorbing for  $\mG $,
in the sense that $ y \in S_{Y}^{*}$ implies $ b_{k}(y)=0$,
for all $k =0, \ldots, A-1$.

\item The drift coefficients $
(b_{k})_{k=0, \ldots, A-1 }$ in \eqref{eq_gen_G} are
Lipschitz continuous and satisfy the
boundary conditions
\eqref{e_cond_b1} and \eqref{e_cond_b2}.

\item The modulation parameters of  $\mA_{ \gamma(\cdot), \theta(\cdot)} $ in
\eqref{e_gen_clPD} are Lipschitz continuous, i.e. they satisfy  $  \theta \in C (S_{Y}, [0,
\infty))$ and $\gamma \in  C(S_{Y}, [0,1]) $. Moreover,
\begin{equation*}
\begin{aligned}
 \gamma \in C^{1}(\overline{ \supp \gamma}, [0,1])\,.
\end{aligned}
\end{equation*}

\item The function $ \beta:  \supp \gamma \to [0, \infty)$ with
\begin{equation}\label{eq_def_beta}
\begin{aligned}
\beta(y):= \sum_{k=0}^{A-1}
b_{k}(y) 
\frac{\partial_{y_{k}} \gamma(y)}{ \gamma(y)}
+
\theta (y) 
\,,
\end{aligned}
\end{equation}
is Lipschitz continuous and non--negative.
We will denote by $ \overline{\beta} \in C ( S_{Y}, [0, \infty) )$ an arbitrary continuous extension
of $ \beta$. 

\end{enumerate}

\end{assumption}

The following proposition states that $\mL$ is indeed  
the generator of a Feller process $(X,Y)$ on $ S$, which we  call the
\emph{modulated Poisson--Dirichlet diffusion}.

\begin{proposition}[The modulated PD diffusion]\label{prop_L_generator}
Let 
$\mA_{ \gamma, \theta}$ and $\mG $ be given by \eqref{e_gen_clPD} and
\eqref{eq_gen_G},
with $( b_{k})_{k =0, \ldots, A-1}$, $ \gamma$ and $ \theta$
satisfying Assumption~\ref{ass_def_L}.
Then the closure of the linear operator~\eqref{def_L_gen},
\begin{equation*}
\begin{aligned}
\mL f = \mA_{ \gamma, \theta} f + \mG f\,, \qquad  f \in \mD\,, 
\end{aligned}
\end{equation*}
 generates  
a Feller process $ (X (t), Y (t))_{t \geqslant 0}$ with
sample paths in $ C ([0, \infty), S)$.
\end{proposition} 

We will prove the proposition by showing that the martingale problem for $\mL$
is well--posed. 
Existence of solutions for the martingale problem follows by approximation of
finite-dimensionsional processes, Lemma~\ref{lem_mp_exists}. 
Uniqueness is a consequence of an explicit hierarchy of equations,
characterising the moments of any solution to the martingale problem uniquely,
see Lemma~\ref{lem_hierarchy}.
The proof of Proposition~\ref{prop_L_generator} is stated in the beginning of
Section~\ref{sec_mod_PD}.

Assumption~\ref{ass_def_L}.(a)~and~(b) are natural in view of $
\mA_{ \gamma, \theta}$, $\mG $ and the definition of the state space $S$, because $
\gamma(y)=0$ necessarily implies $ x=0$.
On the other hand, Assumption~\ref{ass_def_L}.(c)~and~(d) are technical requirements
to represent the modulated PD diffusion in terms of a rescaled classical PD diffusion on
$ \overline{\nabla} \times S_{Y}$, see Lemma~\ref{lem_prop}.
Regarding the existence of the process, it should be possible to lift the last
two assumptions by directly constructing the process on $S$.
\\

The modulated PD diffusion $(X,Y)$ of Proposition~\ref{prop_L_generator} provides a
natural generalisation of the classical PD diffusion $ \overline{X}$, generated
by \eqref{e_gen_clPD}.
A rescaling argument (Lemma~\ref{lem_prop}) shows that the marginal process $X$ of $(X,Y)$ can
be linked to the PD diffusion $
\overline{X}$ on $ \overline{\nabla}$, generated by $ \mA_{1, \beta }$ with $
\beta=  \beta( \theta)$ introduced in \eqref{eq_def_beta}:
\begin{equation*}
\begin{aligned}
X(t) = \gamma(Y(t)) \overline{X}(t)\,.
\end{aligned}
\end{equation*}
Thus, several properties of the classical and modulated PD diffusion are shared, by analogy.
For example, while for the classical PD diffusion $ \overline{X}$ almost surely
\begin{equation*}
\begin{aligned}
 \sum_{i =1}^{\infty} \overline{X}_{i}(t)=1 \,, \qquad \forall t >0 \,,
\end{aligned}
\end{equation*}
for
any initial
condition, cf. \cite[Theorem~2.6]{EK81},  $ (X,Y)$ is supported for all
positive times on $\partial S $ \eqref{eq_boundaryS}, see Corollary~\ref{cor_concentration}. 
Moreover, the unique invariant for the classical PD diffusion $ \overline{X}$, generated by
$ \mA_{1, \Theta}$, is the Poisson--Dirichlet distribution $\mathrm{PD}(\Theta
)$.
Similarly, 
the extremal stationary states of $X$ are given in terms of 
\begin{equation*}
\begin{aligned}
\mathrm{PD}_{[0, \gamma(y^{*})]}( \theta ( y^{*}))
\,, 
\end{aligned}
\end{equation*}
for all stationary states $y^{*} \in S_{Y}$ of $ \mG$.

The PD distribution $ \mathrm{PD} ( \Theta)$
 is a probability measure
on $ \overline{\nabla}$, first introduced by Kingman \cite{Kin75} as natural
limit of Dirichlet distributions. 
Another elegant construction of $\mathrm{PD}( \Theta)$ is based on a
stick--breaking procedure: Given a family of independent $\mathrm{Beta}(1,
\Theta )$ random variables $( U_{i})_{i \in \mathbf{N}}$, construct $ (
V_{i})_{i \in \mathbf{N}}$ recursively by 
\begin{equation}\label{eq_def_PD}
\begin{aligned}
V_{1} := U_{1}
\qquad \text{and}\qquad 
V_{i +1} := U_{i+1} \Big( 1 - \sum_{j=1}^{i} V_{j}\Big)\,.
\end{aligned}
\end{equation}
Then rearranging $ (
V_{i})_{i \in \mathbf{N}}$ into the ordered vector $ (\hat{V}_{1},
\hat{V}_{2} , \ldots) \in \overline{\nabla}$, with $
\hat{V}_{i} \geqslant \hat{V}_{i+1}$, yields a random variable with law
$\mathrm{PD}( \Theta)$ \cite{Fen10}.
The distribution $\mathrm{PD}_{[0, \gamma]} ( \Theta)$ is then constructed by
multiplying every entry of $\mathrm{PD}( \Theta)$ with $
\gamma$, with the special case $ \mathrm{PD}_{\{0\}}(
\Theta)= (1,0,0,\ldots)$ almost surely.\\


The Poisson–Dirichlet diffusion was previously extended to the two-parameter
setting in \cite{Pe09,Feng2009}, yielding a stochastic process with the unique invariant measure $ \mathrm{PD}
( \alpha, \theta)$. 
 A Wright–Fisher derivation of this process has been provided in \cite{CBERS17}. 
See also \cite{Forman2021,Forman2022,Forman2023} for representations of the two-parameter process as Fleming–Viot-type measure-valued diffusion.
In contrast, the modulated Poisson–Dirichlet diffusion of Proposition~\ref{prop_L_generator} 
remains within the one-parameter framework but introduces a generalisation through varying parameters governed by a control process. 
 
For the present work, it suffices to restrict to a deterministic control process $ (Y(t))_{t \geqslant 0}$
on the particular state space $ S_{Y}$. However, it should be possible to
generalise the setup to stochastic control processes on more
general spaces.

\subsection{Scaling limit of the inclusion process with a slow phase}

We consider the following general form of the inclusion process on the complete
graph: 
\begin{equation}\label{e_IP_gen}
\begin{aligned}
\mathfrak{L}_{L, N} f( \eta) :=
\sum_{i,j =1}^{L} u_{1} ( \eta_{i}) u_{2}( \eta_{j}) [ f ( \eta^{i,j}) - f(
\eta)]\,,
\end{aligned}
\end{equation}
where $ u_{1}= u_{1, L}, u_{2}= u_{1, L} \geqslant 0$
 satisfy the following generalisation
of the rates \eqref{eq_IP_example}:
\begin{assumption}\label{ass}
Assume there exists  $ A \in \mathbf{N}_{0}$ and $ \zeta_{L}= \mO (
\tfrac{1}{L})$ such that transition rates can be divided into two phases.
    \begin{enumerate}[label=(\alph*)]
        \item The \emph{slow phase} of positions with no more than $A$ particles on
them. Their transition rates scale as the inverse of the system size,
in the sense that there exist $r_{k}
\geqslant 
q_{k} > 0$ such that 
\begin{equation}\label{eq_def_qr}
\begin{aligned}
            | u_1(k) \, L - q_{k}|&\leqslant
\zeta_{L}\,, \quad \text{for all } k=1, \ldots, A\,, \\
 | u_{2}(k) \, L - r_{k}| & \leqslant \zeta_{L}\,, \quad \text{for all } k=0, \ldots, A \,.
\end{aligned}
\end{equation}
We will write $ \overline{q}:= \max_{k =0, \ldots, A} q_{k}$ and $
\underline{q}:= \min_{k =1, \ldots, A} q_{k}$, and define $ \overline{r}$
and $\underline{r}$ in the same manner.
Moreover, we set $ q_{0}:=0$.

\item The \emph{fast phase} of sites 
with occupation exceeding $A$ particles. Their rates are uniformly close to the
 number of particles that exceed the threshold $ A$:
       \begin{align}
            \sup_{n >A}
            \big|
                u_{\iota}(n) - (n-A)
            \big| \leqslant \zeta_{L} \,, \qquad \text{for}\ \iota=1,2\,.
        \end{align}
        
    \end{enumerate}
\end{assumption}

First, we embed particle configurations $ \eta
\in \Omega_{L,N}$ into the common state space $ \overline{\nabla}\times
S_{Y}$. This can be
done similar to
\eqref{eq_conv_classIP}, where no slow phase was present, with the caveat that now we have to introduce observables tracking the
evolution of slow particles. To this end, we define the maps $\embdcomb=
\embdcomb^{(L,N)}:\Omega_{L,N} \to \overline{\nabla} \times S_{Y} $ with 
\begin{equation}\label{e_def_embd_direct}
\begin{aligned}
\embdcomb ( \eta) := 
 \big( \tfrac{1}{N} \widehat{ \eta}^{\,+}, \tfrac{1}{L} (\#_{0}
\eta, \ldots , \#_{A-1} \eta )\big)\,,
\end{aligned}
\end{equation}
where $\#_{k}$  was introduced in \eqref{eq_Y_motiv} and $
\eta_{i}^{+}:=\max\{ \eta_{i}- A,0\}$ denotes the number of particles
 above the
threshold $ A$.
Recall that $ \widehat{\cdot}$ reorders entries of a vector in decreasing order.
Note that $\embdcomb (\eta)$ does not necessarily lie in $S$, whose definition
will depend on $ \rho \neq \tfrac{N}{L} $.
 \\

For the inclusion process with a slow phase,  $ Y(t) $ will represent the limiting particle occupations of $ \eta
(t)$ in the slow phase, see \eqref{e_def_embd_direct}.
It is a simple observation that particles that are not present in the slow phase
must necessarily be located in
the fast phase.
In view of the embedding $\embdcomb$, it is therefore natural to
choose the function $ \gamma$ in
the definition of the state space $S$ \eqref{e_statespace} 
as
\begin{equation}\label{e_def_gamma}
\begin{aligned}
\gamma (y) := \bigg(1- \frac{1}{ \rho} \sum_{k=0}^{A} k y_{k}\bigg)_{+} \,,
\end{aligned}
\end{equation}
where $(\, \cdot \, )_{+}= \max\{\, \cdot\, , 0\}$ denotes the positive part.
Note that the quantity inside $ ( \, \cdot \, )_{+}$ is only negative if the
configuration $y$
exceeds the particle density $ \rho$.
 \\

We now present the article's main result: The convergence of the
inclusion process with a slow phase to the modulated Poisson--Dirichlet
diffusion $(X,Y)$, with a deterministic control $ (Y (t))_{t \geqslant 0}$
 generated by 
\begin{equation}\label{eq_G_IP}
\begin{aligned}
\mG h (y) =
\rho\, \gamma(y)
\sum_{k =0}^{A-1}
\big(
q_{k +1} y_{k +1} + r_{k -1} y_{k -1}
- (q_{k}+ r_{k}) y_{k}
\big)
\partial_{y_{k}}h (y)\,.
\end{aligned}
\end{equation}
The constants $ (q_{k})_{k =0, \ldots, A}$ and $ (r_{k})_{k =0, \ldots, A}$ were
defined in \eqref{eq_def_qr}. 
Moreover, we set  $ y_{A}:= 1 - ( y_{0}+ \cdots + y_{A-1})$
and $y_{-1}:=0$.

\begin{theorem}\label{thm_main}
Let $ (\eta^{(L,N)}(t))_{t \geqslant 0}$ be the inclusion process generated by  
$ \mathfrak{L}_{L, N} $ \eqref{e_IP_gen}, with Assumption~\ref{ass}
 satisfied for some $ A \in \mathbf{N}_{0}$. 
If $ \embdcomb ( \eta^{(L,N)} (0))\to (x, y) \in S$ as $ \tfrac{N}{L} \to \rho >0$,
then
\begin{equation*}
\begin{aligned}
\big(\embdcomb ( \eta^{(L,N)}(t))\big)_{t \geqslant 0} 
\weakconv 
\big( X (t), Y (t)\big)_{t \geqslant 0}\,, \qquad \text{in }\ \
 D ([0, \infty),
\overline{\nabla} \times S_{Y})\,.
\end{aligned}
\end{equation*}
Here, $( X (t), Y (t))_{t \geqslant 0}$ denotes the modulated
Poisson--Dirichlet diffusion started
from $(x,y)$, see Proposition~\ref{prop_L_generator}, with 
$\mG $ as in \eqref{eq_G_IP},
\begin{equation}\label{e_def_theta}
\begin{aligned}
\gamma( y) = 
\bigg(
1- \frac{1}{ \rho} \sum_{k=0}^{A} k y_{k}
\bigg)_{+}
\quad \text{and}\quad
\theta(y) =
\sum_{k =0}^{A } (r_{k}- q_{k}) y_{k}\,,
\end{aligned}
\end{equation}
where $ y_{A}:= 1- \sum_{k=0}^{A-1} y_{k}$.
\end{theorem}

Theorem~\ref{thm_main} fully characterises the scaling limit of the condensed
(fast) phase of the inclusion process, and its interaction with the slow phase.
Our result demonstrates how microscopic changes in interaction rates can induce nontrivial
macroscopic transport effects.
It shows that while the slow phase traps particles for a longer time,
equilibration occurs on the same time scale as the evolution of
cluster sizes in the condensate. 
In particular, it provides an explicit description of the (deterministic)
evolution of mass $ \gamma (Y (t))$ in the condensed phase.
\\

Let us put the result into context. We recall the specific rates
\eqref{eq_IP_example} in which case $ \theta( \cdot) = \Theta >0$ is fixed.
The generator $\mG $ of the deterministic process $(Y (t) )_{t \geqslant
0}$ then corresponds to the system of differential equations
\begin{equation}\label{eq_ODEs_example}
\begin{aligned}
\frac{\ud Y_{k}(t)}{\ud t} =
 \Theta \,\rho\, \gamma( Y(t))\, 
 \big( Y_{k+1}(t) + Y_{k-1}(t) - (2- \delta_{0,k})
Y_{k}(t)\big)\,, \qquad k =0, \ldots, A-1\,,
\end{aligned}
\end{equation}
where once again $ Y_{-1}(t)=0$ and $ Y_{A}(t) = 1 -( Y_{0}(t) +\cdots +
Y_{A-1}(t) )$. 
In particular, for $ A=0$ the process $(Y(t))_{t \geqslant 0}$ does not exist
and $ \gamma \equiv 1$.
Hence, the system fully condenses into
macroscopic clusters for all $t>0$. 
Therefore, Theorem~\ref{thm_main} recovers the convergence of the inclusion process to the classical PD diffusion
generated by $ \mA_{1, \Theta}$  \cite{CGG24}, see also~\eqref{eq_conv_classIP}.

In a broader context, we note that the finite system of equations associated to $ \mG $ is of
the same form as the infinite system of mean--field equations in
\cite{Grosskinsky2019, Schlichting2019,lam2024variational}, with
the important difference that \eqref{eq_ODEs_example} is expressed in a closed form.
While these mean--field equations  correspond to the evolution driven by the inclusion rate $\sim  \eta_{i}
\eta_{j}$ within the fast phase, \eqref{eq_ODEs_example} captures the
interaction between the fast and the slow phase.

\begin{example}\label{exmpl_A1}
Consider the inclusion process generated by \eqref{eq_IP_example} with $ A=1$.
Then the system of ODEs \eqref{eq_ODEs_example} simplifies to 
\begin{equation}\label{e_A1_ode}
\begin{aligned}
\frac{\ud Y_{0}(t)}{\ud t} 
=
 \Theta \,\rho\, \gamma( Y_{0}(t))
 \big( 1  - 
2 Y_{0}(t)\big)\,.
\end{aligned}
\end{equation}
The evolution of the relative mass $ \gamma( t) := \gamma(
Y_{0}(t)) $ in the fast phase can be explicitely expressed, using
 $ \gamma(Y_{0}(t)) = 1- \tfrac{1}{\rho} (1-Y_{0}(t)) $. If $ \gamma(0)=1$,
i.e.~almost all particles start in the fast phase, then
\begin{equation*}
\begin{aligned}
\gamma(t) = 
\frac{1- 2 \rho}{ e^{\Theta t (1-2 \rho)} - 2 \rho} \,.
\end{aligned}
\end{equation*}
Thus, the relative mass in the fast phase equilibrates over time, and evolves at
the same time scale as the PD--diffusion limit.
In particular, we can read off a critical density $ \rho_{c} = \tfrac{1}{2}$ at which
condensation starts to occur, such that 
\begin{equation*}
\begin{aligned}
\lim_{t \to \infty} \gamma(t) = \big(1- \tfrac{1}{2 \rho}\big)_{+} \,.
\end{aligned}
\end{equation*}
Figure~\ref{fig_massevolution} below depicts this non--trivial mass evolution of the
fast phase. 
\end{example}

\begin{figure}[H]
\centering
\captionsetup{width=.8\linewidth}
\begin{tikzpicture}
\begin{axis}[xmin=0, ymin=0, xmax=3,ymax=1, samples=100,width =
0.75\textwidth,height=0.35\textwidth,restrict y to domain=0.0001:1.1,xlabel={$t$},
  ylabel={$\gamma (t)$},ylabel style={rotate=-90}]
  \addplot[purple, ultra thick] (x,0.005);
  	\addplot[black, ultra thick]  {1/(2-(2-1/0.025)* exp(-x))};
	\addplot[black, ultra thick]  {1/(2-(2-1/0.25)* exp(-x))};
	\addplot[black, ultra thick]  {1/(2-(2-1/0.65)* exp(-x))};
	\addplot[black, ultra thick]  {1/(2-(2-1/1)* exp(-x))};
  \addplot[purple,  ultra thick] (x,0.5);
\end{axis}
\end{tikzpicture}
    \caption{
Evolution of $ \gamma(t) = \gamma(Y (t))$ in the case of \eqref{eq_IP_example} with $A=1$,
$\Theta = 1$ and super--critical density $ \rho= 1> \tfrac{1}{2} =\rho_{c}$. 
Note that $ \gamma = 0$ is an equilibrium.
}
\label{fig_massevolution}
\end{figure}
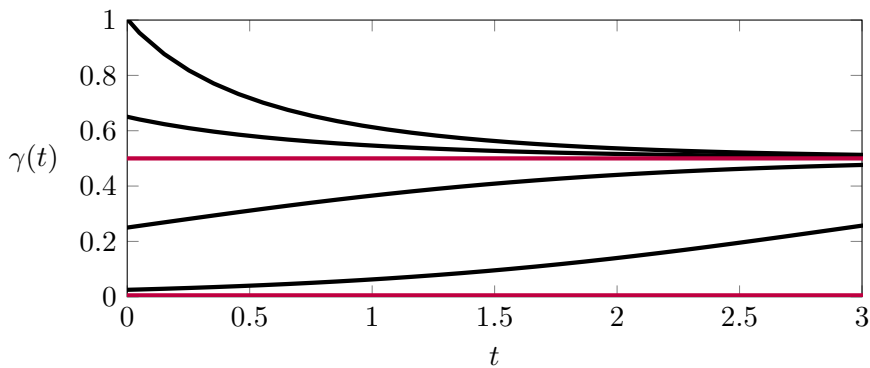

While the infinitesimal evolution of $ Y(t)$ is explicitly given in
terms of the non--linear system \eqref{eq_ODEs_example},
 the exact form of the
solution $ (Y(t))_{t \geqslant 0} $, and therefore $ \gamma(t)$, is not immediate. 
In the following, we will only comment on the long--time behaviour of the limiting process
$(X,Y)$, in particular  $\lim_{t \to
\infty} \gamma (Y (t))$.

First, we point out that the
 stationary states of $\mG$ in \eqref{eq_G_IP} are given by $
S_{Y}^{*} $ \eqref{eq_SyStar} and $ \overline{y}$:
\begin{equation}\label{eq_productmeas}
\begin{aligned}
\overline{y}_{k}:= \frac{1}{Z} \prod_{\ell =1}^{k} \frac{r_{\ell
-1}}{q_{\ell}} \,, \qquad k =0, \ldots, A\,,
\end{aligned}
\end{equation}
where $Z := \sum_{k=0}^{A} \prod_{\ell =1}^{k} \frac{r_{\ell
-1}}{q_{\ell}}$ is a normalising constant.
The coordinate 
 $ \overline{y}$ can be interpreted as a probability counting measure
$ \nu := \sum_{k =0}^{A} \overline{y}_{k} \delta_{k}$, because $ \| \overline{y} \|_{1}=1$.
In the case of our leading example \eqref{eq_IP_example}, the law $ \nu $
is the uniform
distribution on $ \{0, \ldots, A \}$.

In \cite{CGG22}, it was shown that the 
 mean of the probability measure $ \nu$ determines a critical density
threshold
\begin{equation}\label{eq_rhoc}
\begin{aligned}
\rho_{c}:=  \sum_{k =0}^{A} \overline{y}_{k} k =
\mathbf{E}_{\nu}[\mathrm{id}]\,,
\end{aligned}
\end{equation}
above which condensation at stationarity starts to occur.  
More precisely, it was shown that
for $ \eta^{(L,N)}$ distributed according to the unique invariant measure of $
\mathfrak{L}_{L,N}$ 
\begin{equation}\label{eq_stat_PD}
\begin{aligned}
\frac{1}{N} \hat{\eta}^{(L,N)} \weakconv \mathrm{PD}_{[0, \overline{\gamma}]} ( \Theta)\,,
\end{aligned}
\end{equation}
with $ \overline{\gamma}  =\big( 1 - \tfrac{\rho_{c}}{ \rho}\big)_{+}$, see also Example~\ref{exmpl_A1} above.
We expect this
phase transition at $ \rho_{c} $ to persist in the non--equilibrium setting.
Theorem~\ref{thm_main} provides confirmation that this is indeed the case:
\begin{itemize}
\item If $ \rho< \rho_{c}$ then $ \overline{y}$ can never be
attained, as it will otherwise violate conservation of mass. Hence, we
necessarily have $
\lim_{ t \to \infty} \gamma(Y(t))=0$.

\item If $ \rho> A $ then $ S_{Y}^{*} = \emptyset $, because
$\gamma(y) = 1 - \frac{1}{ \rho} \sum_{k =0}^{A}k y_{k} \geqslant 1-
\frac{A}{\rho}  >0$, for all  $y \in S_{Y}$. 

Therefore, the only feasible equilibrium state of $
\mG$ is $
\overline{y}$. 
Consequently, 
\begin{equation}\label{eq_gam_evo_supercrit}
\begin{aligned}
 \lim_{ t \to \infty } \gamma( Y (t)) =  1 - \tfrac{\rho_{c}}{
\rho}= \gamma ( \overline{y}) \,, \quad \text{whenever } \rho> A\,.
\end{aligned}
\end{equation}

\end{itemize}
We are convinced that a detailed analysis of $\mG$, or the associated system of
ODEs, will show that $ \lim_{ t \to \infty} Y(t) = \overline{y}$ whenever $
\rho> \rho_{c}$ and $
\gamma (Y (0)) >0$. In particular, \eqref{eq_gam_evo_supercrit} is expected
to hold for all $ \rho > \rho_{c}$, as seen in Example~\ref{exmpl_A1}.
Finally, we notice that \eqref{eq_gam_evo_supercrit} in combination with
Lemma~\ref{lem_prop} implies
\begin{equation*}
\begin{aligned}
\big(X (t),Y(t)\big) \weakconv \mathrm{PD}_{[0,\gamma( \overline{y})]} ( \theta
( \overline{y}))
\otimes \delta_{ \overline{y}}\,,\quad \text{as}\quad 
t \to \infty\,,
\end{aligned}
\end{equation*}
which again is expected to be true for all $ \rho> \rho_{c}$, provided $
\gamma (Y (0)) >0$.\\

We conclude this section with a discussion of potential generalisations and
future directions regarding Theorem~\ref{thm_main}. 
A natural progression would be to replace the underlying complete graph with an
 sequence of diverging graphs, such as an expanding nearest-neighbour lattice.
 Even for the classical inclusion process \eqref{eq_classic_IP}, the present
setup is unsuitable, as the process loses its Markov property when particle configurations are embedded into $\overline{\nabla}$. 
Moreover, it is known that the random-walk component (in terms of $ \Theta$)
and the inclusion component of the dynamics evolve on different time scales \cite{ACR21}. 
Tackling this question will require more advanced analytical methods.

A possible relaxation of the condition $ \zeta_{L} = \mO ( L^{-1})$ in
Assumption~\ref{ass}
is discussed in Remark~\ref{rem_generalisation_rates}.
Furthermore, it would be of interest to explore inclusion dynamics without the hard cut-off
$A$ of the slow phase, which should result in a doubly infinite-dimensional limiting process $(X,Y)$.
 A comparable setting has been studied in \cite{jat2024conde}, where the
authors established condensation for the stationary measures of
 the zero-range process with a single fast rate $ u_{1}(A_{L}) \sim
L^{\gamma}$, where $ \gamma \in (0, 2] $ and $ A_{L} \to \infty $. 

\subsection*{Overview of the article}
In Section~\ref{sec_mod_PD}, we prove well--posedness of the modulated PD
diffusion as a Feller process and study its invariant measures.
Section~\ref{sec_main} contains the proof of Theorem~\ref{thm_main}, which is
divided into two substeps, Proposition~\ref{prop_occ} and Lemma~\ref{lem_gen_est}.
The content of Section~\ref{sec_occ} is the proof of instantaneous
condensation
(Proposition~\ref{prop_occ}) using an explicit coupling of random walks.
In Section~\ref{sec_conv_gen},
we bound the error of approximating the generator of the modulated PD diffusion
with the one of the
inclusion process (Lemma~\ref{lem_gen_est}).
This error bound vanishes probabilistically by the aforementioned instantaneous
condensation.

\subsection*{Notation}

Whenever clear from context, we write $ X$ for a process $ (X(t))_{t \geqslant
0}$.
Also, we occasionally use little/big--O notation, for example $
o_{f}(1)$ is a vanishing constant (as $\tdlim $), that depends only on the variable
$f$.

\subsection*{Acknowledgments}

The author thanks Paul Chleboun, Stefan Grosskinsky, and Felix Otto for valuable
discussions and comments. 
This project has received funding from the European Research Council (ERC) under the
European Union’s Horizon 2020 research and innovation programme (Grant
agreement No. 101045082).
Moreover, support by the Deutsche Forschungsgemeinschaft (DFG, German Research
Foundation) under Germany's Excellence Strategy EXC 2044–390685587, Mathematics
Münster: Dynamics–Geometry–Structure, is greatly appreciated.

\section{Modulated Poisson--Dirichlet diffusions}\label{sec_mod_PD}

The purpose of this section is to prove
Proposition~\ref{prop_L_generator}, namely that
the closure of the linear operator 
\begin{equation*}
\begin{aligned}
\mL f(x,y) = \mA_{ \gamma(y), \theta(y)}f (\cdot, y )(x) + \mG f(x, \cdot)(y)\,, \quad f \in \mD
\,,
\end{aligned}
\end{equation*}
generates the modulated Poisson--Dirichlet diffusion
$(X,Y)= (X (t), Y(t))_{ t \geqslant 0}$, a Feller diffusion on $ S $. 
Recall that 
\begin{equation}\label{eq_genA}
\begin{aligned}
\mA_{ \gamma(y), \theta(y)} = 
\sum_{i,j =1}^{ \infty} x_{i} (\gamma (y) \, \delta_{i,j} - x_{j} )
\partial_{x_i}\partial_{x_j} 
- \theta (y) \sum_{i =1}^{\infty} x_{i} \partial_{x_i}
\quad \text{and}
\quad 
\mG = \sum_{k =0}^{A-1} b_{k}(y) \partial_{y_{k}}\,.
\end{aligned}
\end{equation}
Moreover, we will prove some properties of the modulated PD diffusion $(X,Y)$ in Section~\ref{sec_props}.
We stress that Proposition~\ref{prop_L_generator} is not a simple consequence
of a perturbation argument, see e.g.~\cite[Chapter~4.10]{EK_book}, which is
restricted to bounded $\mG $.
Instead, our proof shows that the modulated PD diffusion can be
constructed as the limit of finite dimensional Wright--Fisher models. 

Throughout this section, we assume that the conditions in
Assumption~\ref{ass_def_L} are satisfied.

\subsection{Well--posedness of the dynamics}\label{sec_wp_pd}

In the following we prove well--posedness of the martingale
problem for $\mL$:

\begin{proposition}\label{prop_MP_well_posed}
The martingale problem for $ (\mL, \mD)$ is well--posed for any initial
condition in $ S$, and the solution is supported on $ C([0, \infty), S)$. 
\end{proposition}

We note that Proposition~\ref{prop_L_generator} is equivalent to Proposition~\ref{prop_MP_well_posed}.
The proof of Proposition~\ref{prop_MP_well_posed} is split into existence
(Lemma~\ref{lem_mp_exists}) and uniqueness (Lemma~\ref{lem_hierarchy}) of
solutions for the associated
martingale problem. 

\begin{lemma}[Existence]\label{lem_mp_exists}
For every  $(x,y) \in S $, there exists a solution to the martingale problem for $ (\mL,\mD,\delta_{(x,y)})$, with sample paths in $ C ([0, \infty), S)$.
\end{lemma}

\noindent
The idea is to approximate solutions of martingale problems for $ \mL$ by 
finite-dimensional approximations
with state spaces 
\begin{equation*}
\begin{aligned}
S_{M}:= 
\bigg\{ (x,y) \in S\,:\, 
\sum_{i =1}^{ \infty}x_{i} = \gamma(y)  \,: \, x_{i}= 0 \ \text{for all }  i >M
\bigg\}\,.
\end{aligned}
\end{equation*} 
For this purpose, we introduce multi-loci Wright--Fisher diffusions $(
Z^{(M)} (t),
Y(t))_{t \geqslant 0}$ with state
spaces $K_{M-1}\times S_{Y}$,
\begin{equation*}
\begin{aligned}
K_{M-1}:= \bigg\{x \in [0, 1]^d\,:\, \sum_{i=1}^{M-1} x_{k} \leqslant 1
\bigg\}\,,
\end{aligned}
\end{equation*}
associated to the linear operator $\mathtt{L}^{(M)} =
\mathtt{A}^{(M)}_{ 1,\beta} + \mG$.
Here, for $ g \in C^{2}( K_{M-1}\times S_{Y})$
\begin{equation}\label{eq_def_AM_beta}
\begin{aligned}
\mathtt{A}^{(M)}_{1, \beta} g ( z,y) := 
\sum_{i,j =1}^{M-1} 
z_{i} ( \delta_{i,j} - z_{j} )
\partial_{z_{i} , z_{j} }^{2}g(z,y)
+ 
\overline{\beta}(y)
\sum_{i =1}^{M-1} \big( \tfrac{1 }{M-1} (1- z_{i} ) -  z_{i} \big)
\partial_{ z_{i}}g(z,y)
\,,
\end{aligned}
\end{equation} 
where $ \overline{\beta} \in C ( S_{Y}, [0, \infty))$ is a continuous
extensions of $ \beta$ \eqref{eq_def_beta}, cf. Assumption~\ref{ass_def_L}.
The martingale problem associated to $\mathtt{L}^{(M)} $
admits solutions with continuous sample paths, which is established in
Appendix~\ref{sec_app_WF}. 
In the following, $( Z^{(M)} (t),
Y(t))_{t \geqslant 0}$ will denote such a diffusion process.\footnote{Uniqueness of solutions to the martingale
problem is not required in this step. However, as elaborated in the appendix,
it is expected to hold.
}  
\\

Before we state the proof of Lemma~\ref{lem_mp_exists}, we first present a
simple consequence of rescaling the involved generators, using the map 
\begin{equation}\label{eq_def_Pi}
\begin{aligned}
\proj( z, y) := ( \gamma (y) \widehat{ \overline{z}}, y)\,,
\qquad \text{where } \overline{z} := \Big( z_{1}, \ldots, z_{M-1}, 1-
\sum_{i =1}^{M-1} z_{i}\Big) \,,
\end{aligned}
\end{equation}
and $ \widehat{\cdot}$ denotes the descending reordering. 

\begin{lemma}\label{lem_rescaling}
Let $M \in \mathbf{N}$.
Then for every $ f \in \mD$
\begin{equation}\label{e_gen_approx}
\begin{aligned}
\lim_{M \to \infty} 
\sup_{ \substack{ (z,y) \in K_{M-1}\times S_{Y} }} 
\big|\mathtt{L}^{(M)} \big(f \circ
\proj\big) (z,y) - \mL f \big( \proj (z,y)\big)\big|
=0\,.
\end{aligned}
\end{equation}

\end{lemma}

\begin{proof}
Let $ (z,y) \in K_{M-1}\times S_{Y}$ such that $ \gamma(y)>0$, and write $(x,y) := \proj (z,y)$.
An explicit calculation shows that
\begin{equation}\label{e_rescale_A}
\begin{aligned}
&\mathtt{A}_{1, \beta}^{(M)} 
\big(f \circ \proj ( \cdot, y) \big)
(z) \\
&=
\sum_{i,j =1}^{M} 
x_{i} ( \gamma(y) \delta_{i,j} -  x_{j})
\partial_{x_{i} , x_{j} }^{2}
f( x,y)
+ \overline{\beta} (y) \sum_{i =1}^{M} 
 \Big( \tfrac{1}{M-1} \big( \gamma(y)- x_{i}  \big)- x_{i}\Big)
  \partial_{ x_{i}} 
f( x,y)\\
&= 
\mA^{(M)}_{ \gamma (y),\theta(y)} f (\cdot,y) (x)
+
\sum_{i =1}^{M}
 \Big( \tfrac{ \beta(y)}{M-1} \big( \gamma(y)- x_{i}  \big)- 
\big( \beta(y) - \theta(y)\big)x_{i}\Big)
 \partial_{ x_{i}} 
f( x,y)\,,
\end{aligned}
\end{equation}
where in the first equality we rearranged the summands after noticing that 
\begin{equation}\label{eq_supp1_fd}
\begin{aligned}
\partial_{z_{i}} (
f \circ \proj) (z,y) = \gamma(y) \big( \partial_{x_{i}}f (\proj (z,y)) -
\partial_{x_{M}} f (\proj (z,y)) \big)\,, \quad i =1, \ldots, M-1\,,
\end{aligned}
\end{equation}
because $ f\circ \proj$ is of the form
\begin{equation*}
\begin{aligned}
(f \circ \proj)(z,y)
=
f\big( \gamma(y) \big( z_{1}, \ldots,  z_{M-1}, 1-
(z_{1}+ \cdots + z_{M-1})\big),y\big)\,.
\end{aligned}
\end{equation*}
We also used the fact that $ f \in \mD $ is invariant under reordering of the $
x$--entries.
Similarly, for the generator of the control process
\begin{equation}\label{e_rescale_G}
\begin{aligned}
\mG \big( f \circ \proj (z, \cdot)\big)(y)
&= \mG \big( f ( \gamma(\cdot)  \overline{z}, \cdot)
\big) (y)\\
&= 
\sum_{k =0}^{A-1} b_{k}(y)\bigg( \partial_{y_{k}}f (x,y)
+
 \sum_{i =1}^{M} \partial_{x_{i}}f (x,y)
 x_{i} \frac{\partial_{y_{k}} \gamma (y)}{ \gamma(y)} \bigg)\\
& = \mG f ( x,\cdot )(y) + 
\big( \beta(y) - \theta(y)\big)
\sum_{i =1}^{M} x_{i} \partial_{x_{i}} f (x,y)
\,.
\end{aligned}
\end{equation}
On the other hand, if $ \gamma(y)=0$ then $ \proj (z,y) =
(0,y)$ such that 
\begin{equation*}
\begin{aligned}
\mathtt{L}^{(M)} (f \circ \proj )(z,y) =0 = \mL f (0,y)\,,
\end{aligned}
\end{equation*}
where we used Assumption~\ref{ass_def_L}.(a) to conclude $ \mG f (0,\cdot)(y)=0$.

 Using \eqref{e_rescale_A} and \eqref{e_rescale_G}, we establish that for every $ f \in \mD $
and $(z^{(M)},y) \in K_{M-1}\times S_{Y} $, uniformly
\begin{equation}
\begin{aligned}
\big|\mathtt{L}^{(M)} \big(f \circ
\proj\big) (z^{(M)},y) - \mL f \big( \proj (z^{(M)},y)\big)\big|
\leqslant 
\frac{
\| \beta \|_{\infty} 
}{M-1}  
\sup_{(x^,y) \in S } \sum_{i = 0}^{\infty}|\partial_{x_{i}} f (x,y)
|\,.
\end{aligned}
\end{equation}
Note that the right--hand side of \eqref{e_gen_approx} vanishes as $M \to \infty $, since
\begin{equation*}
\begin{aligned}
\sup_{(x,y) \in S} 
\sum_{i = 0}^{\infty}|\partial_{x_{i}} f (x,y) | < \infty \,, \quad \text{ for
every } f
\in \mD \,,
\end{aligned}
\end{equation*}
and $ \| \beta\|_{\infty}< \infty$ due to Assumption~\ref{ass_def_L}.(d).
Lastly, taking the limit $M \to \infty$ concludes the proof.
\end{proof}

We are now ready to prove that every limit point of $ \proj ( Z^{(M)}, Y)$ solves the martingale problem for
$(\mL,\mD)$.
The proof is elementary and follows closely the proof of Lemma~5.1 of Chapter~4 in 
\cite{EK_book}, up to the above embedding. We include the key steps for
completeness.

\begin{proof}[Proof of Lemma~\ref{lem_mp_exists}]
Let  $(x,y) \in S $, then there exists a sequence $ (z^{(M)}, y) \in
K_{M-1} \times S_{Y}$ such that $ \proj (z^{(M)}, y) \to (x,y) $.
Denote by 
\begin{equation*}
\begin{aligned}
(X^{(M)} (t) , Y(t) )_{t \geqslant 0}
= \proj \big(
Z^{(M)} (t), Y (t)
\big)_{t \geqslant 0}
\end{aligned}
\end{equation*}
a projected solution of the martingale problem for $ \mathtt{L}^{(M)}$ with initial condition
$(z^{(M)}, y)$.
The remainder of the proof is to verify that $(X^{(M)} , Y)_{M \in \mathbf{N}}$
is tight in $ D([0, \infty),S)$, therefore in $ C([0, \infty),S)$, and that every accumulation point solves
the martingale problem associated to $ (\mL,
\delta_{(x,y)})$.\\

\noindent\textit{Tightness:}
Because $(Z^{(M)} , Y)$ solves the martingale problem for $( \mathtt{L}^{(M)},
\delta_{(z^{(M)}, y)})$, $M \in \mathbf{N}$, tightness follows from \cite[Theorem 3.9.4]{EK_book}
and \eqref{e_gen_approx}.
See also Remark~5.2 of Chapter~4 in the same reference.
Consequently, $ ( X^{(M)}, Y)$ is tight because $ \proj $ is a continuous map.
\\

\noindent\textit{Martingale property of accumulation points:}
Let $(X,Y)$ be an accumulation point, and $\proj(Z^{(M_{n})} , Y)$ a subsequence
converging (weakly) to it.
We know that 
\begin{equation*}
\begin{aligned}
t \mapsto f \circ \proj \big(Z^{(M_{n})}(t), Y (t)\big) 
- \int_{0}^{t} \mathtt{L}^{(M_{n})} \big(f \circ \proj \big)
\big(Z^{(M_{n})}(s), Y(s)\big)  \ud s
\end{aligned}
\end{equation*}
is a martingale, and that $ \mathtt{L}^{(M_{n})} (f\circ \proj) \big(Z^{(M_{n})}(s), Y(s)\big) $ in the integral above can
be replaced by $\mL f \big(X^{(M_{n})}(s), Y(s)\big)$ when $M_{n}\to \infty$,
using Lemma~\ref{lem_rescaling}.
This puts all the $M_{n}$-dependency on $
X^{(M_{n})}$ only.
 Since $(X,Y)$ is supported on  $ C([0, \infty),S)$,
both $f (\bullet (t))$ and $ \int_{0}^{t}\mL f(\bullet
(s)) \ud s $ are continuous.
Therefore, the
above expression converges weakly to 
\begin{equation}\label{eq_limit_mart}
\begin{aligned}
t \mapsto f (X(t), Y (t)) 
- \int_{0}^{t} \mL f (X(s), Y(s))  \ud s\,,
\end{aligned}
\end{equation}
for every such $t$, 
which is also a martingale.
Because this follows verbatim from Lemma~5.1 Chapter 4 \cite{EK_book},
we refrain from giving superfluous details here.
Overall, we conclude that $ (X,Y)$ is a solution to the martingale problem for  $ (\mL,
\mD,\delta_{(x,y)})$.
\end{proof}

\begin{remark}\label{rem_notindomain}
The attentive reader may have spotted that we applied $ \mathtt{L}^{(M)}$ to $
f\circ \proj $, in Lemmas~\ref{lem_mp_exists} and \ref{lem_rescaling}.
However, it is possible that  $ f\circ \proj \notin C^{2} ( K_{M-1}\times
S_{Y})$, because we  assumed $ \gamma $ to be differentiable only in $
\overline{\supp \gamma }$.

We ommited this complication, because $ \mathtt{L}^{(M)}$ acts locally and vanishes on
$ S_{Y}^{*}$ \eqref{eq_SyStar}, leading to absorbing boundary conditions. 
Instead, we could have introduced stopping times 
\begin{equation*}
\begin{aligned}
T := \inf \{t \in [0 , \infty)\,: \, \gamma(Y(t)) =0 \}\,,
\end{aligned}
\end{equation*}
until which the martingale problem for $ f \circ \proj $ is well--defined. 
Because $\proj ( \cdot, Y (t)) =(0, Y (T))$ for all $ t \geqslant T$, it is
then possible to
extend the solution from $[0,T]$ to the full time horizon.
\end{remark}

\begin{lemma}[Uniqueness]\label{lem_hierarchy}
Let $ (x,y) \in S$.
There exists a unique family of probability measures $( \mu_{t})_{t \geqslant 0}
\subset \mM_{1}( \overline{\nabla})$ such that every
solution $ (X,Y) \in D([0, \infty),
S) $ to the martingale problem
$(\mL ,\mD, \delta_{(x,y)})$ satisfies
\begin{equation*}
\begin{aligned}
\mathrm{Law}(X (t)) = \mu_{t}\,, \quad \forall t \geqslant 0 \,.
\end{aligned}
\end{equation*}
In particular, the law of $ (X(t),Y(t))$ is uniquely determined.
\end{lemma}

To prove the lemma, we will exploit the explicit form of the generator
$\mA_{\gamma, \theta}$, leading to a hierarchy of differential equations,
that represents the evolution of moments. To this end, let us introduce for all $ n \in \mathbf{N}$
the subspaces
\begin{equation*}
\begin{aligned}
\mD_{X , n}:= \mathrm{span}\Big\{ \varphi_{\bf m} = \varphi_{m_{1}}\cdots
\varphi_{m_{k}} \,,\ m_{1} + \cdots +m_{k} \leqslant n \Big\}
\,,
\end{aligned}
\end{equation*}
of ``polynomials''  on $ \overline{\nabla}$ of at most degree $n$.

\begin{proof}
Let $ (x,y) \in S$ and $ (X(t) , Y (t))_{t
\geqslant 0}$ be a solution of the martingale problem for $(\mL ,
\mD,\delta_{(x,y)})$, possibly with sample paths in $
D([0 , \infty), S)$.
 Then, for every $ f \in \mD_{X,n}$, $n \in \mathbf{N}$, we have
\begin{equation*}
\begin{aligned}
 \mathbf{E} [ f (X(t)) ]
= f(x)+
\int_{0}^{t} \mathbf{E} [\mA_{ \gamma(Y (t)), \theta( Y (t))} f(X (t))] \ud s\,.
\end{aligned}
\end{equation*}
Note that on each ``monomial'' $ \varphi_{ \bf m}$, the operator $\mA$ acts as 
\begin{equation}\label{e_A_to_poly}
\begin{aligned}
\mA_{ \gamma( y) , \theta(y)} \varphi_{\bf m} 
&= 
 \gamma(y)\left\{\sum_{\ell =1}^{k}
m_{\ell} ( m_{\ell}-1 ) \varphi_{m_{\ell}-1} \prod_{\ell ' \neq \ell
}\varphi_{m_{\ell '}}
+ \sum_{\ell \neq \ell ' } m_{\ell}m_{ \ell '} \varphi_{m_{\ell}+ m_{\ell '}-1} 
\prod_{ \tilde{\ell} \neq \ell , \ell '} \varphi_{m_{ \tilde{\ell }}}\right\}\\
&\qquad -
\left\{ \sum_{\ell =1}^{k} [m_{\ell} (m_{\ell}-1) + m_{\ell} \theta(y)] +
\sum_{\ell \neq \ell '} m_{\ell} m_{ \ell '} \right\} \varphi_{\bf m}\\
&=:
 \gamma(y) g^{(\bf m)} - a ( \theta(y),  {\bf m})  \varphi_{\bf m}\,,
\end{aligned}
\end{equation}
where $ g^{(\bf m)} \in \mD_{X,n-1}$ denotes the function in the parenthesis on the
right--hand side in the first line. Since $f \in \mD_{X, n}$ is a linear combination of $\varphi_{\bf m}$'s,
there exists a function $ g= g^{(f)} \in \mD_{X,n-1}$ such that 
\begin{equation}\label{eq_supp_momODE}
\begin{aligned}
 \mathbf{E} [ f (X(t)) ]
=f(x)+
\int_{0}^{t}
\gamma(s) \mathbf{E} [g ( X(s))] - a (s)  \mathbf{E} [f( X(s))]
 \ud s\,,
\end{aligned}
\end{equation}
with  $ \gamma(s) := \gamma(Y (s))$ and $ a (s):= a(\theta(Y(s)),f) \in
\mathbf{R}$ the
corresponding linear
combination of $a ( \theta(y),  {\bf m})$ from \eqref{e_A_to_poly}.
For the sake of clarity we replaced in the previous step the $Y$-dependency by
a time-dependency
only, since $(Y(s))_{s \geqslant 0}$ is deterministic and the unique process associated to $\mG$.
Consequently, the functions $ s \mapsto a
(s), \gamma(s)$ are also deterministic and continuous.

We proceed by induction. 
First note that 
for $ f \in \mD_{X,0}$, $t \mapsto \mathbf{E} [f (X(t))] = f =
\mu_{t}(f)$ is constant.
Next, we assume that for every $ g \in \mD_{X,n'}$, $ n'<n $, 
$$ t \mapsto \mathbf{E} [g (X (t))]=: \mu_{t}(g)  $$
 is uniquely determined (independent of the specific choice of solution
$(X,Y)$) and continuous.
%
%
Therefore, for every $ f \in \mD_{X , n}$ the differential equation 
\begin{equation*}
\begin{aligned}
\frac{\ud}{\ud t} \mu_{t}(f) = \gamma(t)   \mu_{t}(g) - a (t) \mu_{t}(f) \,,
\quad \mu_{0}(f) = f(x)\,,
\end{aligned}
\end{equation*}
has a unique, continuous solution $ ( \mu_{t}(f))_{t \geqslant 0}$. For example, by
applying the method of variation of constants. 

We conclude that for every $ t \geqslant 0 $ and every $ f \in \mD_{X}$
\begin{equation}\label{e_moments_supp}
\begin{aligned}
  \mathbf{E} [f (X(t))]=\mu_{t}(f)\,,
\end{aligned}
\end{equation}
as a consequence of \eqref{eq_supp_momODE}.
Because $ \mD_{X}$ seperates points, i.e.~for every $ x \neq x'$ there exists a
$f \in \mD_{X}$ such that $ f(x) \neq f (x')$, \eqref{e_moments_supp} uniquely
determines a probability measure $ \mu_{t} \in \mM_{1}( \overline{\nabla}) $
\cite[Theorem~3.4.5]{EK_book}.

Overall, the law of $ (X(t),Y(t))$ is uniquely determined by $
\mu_{t} $, since $ Y(t)$ is
a deterministic process.
This concludes the proof, because
 uniqueness of
one-dimensional marginals characterises the law of $(X,Y)$ \cite[Theorem~4.4.2(a)]{EK_book}. 
\end{proof}
\subsection{Properties of the modulated PD diffusion}\label{sec_props}

In the remainder of this section, we summarise some key properties of the
modulated Poisson–Dirichlet diffusion $(X,Y)=(X (t),Y(t))_{t \geqslant 0} $.
In particular, we provide a rescaling argument that links $(X,Y)$ to the
classical PD diffusion $ \overline{X}$ of Ethier and Kurtz \cite{EK81}. 
Moreover, we characterise the diffusion's extremal invariant measures and show that it
concentrates for all positive times on the ``boundary'' of $S$.


\begin{lemma}\label{lem_prop}
Let $ (X,Y)$ be the modulated Poisson--Dirichlet diffusion generated by $\mL=\mA_{\gamma, \theta}
+\mG $,
as in Proposition~\ref{prop_L_generator}. 
\begin{enumerate}
\item[(i)] If $(X(0),Y(0))=(x,y) \in S$, then for every $ t \geqslant  0 $
\begin{equation*}
\begin{aligned}
( X(t), Y (t))
\stackrel{d}{=}
\big( \gamma ( Y (t)) \overline{X} (t), Y (t)\big)\,, 
\end{aligned}
\end{equation*}
where $ (  \overline{X}, Y )$ denotes the process on $
\overline{\nabla} \times S_{Y}$, generated by 
$\mA_{1, \beta}+\mG $, for any initial condition $ ( \overline{x}, y)$ satisfying
$ \gamma(y) \overline{x} =x $. Recall that $ \beta$ was defined in
\eqref{eq_def_beta}.

In particular, the marginal process $ \overline{X} $ agrees with the classical
Poisson--Dirichlet diffusion \eqref{e_gen_clPD} whenever $ \beta(\cdot) $ is constant.

\item[(ii)] 
For every stationary state $ y^{*} \in S_{Y}$ of $\mG$, the probability measure
\begin{equation}\label{eq_inv_meas}
\begin{aligned}
\mathrm{PD}_{[0, \gamma^{*}]} ( \theta^{*}  ) \otimes \delta_{ y^{*}} \in
\mM_{1}( S) \,,
\end{aligned}
\end{equation}
with  $\theta^{*} := \theta (y^{*}) \in [0, \infty)$ and $ \gamma^{*}:=
\gamma(y^{*}) \in [0,1]$, is invariant for $ \mL$.
Moreover, the measures \eqref{eq_inv_meas} are reversible and ergodic, in the sense that
if $ \lim_{t \to \infty} Y (t) = y^{*}$ then 
\begin{equation}\label{eq_ergodic}
\begin{aligned}
(X (t), Y (t)) \weakconv \mathrm{PD}_{[0, \gamma^{*}]} ( \theta^{*}  ) \otimes
\delta_{ y^{*}}\,,
\quad \text{as } t \to \infty\,.
\end{aligned}
\end{equation}
\end{enumerate}
\end{lemma}

\begin{proof}
(i). To conclude the first statement, it suffices to check that $\big( \gamma ( Y (t)) \overline{X} (t), Y
(t)\big)_{t \geqslant 0}$ solves the martingale problem for $ (\mL, \mD)$.
Similar as in the proof of Lemma~\ref{lem_rescaling}, we can show by
explicit calculations that
for all $ f \in \mD $ and $ ( \overline{x}, y) \in \overline{\nabla}\times S_{Y}$ 
\begin{equation*}
\begin{aligned}
\mA_{1, \beta(y)} f( \gamma(y) \cdot, y)( \overline{x})
&=
\gamma(y)^{2} 
\sum_{i ,j =1}^{ \infty} \overline{x}_{i} ( \delta_{i,j} - 
\overline{x}_{j}) \partial_{i,j}^{2}f ( \gamma (y)
\overline{x}, y)
-
\beta (y) \gamma(y) \sum_{i=1}^{ \infty} \overline{x}_{i} \partial_{i} f(
\gamma(y) \overline{x}, y)
\\
&=
\mA_{ \gamma (y), \theta (y)}f (\cdot,y)( \gamma (y) \overline{x})
- \big( \beta(y) - \theta (y)\big)
\sum_{i =1}^{ \infty} \gamma(y) \overline{x}_{i} \partial_{i} f ( \gamma(y) \overline{x}, y)\,.
\end{aligned}
\end{equation*}
Similarly, 
\begin{equation*}
\begin{aligned}
\mG f ( \gamma (\cdot) \overline{x}, \cdot) (y)
&= 
\sum_{k =0}^{A-1} b_{k}(y)\bigg( \partial_{y_{k}}f ( \gamma(y) \overline{x},y)
+
 \sum_{i =1}^{M} \partial_{x_{i}}f ( \gamma(y) \overline{x},y) \overline{x}_{i} \partial_{y_{k}} \gamma (y) \bigg)\\
&=
 \mG f ( \gamma(y) \overline{x},\cdot )(y) + 
\big( \beta(y) - \theta(y)\big)
\sum_{i =1}^{M} \gamma (y) \overline{x}_{i} \partial_{x_{i}} f ( \gamma(y)
\overline{x} ,y)\,,
\end{aligned}
\end{equation*}
such that overall
\begin{equation}\label{e_rescal_PDdiff}
\begin{aligned}
\Big(
\mA_{1, \beta} 
+
\mG\Big) f ( \gamma( \cdot) \cdot, \cdot)( \overline{x},y)
= 
\Big(
\mA_{ \gamma, \theta} 
+
\mG\Big) f ( \gamma(y) \overline{x},y) = \mL f (\gamma(y) \overline{x}, y)\,.
\end{aligned}
\end{equation}
Consequently, because $ ( \overline{X}, Y)$ solves the martingale problem for $
(\mA_{1 , \beta} + \mG , \mD)$,
\begin{equation*}
\begin{aligned}
t\mapsto
&\ g ( \overline{X}_{t}, Y_{t}) - \int_{0}^{t} (\mA_{1, \beta (Y_{s})}+ \mG )g
( \overline{X}_{s}, Y_{s} ) \ud s\\
 &=f( \gamma (Y_{t}) \overline{X}_{t}, Y_{t}) 
- \int_{0}^{t} 
\mA_{\gamma(Y_{s}), \theta(Y_{s})} f ( \cdot, Y_{s}) ( \gamma(Y_{s})
\overline{X}_{s}) 
+ \mG f ( \gamma (Y_{s}) \overline{X}_{s}, \cdot ) (Y_{s})
\ud s\,,
\end{aligned}
\end{equation*}
is a martingale, where $ g ( \overline{x}, y) := f ( \gamma(y) \overline{x},
y)$.
 \\

(ii). First, we prove that for every stationary state $y^{*}$ of $\mG$ the
measure \eqref{eq_inv_meas} is invariant (and reversible)
with respect to $\mL$.
For this, we leverage on the first part of the lemma and use the fact that
$\mathrm{PD}( \beta^{*}) \in \mM_{1}( \overline{\nabla})$, $ \beta^{*}:= \beta(
y^{*}) $, is the unique invariant (and reversible)
distribution for $ \mA_{1, \beta^{*}}$
\cite{EK81}.
Therefore, for every stationary state $ y^{*} \in S_{Y}$ of $\mG$, the measure
\eqref{eq_inv_meas} is invariant (and reversible) for $\mL $, because
$\mL = \mA_{ \gamma^{*},
\theta^{*}}+ \mG $ on the support of \eqref{eq_inv_meas}.
\\

To show ergodicity, we assume that $y^{*}= \lim_{ t\to \infty} Y(t)$ exists.
Then \eqref{eq_ergodic} is equivalent to 
\begin{equation*}
\begin{aligned}
\overline{X} (t) \weakconv \mathrm{PD} ( \theta^{*}) \,,
\end{aligned}
\end{equation*}
where $ ( \overline{X},Y)$ is generated by $ \mA_{1, \beta} + \mG $, see (i).
Because $ \mM_{1}( \overline{\nabla}) $ is compact, the sequence of laws $
(\overline{X}(t))_{t \geqslant 0}$ is tight and admits subsequential limits.
If $ \nu$ is such an accumulation point, then $\nu = \mathrm{PD}(
\beta^{*})$, since $ \nu$ must be the unique invariant distribution for $
\mA_{1 , \beta^{*}}$. 

It remains to note that
\begin{equation*}
\begin{aligned}
\beta^{*}= 
\beta( y^{*}) 
= 
\theta( y^{*}) 
+
\sum_{k=0}^{A-1} b_{k}( y^{*})
\frac{\partial_{y_{k}} \gamma ( y^{*})}{ \gamma ( y^{*})} 
= \theta^{*}\,,
\end{aligned}
\end{equation*}
where we used that $ b_{k}( y^{*}) =0$, because $\mG f (y^{*}) =0$ for $f (y) =
y_{k}$, for all
$k$.
\end{proof}

\begin{corollary}\label{cor_concentration}
Let $ (X,Y)$ be the modulated Poisson--Dirichlet diffusion generated by $\mL=\mA_{\gamma, \theta}
+\mG $,
as in Proposition~\ref{prop_L_generator}. 
Then for any initial condition $(X (0), Y(0)) \in S $, almost surely 
\begin{equation*}
\begin{aligned}
\sum_{i=1 }^{\infty}X_{i}(t) = \gamma(Y (t))\,, \quad \text{for all } t >0\,.
\end{aligned}
\end{equation*}
\end{corollary}

The statement follows almost verbatim from the proof of \cite[Theorem~2.6]{EK81}, where
the case $ \gamma =1 $ and $ \theta (\cdot) = \beta(\cdot)= \Theta$ is treated.
The only explicit occurrence of the
variable $ \beta (\cdot)$ is in the estimate \eqref{eq_theta_est_boundary}
below, which remains unaffected.
We sketch the argument for completeness. 

\begin{proof}
Due to Lemma~\ref{lem_prop} it suffices to prove that the process $ (
\overline{X},Y)$, generated by $ \mA_{1, \beta}+ \mG $ satisfies almost surely
\begin{equation*}
\begin{aligned}
\sum_{i =1}^{\infty } \overline{X}_{i }(t) = 1\,,  \quad \text{for all } t >0\,.
\end{aligned}
\end{equation*}

The idea of the proof is to extend the domain of $\mL$, more precisely, the domain of
$\mA_{1, \beta ( \cdot)} $
from $\mD_{X}$ to include functions $ \varphi_{m} (x ) = \sum_{i
=1}^{ \infty} x_{i}^{m}$, $m \geqslant 1$. For these new test functions, $\mA_{1, \theta (\cdot)}$
is naturally extended to
\begin{equation*}
\begin{aligned}
\mA_{1, \beta (y)} \varphi_{m} (x) 
= 
m(m-1) \varphi_{m-1}(x)
- m ( (m-1)+ \beta(y))  \varphi_{m}(x)\,,
\end{aligned}
\end{equation*}
noting that $ \varphi_{1}:= 1$.
In particular, we may define the processes
\begin{equation}\label{eq_conc_supp3}
\begin{aligned}
Z_{m}(t) := \varphi_{m}( \overline{X} (t)) - \varphi_{m}( \overline{X} (0)) - \int_{0}^{t} \mA_{1,
\beta (Y(s))} \varphi_{m}( \overline{X}(s))  \ud s\,,
\end{aligned}
\end{equation}
which are continuous square-integrable martingales. For $ m \geqslant 2$, this
is simply a consequence of Dynkin's formula, whereas the case $m \in (1,2)$
requires additional work since $ u \mapsto u^{m}$ is not an element of $
C^{2}([0,1])$.

First, we can use the martingale formulation \eqref{eq_conc_supp3} to show 
\begin{equation}\label{eq_conc_supp1}
\begin{aligned}
\lim_{m \searrow 2} 
\frac{1}{2} 
\mathbf{E} \big[Z_{m}(t) - Z_{2}(t) \big]
=
\mathbf{E}\bigg[\int_{0}^{t}
\Big(
1- \sum_{i =1}^{ \infty} \overline{X}_{i}(s)
\Big) \ud s\bigg]
=0\,,
\end{aligned}
\end{equation}
which is a consequence of the bounded and pointwise convergence of
\begin{equation*}
\begin{aligned}
\mA_{1,
\beta (y)} \big(\varphi_{2}(x)
- \varphi_{m}(x)\big) \to 2\Big(
1- \sum_{i =1}^{ \infty} x_{i}
\Big)\,, \quad \text{on }  \overline{\nabla} \times S_{Y}\,.
\end{aligned}
\end{equation*}
In particular, \eqref{eq_conc_supp1} implies that almost surely 
\begin{equation}\label{eq_conc_supp2}
\begin{aligned}
\overline{X} (t) \in \nabla \quad \text{for almost every } t >0\,.
\end{aligned}
\end{equation}

To lift \eqref{eq_conc_supp2} to all positive times, we perform a second
moment calculation for $ Z_{m}(t)$ which, in combination with \eqref{eq_conc_supp2},
yields 
\begin{equation*}
\begin{aligned}
\lim_{m \searrow 1} \mathbf{E} \big[ |Z_{m} (t) |^{2}\big]
= 
\lim_{m \searrow 1} 
m^{2} \int_{0}^{t} \mathbf{E} \big[ \varphi_{2m-1} ( \overline{X} (s)) -
\varphi^{2}_{m}( \overline{X} (s)) \big] \ud s=0\,.
\end{aligned}
\end{equation*}
In turn, this implies that $\lim_{m \searrow 1} \sup_{0 \leqslant t \leqslant
T} | Z_{m} (t) | =0$ in $ \mathbf{P}$-probability, by Doob's martingale
inequality, for any $T >0$. 
Thus, there exists a  subsequence $(m_{n})_{n}$ along which the convergence
holds almost surely.

We rewrite \eqref{eq_conc_supp3} in the limit $m_{n} \to 1$, which yields  almost
surely 
\begin{equation}\label{eq_conc_supp4}
\begin{aligned}
\sum_{i =1}^{\infty} \overline{X}_{ i} (t) = \sum_{i =1}^{\infty} \overline{X}_{i}(0) 
+ \chi(t)
-\int_{ 0}^{t} \beta(Y(s))  \ud s\,, \qquad \forall \ 0 \leqslant t \leqslant
T\,,
\end{aligned}
\end{equation}
where $ \chi$ is of the explicit form 
\begin{equation*}
\begin{aligned}
\chi (t) := 
\limsup_{n \to \infty}  
\int_{0}^{t} (m_{n}-1)  \varphi_{m_{n}-1}( \overline{X} (s)) \ud s
\,.
\end{aligned}
\end{equation*}
Finally, let $ \delta>0$ be arbitrary. 
Then almost surely for every $ t \in (0, T]$
there exists a $ t' \in  (t- \delta, t ) $ such
that $ \overline{X} (t') \in \nabla$, due to \eqref{eq_conc_supp2}. In particular, using
the representation \eqref{eq_conc_supp4}, we have
\begin{equation}\label{eq_theta_est_boundary}
\begin{aligned}
\sum_{i =1}^{\infty} \overline{X}_{ i} (t)
- 
1
= \chi(t) - \chi(t') - \int_{t'}^{t} \beta(Y (s)) \ud s
\geqslant  - \delta \sup_{y \in S_{Y}} \beta(y) \,,
\end{aligned}
\end{equation}
where we used that $ \chi$ is increasing in time.
Because $ \delta >0$ was arbitrary, we conclude the statement of the corollary. 
\end{proof}

\section{Outline and proof of the main result}\label{sec_main}

In this section, we  establish the
convergence of the inclusion process in Theorem~\ref{thm_main}.
To motivate the subsequent steps, we first explain the dilemma one encounters
when trying to prove this convergence naively.

\subsection{A discontinuity caused by instantaneous condensation}

Naturally, the analysis of the inclusion process in the thermodynamic limit 
requires observation of both the slow and the fast phase separately. 
The first reasonable guess in doing so is the embedding
\begin{equation}\label{e_embd_naiiv}
\begin{aligned}
\overline{\embdcomb}( \eta) 
= \big( \tfrac{1}{N} \hat{\eta}^{+}, \tfrac{1}{L} ( \#_{0} \eta, \ldots,
\#_{A} \eta) \big)\,.
\end{aligned}
\end{equation}
Recall that $ \eta^{+}_{i} = \max \{ \eta_{i}-A, 0\}$ and $ \#_{k} \eta =
\sum_{i =1}^{L} \mathds{1}_{ \eta_{i}=k}$.
The difference to $\embdcomb$ in
\eqref{e_def_embd_direct} is that \eqref{e_embd_naiiv} also includes the occupations
$\#_{A} \eta $.

Next, let us assume that
\begin{equation}\label{eq_wrong_conv}
\begin{aligned}
\Big(\overline{\embdcomb}\big( \eta^{(L,N )} (t)\big)\Big)_{t \geqslant 0} 
\weakconv (X (t) , \overline{Y}(t))_{t \geqslant 0} \,,
\quad \text{with }\quad \overline{Y}_A (t) = \lim_{\tdlim} \frac{\#_{A}
\eta^{(L,N)}(t)}{L}  \,.
\end{aligned}
\end{equation}
In the following, we are going to explain that $ \overline{Y}_{A}(t)$ has a discontinuity 
at $ t=0$. 
This violates that paths of the limit process are right--continuous, 
and deems a proof of the convergence \eqref{eq_wrong_conv} by means of Markov
processes unfeasible.

As an example, let us  consider the classical ($A=0$) inclusion process $ ( \eta (t))_{t
\geqslant 0}$ with generator \eqref{eq_classic_IP}, and initial configuration $ \eta_{i}(0) =1$ for all $i =1,
\ldots, L $. 
It was proven in \cite{CGG24} that almost all particles immediately cluster
together such that $ \eta_{i}(t) =0$ except for a vanishing fraction of sites,
when taking the thermodynamic limit. In particular, when $ \tdlim $
\begin{equation*}
\begin{aligned}
\overline{Y}_{0}(0)=0\,,\quad \text{but} \quad \overline{Y}_{0}(t)=1 \quad
\text{ for all }\  t >0\,.
\end{aligned}
\end{equation*}
When considering more general transition rates, as outlined in  Assumption~\ref{ass}, 
this discontinuity persists, making it unsuitable to observe $ \overline{Y}_{A}$. Consequently, the observable
$\#_{A} \eta $ in \eqref{e_embd_naiiv}  must be disregarded, leading to the
embedding $\embdcomb$ \eqref{e_def_embd_direct}.

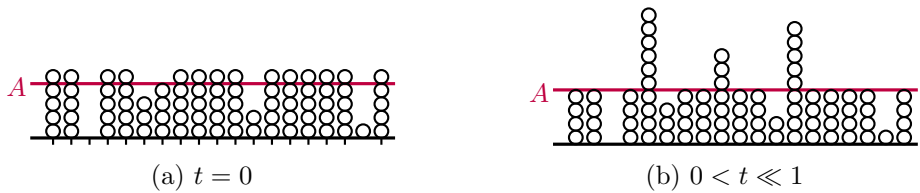
\begin{figure}[H]
\centering
\captionsetup{width=.8\linewidth}
\captionsetup{width=.8\linewidth}
\begin{subfigure}{0.45\textwidth}
\centering
    {
	\begin{tikzpicture}[scale = 0.6]
\draw (2.2,1.55) node[anchor=north]  {\textcolor{purple}{{\small $A$}}};
\draw[very thick, purple] (2.5,1.2) -- (10.5,1.2);

\draw[ thick] (9.8,.15) circle (0.15);

\foreach \x in {3, 3.4, 4.2, 4.6, 5, 5.4, 5.8, 6.2, 6.6, 7, 7.4, 7.8, 8.2,
8.6, 9, 9.4, 10.2} {
    \foreach \y in {0.15, 0.45} {
        \draw[ thick] (\x,\y) circle (0.15);
    }
}

\foreach \x in {3, 3.4, 4.2, 4.6, 5, 5.4, 5.8, 6.2, 6.6, 7, 7.8, 8.2,
8.6, 9, 9.4, 10.2} {
    \foreach \y in { .75} {
        \draw[ thick] (\x,\y) circle (0.15);
    }
}

\foreach \x in {3, 3.4, 4.2, 4.6, 5.4, 5.8, 6.2, 6.6, 7, 7.8, 8.2,
8.6, 9, 9.4, 10.2} {
    \foreach \y in {1.05} {
        \draw[ thick] (\x,\y) circle (0.15);
    }
}

\foreach \x in {3, 3.4, 4.2, 4.6, 5.8, 6.2, 6.6, 7, 7.8, 8.2,
8.6, 9, 9.4, 10.2} {
    \foreach \y in {1.35} {
        \draw[ thick] (\x,\y) circle (0.15);
    }
}

\foreach \x in {3, 3.4, 3.8, 4.2, 4.6, 5, 5.4, 5.8, 6.2, 6.6, 7, 7.4, 7.8, 8.2,
8.6, 9, 9.4, 10.2} {
	\draw[thick] (\x,0) -- (\x,-0.15) {};	
}

\draw[very thick] (2.5,0) -- (10.5,0);

\end{tikzpicture}
    }
  \caption{$ t = 0$}
\end{subfigure}
\begin{subfigure}{0.45\textwidth}
\centering
    {
	\begin{tikzpicture}[scale = 0.6]
\draw (2.2,1.55) node[anchor=north]  {\textcolor{purple}{{\small $A$}}};
\draw[very thick, purple] (2.5,1.2) -- (10.5,1.2);

\draw[ thick] (9.8,.15) circle (0.15);

\foreach \x in {3, 3.4, 4.2, 4.6, 5, 5.4, 5.8, 6.2, 6.6, 7, 7.4, 7.8, 8.2,
8.6, 9, 9.4, 10.2} {
    \foreach \y in {0.15, 0.45} {
        \draw[ thick] (\x,\y) circle (0.15);
    }
}

\foreach \x in {3, 3.4, 4.2, 4.6, 5, 5.4, 5.8, 6.2, 6.6, 7, 7.8, 8.2,
8.6, 9, 9.4, 10.2} {
    \foreach \y in { .75} {
        \draw[ thick] (\x,\y) circle (0.15);
    }
}

\foreach \x in {3, 3.4, 4.2, 4.6, 5.4, 5.8, 6.2, 6.6, 7, 7.8, 8.2,
8.6, 9, 9.4, 10.2} {
    \foreach \y in {1.05} {
        \draw[ thick] (\x,\y) circle (0.15);
    }
}

\foreach \y in {1.35, 1.65,1.95,2.25,2.55,2.85} {
     \draw[ thick] (4.6,\y) circle (0.15);
}

\foreach \y in {1.35, 1.65,1.95} {
     \draw[ thick] (6.2,\y) circle (0.15);
}

\foreach \y in {1.35, 1.65,1.95,2.25,2.55} {
     \draw[ thick] (7.8,\y) circle (0.15);
}

\draw[very thick] (2.5,0) -- (10.5,0);

\end{tikzpicture}
    }
  \caption{$0< t \ll 1$}
\end{subfigure} 
\caption{
On the left we see a possible initial particle configuration, which leads to
instantaneous clustering after an infinitesimal small time on the right.
}
\label{fig_fast_clustering}
\end{figure}

This discontinuous behaviour makes the convergence of the inclusion process delicate.
Classically, convergence of particle systems is based on the
Trotter--Kurtz theorem which allows to deduce convergence of the processes from
the convergence of infinitesimal
generators
\begin{equation}\label{eq_ex_opconv}
\begin{aligned}
\lim_{ \tdlim}\mathfrak{L}_{L, N} (f \circ \embdcomb)  = \mL f (\embdcomb (
\cdot))\,.
\end{aligned}
\end{equation}
When doing so, one will inevitably encounter $\#_{A} \eta $. 
For instance, consider a smooth function $f $ in \eqref{eq_ex_opconv} that only
depends on $ \embdcomb ( \eta) $ in terms of  $ \tfrac{1}{L} \#_{A-1} \eta$.
The infinitesimal change of $ f$ depends on particles jumping from $
\eta_{i} = A $ to another site, since such jumps leave behind $ \eta_{i} =
A-1$ particles which in turn leads to an increase of $ \tfrac{1}{L} \#_{A-1} \eta$ by $+
\tfrac{1}{L}$:
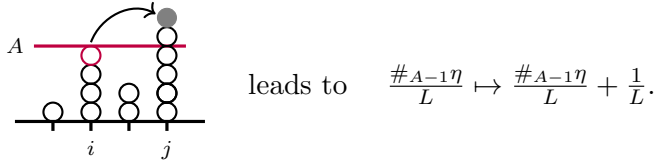
\begin{center}
	\begin{tikzpicture}[scale = 0.5]
\draw (6,-.25) node[anchor=north]  {\scriptsize{$j$}}
(4,-.25) node[anchor=north]  {\scriptsize{$i$}}
(2,2) node  {\scriptsize{$A$}}
(13.5,1) node  {leads to \quad {$\tfrac{\#_{A-1} \eta}{L} \mapsto
\tfrac{\#_{A-1} \eta}{L} + \tfrac{1}{L}$.}}
;
\draw[very thick, purple] (2.5,2) -- (6.5,2);
\draw[thick]
(3,0.25) circle (0.25) {}
(4,0.25) circle (0.25) {}
(5,0.25) circle (0.25) {}
(6,0.25) circle (0.25) {}
(6,0.75) circle (0.25) {}
(6,1.25) circle (0.25) {}
(6,1.75) circle (0.25) {}
(6,2.25) circle (0.25) {}
;
\draw[very thick]
(3,0) -- (3,-0.25) {}
(4,0) -- (4,-0.25) {}
(5,0) -- (5,-0.25) {}
(6,0) -- (6,-0.25) {}
;
\draw[thick] (4,0.75) circle (0.25) {}
(4,1.25) circle (0.25) {}
 (6,.75) circle (0.25) {} ;
\draw[very thick] (2,0) -- (7,0);
\draw[thick, purple] (4,1.75) circle (0.25) {};
\draw[thick] (5,.75) circle (0.25) {} ;
\draw[fill, gray] (6,2.75) circle (0.25) {} ;
\draw[thick,->] (4,2.1) to [out=70,in=150] (5.7, 2.85);
\end{tikzpicture} 
\end{center}
One can check that the corresponding infinitesimal shift in the observable $ f$ is then
approximately given by
\begin{equation}\label{eq_exmpl_op_conv}
\begin{aligned}
&\sum_{i ,j =1}^{L}
\mathds{1}_{\eta_{i}=A}
u_{1}( A) u_{2}( \eta_{j} ) \mathds{1}_{
\eta_{j}> A}
\frac{ 1}{L} f'\big( \tfrac{1}{L} \#_{A-1} \eta  \big)
 \simeq 
\rho \,
q_{A}\  \frac{\#_{A} \eta}{L} 
\bigg\{\sum_{j =1}^{L}
 \frac{ \eta_{j}^{+}}{N}\bigg\}
f'\big( \tfrac{1}{L} \#_{A-1} \eta  \big)\,,
\end{aligned}
\end{equation}
where we restricted ourselves to particles jumping from $ \eta_{i} =A$ to
sites $ j $ with $ \eta_{j}>A$.
Here, $ q_{A}$ as in Assumption~\ref{ass}.

To establish convergence of generators as in \eqref{eq_ex_opconv},  the
right--hand side of \eqref{eq_exmpl_op_conv} had to be entirely
expressed in terms of $\embdcomb ( \eta)$. 
Indeed, the sum on the right--hand side counts the relative number of particles
above the threshold $A$ and can be  expressed in terms of $
\tfrac{\#_{k} \eta}{L}$, $k=0,\ldots, A-1$.
However, this is not possible for 
\begin{equation}\label{eq_supp1_dilemma}
\begin{aligned}
\frac{\#_{A} \eta(t)}{L} 
=
1- \frac{1}{L}
\sum_{k =0}^{A-1} \#_{k} \eta(t)
- \frac{1}{L} \sum_{i =1}^{L} \mathds{1}_{ \eta_{i}(t)>A}\,.
\end{aligned}
\end{equation}
While the first sum on the right--hand side of \eqref{eq_supp1_dilemma} converges to $ \sum_{k
=0}^{A-1}
Y_{k}(t) $,
 the last sum cannot be expressed in terms of $X(t)$ or $Y (t)$, when taking
the limit. 
In fact, the last sum in \eqref{eq_supp1_dilemma} is not continuous, when
viewed as a function on $S$, cf. Figure~\ref{fig_fast_clustering}.
Therefore, trying to prove convergence directly at the level of generators seems destined for failure.

At the same time, the behaviour described in Figure~\ref{fig_fast_clustering}
indicates that within an infinitesimal
time frame, all particles above the threshold $A$ are clustering together, leaving
behind sites that are occupied by exactly $A$ particles only.
This suggests that for positive times,
$\tfrac{1}{L} \#_{A} \eta(t)$ can 
be well approximated in terms of \eqref{eq_supp1_dilemma}, by neglecting the last
term on the right--hand side.
The proof of Theorem~\ref{thm_main} is based on this heuristic.

\subsection{Proof of Theorem~\ref{thm_main}}

The first step towards the proof of Theorem~\ref{thm_main} is to approximate
the generator $\mL $ of the modulated PD diffusion by $\mathfrak{L}_{L,N}$
\eqref{e_IP_gen}. 
To this end, we recall the embedding \eqref{e_def_embd_direct} of particle
configurations into $ \overline{\nabla} \times S_{Y}$:
\begin{equation*}
\begin{aligned}
\embdcomb ( \eta) = 
 \big( \tfrac{1}{N} \widehat{ \eta}^{\,+}, \tfrac{1}{L} (\#_{0}
\eta, \ldots , \#_{A-1} \eta )\big)\,.
\end{aligned}
\end{equation*}
Moreover, we define 
\begin{equation*}
\begin{aligned}
\#_{>A} \eta := \sum_{i=1}^{L}\mathds{1}_{ \eta_{i}> A}\,,
\end{aligned}
\end{equation*}
which counts the number of non--zero elements in $\widehat{\eta}^{+}$.
As pointed out in the previous section, our estimates do not suffice to show convergence of the generators
 in the strong operator topology. Instead, the approximation holds only up to an error given in terms of
$\#_{>A} \eta$.

\begin{lemma}[Generator approximation]\label{lem_gen_est}
For every $L,N$ and $ \eta \in \Omega_{L,N}$, we have that for all $ f \in \mD
$ there exists a finite constant $ C_{f}$ (independent of $N $ and $L$) such
that 
\begin{equation*}
\begin{aligned}
\Big|\mathfrak{L}_{L,N} ( f \circ \embd) (\eta)
- 
\mL f ( \embd ( \eta))\Big|
\leqslant C_{f} \frac{\#_{>A} \eta}{ L}   + o_{f}(1) \,,
\end{aligned}
\end{equation*}
where $o_{f}(1)$ denotes a constant that vanishes in $\tdlim$, uniformly over $
\Omega_{L,N}$.
Here, $\mL f \in \mD $ is interpreted as the natural extension
from $ S$ to $ \overline{\nabla} \times S_{Y}$. 
\end{lemma}

The proof of the lemma is the content of Section~\ref{sec_conv_gen}.
In order for this estimate to
help us prove Theorem~\ref{thm_main}, we need $\#_{>A} \eta(t) = o(L)$ for 
almost all $t$.

\begin{proposition}[Instantaneous condensation]\label{prop_occ}
Let $ \big(\eta^{(L,N)} (t) \big)_{t \geqslant 0}$
be the inclusion process generated by $\mathfrak{L}_{L,N}$, with rates
satisfying Assumption~\ref{ass}. If started from $ \overline{\eta} =
\overline{\eta}^{(L,N)} \in
\Omega_{L,N}$ 
such that $ \lim_{\tdlim} \gamma_{N}( \overline{\eta}) \in [0,1]$,
then for every $ T \in (0, \infty)$ 
\begin{equation}\label{eq_inst_cond}
\begin{aligned}
 \lim_{\tdlim} \int_{0}^{T} \frac{\#_{>A} \eta^{(L,N)} (t)}{L} \ud t 
= 0 \,, 
\quad \text{in $ \mathbf{P}$-probability}\,,
\end{aligned}
\end{equation}
i.e.~the fraction of sites hosting particles in the fast phase vanishes
for almost every time.
\end{proposition}

The proof of the proposition is deferred to Section~\ref{sec_occ}.
Note that if we were given the convergence in Theorem~\ref{thm_main}, then
Proposition~\ref{prop_occ} would be an immediate consequence of
Corollary~\ref{cor_concentration}. In fact, this is how instantaneous
condensation of the inclusion process for $A=0$ was derived in \cite{CGG24}.
To our best knowledge, this is the first proof that shows directly
the instantaneous condensation of mass in the inclusion process, when taking
the thermodynamic limit.
It would be of interest to generalise the proof to a
general class of condensing stochastic particle systems.
\\

Combining the two key ingredients presented above,
we are now ready to conclude our main result: The convergence of the inclusion
process with a non--trivial slow phase to the modulated Poisson--Dirichlet
diffusion.

\begin{proof}[Proof of Theorem~\ref{thm_main}]

Let $(x,y) \in S $ and  $ \eta^{(L,N)}(0) \in \Omega_{L,N}$ be a sequence of
initial conditions such that $ \embdcomb(
\eta^{(L,N)}(0)) \to (x, y)$.
Following \cite[Theorem~3.9.4]{EK_book} in combination with \cite[Theorem~3.9.1]{EK_book},
tightness of the sequence of processes $\big(\embd( \eta^{(L,N)})\big)_{N,L}$
in $ D ([0, \infty), \overline{\nabla} \times S_{Y})$  holds, if
\begin{equation}\label{e_tight_supp}
\begin{aligned}
\limsup_{\tdlim} \mathbf{E}\Big[
\sup_{ t \in [0,T]} \big| \mathfrak{L}_{L,N} ( f \circ \embd) \big(  \eta^{
(L,N)}(t)\big)\big|
\Big]< \infty\,, \quad \forall f \in \mD \,,\ T \in (0, \infty) \,,
\end{aligned}
\end{equation}
since $ \eta^{(L,N)}$ is the unique solution of the martingale problem
for  $\mathfrak{L}_{L,N}$.
As a consequence of Lemma~\ref{lem_gen_est}, \eqref{e_tight_supp} is equivalent to 
\begin{equation}\label{e_tight_supp2}
\begin{aligned}
\limsup_{\tdlim} \mathbf{E}\Big[
\sup_{ t \in [0,T]} \big| \mL  f \big(\embd \big(  \eta^{(L,N)}(t)\big)\big)\big|
\Big]< \infty\,, \quad \forall f \in \mD \,,\ T \in (0, \infty) \,,
\end{aligned}
\end{equation}
which holds because $ \mL f \in C ( \overline{\nabla}\times S_{Y})$, and thus $ \| \mL f \|_{\infty}< \infty$. 
Hence, $\big(\embd( \eta^{(L,N)})\big)_{N,L}$ is tight in $ D([0,
\infty), \overline{\nabla}\times S_{Y})$ and admits
 subsequential limits. 

Without loss of generality, let us  assume 
\begin{equation}\label{eq_conv}
\begin{aligned}
\big(\embd( \eta^{(L,N)}(t))\big)_{t \geqslant 0} \weakconv (X (t), Y (t))_{t \geqslant 0}\,, \qquad
\text{in }\ D ([0, \infty), \overline{\nabla} \times S_{Y})\,,
\end{aligned}
\end{equation}
for some random accumulation point $ (X,Y) \in D ([0, \infty), \overline{\nabla}\times
S_{Y})$.
In fact, $ (X, Y ) \in D ([0, \infty), S)$, because $ \lim_{\tdlim}  \embd(
\eta^{(L,N)}) \in S $ by definition, for any accumulation point.
It only remains to identify $(X,Y)$ as the modulated Poisson--Dirichlet
diffusion.
For this, we again leverage on the  martingale problem for $\mL $: 
First, we notice that 
\begin{equation*}
\begin{aligned}
t \mapsto
M_{L,N}^{ f} (t):=
f \big(
\embd( \eta^{ (L,N)}(t))
\big)
-
\int_{0}^{t} 
\mathfrak{L}_{L,N}(f 
\circ \embd)( \eta^{ (L,N)}(s))
\ud s
\end{aligned}
\end{equation*}
is a martingale for every $N,L$.
Next, 
we
replace  $\mathfrak{L}_{L,N}$ with $ \mL$ up to the error--term from Lemma~\ref{lem_gen_est}, such that
\begin{equation*}
\begin{aligned}
&
 \mathbf{E} \Big[
\sup_{ t \in [0,T]} 
\Big|
f \big(
\embd( \eta^{(L,N)}(t))
\big)
-
\int_{0}^{t} 
\mL f 
\big( \embd ( \eta^{ (L,N)}(s))\big)
\ud s
-
M_{L,N}^{f}(t) 
\Big|
\Big] \\
&\qquad 
\leqslant C_{f}
 \mathbf{E} \Big[
\int_{0}^{T} \frac{\#_{>A} \eta^{(L,N)}(s)}{L}  \ud s
\Big] +  o_{f}(1) T \to 0
\,,
\end{aligned}
\end{equation*}
as $ \tdlim$, 
where the convergence in the last step is a consequence of Proposition~\ref{prop_occ}.
Thus, by \eqref{eq_conv}
\begin{equation*}
\begin{aligned}
M^{f}(t):=\lim_{ \tdlim} 
M_{L,N}^{f}(t) 
=
f \big(
X(t), Y(t)
\big)
-
\int_{0}^{t} 
\mL f \big(
X(s), Y(s)\big)
\ud s\,,
\end{aligned}
\end{equation*}
in the (probabilistically) weak sense.
Lastly, $ M^{f}_{L,N}(t)$ is uniformly bounded for every $t$, therefore, the limit
$M^{f}$ is a martingale, using the dominated convergence theorem.
Hence, $(X,Y)$ is the unique solution of the martingale problem for $\mL$, 
 which
identifies it as the modulated Poisson--Dirichlet diffusion, see Lemma~\ref{lem_hierarchy}.
\end{proof}

\begin{remark}\label{rem_generalisation_rates}
Extending the proof of Theorem~\ref{thm_main} to more general choices of $
\zeta_{L}$ in Assumption~\ref{ass}, beyond the order $ \mO (L^{-1})$, will
require a more precise estimate than that
provided in  Proposition~\ref{prop_occ}.
More precisely, if $ L^{-1} \ll \zeta_{L} = o(1) $, the bound in
Lemma~\ref{lem_gen_est} will be expressed in terms of $ \zeta_{L} \#_{>A} \eta$.
To address this, it would be necessary to show that
\begin{equation*}
\begin{aligned}
 \lim_{\tdlim} \zeta_{L}\int_{0}^{T}  \#_{>A} \eta^{(L,N)} (t) \ud t 
= 0 \,, 
\quad \text{in $ \mathbf{P}$-probability}\,.
\end{aligned}
\end{equation*}
\end{remark}

\begin{remark}\label{rem_gen_conv}
In Theorem~\ref{thm_main}, we prove convergence of embedded
Markov processes, generated by $ \mathfrak{L}_{L,N}$, to a Feller process on $ S$
with generator $ \mL$. 
However, as mentioned below \eqref{eq_supp1_dilemma}, we are unable to prove convergence of the
corresponding generators directly. 
This is surprising, as the two statements are equivalent according to the Trotter--Kurtz
convergence theorem, see e.g.~\cite[Theorem~17.25]{Kallenberg2021}.
Currently, we do not have a satisfactory explanation for why it is not possible
to prove the convergence of generators directly. Although, we expect
this to be related to the instantaneous condensation and ``projection``
 of the limiting process from $ \overline{\nabla}$ onto $ \nabla$ for all
positive times, which is precisely the missing probabilistic ingredient to
conclude the main result from Lemma~\ref{lem_gen_est}.

Encountering similar difficulties, the authors of \cite{CBERS17} 
were also unable to exploit the explicit form of the generator 
to prove the convergence of a Wright–Fisher diffusion model to the 
two-parameter Poisson–Dirichlet diffusion, albeit for different reasons. 
Instead, they also established convergence using the martingale problem formalism.
\end{remark}

\section{Instantaneous condensation}\label{sec_occ}

In this section, we prove Proposition~\ref{prop_occ}. 
First, we will present the two key ingredients, Lemma~\ref{lem_apriori} and
Lemma~\ref{lem_first_site}, and apply them to conclude the proposition.
The remainder of the section is then devoted to the proofs of the two lemmas.

The first step consists of a priori estimates for
\begin{equation*}
\begin{aligned}
\gamma_{N} ( \eta)
:=
\sum_{i=1}^{L} \frac{ \eta^{+}_{i}}{N}
=
1 
- \frac{L}{N} \sum_{k =0}^{A} k \frac{\#_{k} \eta}{ L} 
 \,,
\end{aligned}
\end{equation*}
namely  the relative mass of particles
in the fast phase.
The following 
lemma establishes that if the initial mass in the fast phase satisfies $\gamma > 0$, 
it remains bounded away from zero at all times with high probability. 
Conversely, if the fast phase is initially empty, it remains empty at all times.

\begin{lemma}\label{lem_apriori}
Let $ \overline{\eta} = \overline{\eta}^{(L,N)}\in \Omega_{L,N}$ be a sequence
of initial conditions. 
\begin{enumerate}
\item[(i)] If $ \liminf_{ \tdlim} \gamma_{N} ( \overline{\eta}) >0$, then 
for every $ T>0$ there exists $ \delta \in (0,1)$ such that 
\begin{equation*}
\begin{aligned}
\lim_{\tdlim} \mathbf{P}\Big(
\inf_{t \in [0,T]} \gamma_{N}( \eta (t)) \leqslant \delta
\Big)=0
\,.
\end{aligned}
\end{equation*}

\item[(ii)]  If $ \limsup_{ \tdlim} \gamma_{N} ( \overline{\eta}) =0$, then for every
$T>0$ and $ \delta>0$
\begin{equation*}
\begin{aligned}
\lim_{\tdlim} \mathbf{P}\Big(
\sup_{t \in [0,T]} \gamma_{N}( \eta (t)) \geqslant  \delta
\Big)=0
\,.
\end{aligned}
\end{equation*}

\end{enumerate}
\end{lemma}

The proof of Lemma~\ref{lem_apriori} is deferred to Section~\ref{sec_aprio}.
Because the dynamics in the modulated PD diffusion are driven by $ \gamma$, 
the a priori estimate in (i) guarantees that the limiting inclusion dynamics are
sufficiently strong to drive the exchange of particles in the system. 
In light of this event,
we can conclude that almost every site instantly repels any
excess particles in the fast phase that it may have started with.

\begin{lemma}\label{lem_first_site}
For every $M \in \mathbf{N}$,
 $ T \in (0, \infty)$ and $ \delta >0$
\begin{equation*}
\begin{aligned}
\lim_{\tdlim} 
\sup_{ \substack{ \overline{\eta}\in \Omega_{L,N}\\ \overline{\eta}_{1}
\leqslant M\,, \ \gamma_{N}( \overline{\eta}  ) > \delta}}
\mathbf{E}_{ \overline{\eta}}
\bigg[
\int_{0}^{T} \mathds{1}_{\eta_{1} (t)>A} \ud t
\bigg\vert
 \inf_{t \in [0,T]} \gamma_{N}( \eta (t)) > \delta
\bigg] =0\,,
\end{aligned}
\end{equation*}
where $ ( \eta (t) )_{t \geqslant 0} =\big(\eta^{(L,N)} (t) \big)_{t \geqslant
0}$ is the inclusion process started from $ \overline{\eta}$ generated by  $\mathfrak{L}_{L,N}$, with rates
satisfying Assumption~\ref{ass}.
\end{lemma}

The proof of Lemma~\ref{lem_first_site} is a consequence of comparing $
\eta_{1}(t)$ to a
continuous--time random walk, which is subject of Sections~\ref{sec_first_site} and
\ref{sec_coup}.\\

Equipped with the statements of Lemmas~\ref{lem_apriori} and \ref{lem_first_site}, we are ready to
provide a proof of Proposition~\ref{prop_occ}, based on a simple symmetrisation
argument.

\begin{proof}[Proof of Proposition~\ref{prop_occ}]
First, we consider the case $ \lim_{ \tdlim} \gamma_{N} ( \overline{\eta})
=0$. Then we can simply bound $\#_{>A} \eta$ by the total numbers of particles
in the fast phase $ \gamma_{N}( \eta) N$, such that
\begin{equation*}
\begin{aligned}
\int_{0}^{T} \frac{\#_{>A} \eta (t)}{L} \ud t 
\leqslant 
T \frac{N}{L} \sup_{t \in [0,T]} \gamma_{N} ( \eta (t))
\to 0 \,, \quad \text{in } \mathbf{P}-\text{probability}\,,
\end{aligned}
\end{equation*}
accordingly to Lemma~\ref{lem_apriori}(ii).\\

On the other hand, when  $ \lim_{ \tdlim} \gamma_{N} ( \overline{\eta})
> 0$, we make use of the fact that the inclusion process $ (\eta (t))_{t \geqslant 0}$ if fully symmetric, due to the underlying
complete graph. 
Consequently, 
\begin{equation*}
\begin{aligned}
\mathbf{P}_{ \overline{\eta}}  \big( \eta_{i} (t)>A\big)
=
\mathbf{P}_{ \overline{\eta}} \big( \eta_{\sigma_{i}(1)} (t)>A\big)
=
\mathbf{P}_{ \overline{\eta}_{\sigma_{i}}} \big( \eta_{1} (t)>A\big)\,,
\end{aligned}
\end{equation*}
where
 $ \sigma_{i} =(1,i) = \sigma_{i}^{-1} $ is the permutation on
$\{1,\ldots,L\}$ that
switches $i$ and $1$. 

Therefore, for every $ t \geqslant 0 $ we can write 
\begin{equation*}
\begin{aligned}
\mathbf{E}_{ \overline{\eta}}
\bigg[
\frac{\#_{> A} \eta (t)}{L} 
\bigg]
&=
\frac{1}{L}
\sum_{i =1}^{L}
\mathbf{P}_{ \overline{\eta}} \big( \eta_{i} (t)>A\big)
=
\frac{1}{L}
\sum_{i =1}^{L}
\mathbf{P}_{ \overline{\eta}_{\sigma_{i}}} \big( \eta_{1} (t)>A\big)\,.
\end{aligned}
\end{equation*}
Together with Fubini's theorem, this implies
\begin{equation}\label{eq_symm_supp3}
\begin{aligned}
\mathbf{E}_{ \overline{\eta}}\bigg[
\int_{0}^{T} 
\frac{\#_{> A} \eta (t)}{L} 
\ud t
\bigg]
=
\frac{1}{L}  \sum_{ i=1}^{L}
\mathbf{E}_{ \overline{\eta}_{\sigma_{i}}}
\bigg[
\int_{0}^{T} 
\mathds{1}_{ \eta_{1} (t) >A}
\ud t
\bigg]\,.
\end{aligned}
\end{equation}
Next, we upper bound each summand 
using
\begin{equation*}
\begin{aligned}
\mathbf{E}_{ \overline{\eta}_{\sigma_{i}}}
\bigg[
\int_{0}^{T} 
\mathds{1}_{ \eta_{1} (t) >A}
\ud t
\bigg]
&\leqslant 
T\, \mathbf{P}_{ \overline{\eta}}
\big( 
\inf_{t \in [0,T]} \gamma_{N} (\eta (t)) \leqslant \delta
\big)\\
& \quad +
\mathbf{E}_{ \overline{\eta}_{\sigma_{i}}}
\bigg[
\int_{0}^{T} 
\mathds{1}_{ \eta_{1} (t) >A}
\ud t
\bigg\vert 
\inf_{t \in [0,T]} \gamma_{N} (\eta (t))> \delta
\bigg]\,,
\end{aligned}
\end{equation*}
with $ \delta = \delta (T, \overline{\eta}) \in (0,1)$ as in Lemma~\ref{lem_apriori}(i).
Note that the contribution of the first term on the
right--hand side of \eqref{eq_symm_supp3} vanishes, as a consequence of Lemma~\ref{lem_apriori}(i). 
Thus, 
\begin{equation}\label{eq_symm_supp4}
\begin{aligned}
\mathbf{E}_{ \overline{\eta}}\bigg[
\int_{0}^{T} 
\frac{\#_{> A} \eta (t)}{L} 
\ud t
\bigg]
\leqslant 
\frac{1}{L}  \sum_{ i=1}^{L}
\mathbf{E}_{ \overline{\eta}_{\sigma_{i}}}
\bigg[
\int_{0}^{T} 
\mathds{1}_{ \eta_{1} (t) >A}
\ud t
\bigg\vert 
\inf_{t \in [0,T]} \gamma_{N} (\eta (t))> \delta
\bigg]+ o(1)\,.
\end{aligned}
\end{equation}
Let $ M \in \mathbf{N}$, we split the sum on the right--hand side into two parts with the partitioning
$\{ \overline{\eta}_{i}~>~M\}~\cup \{\overline{\eta}_{i}~\leqslant~M\}$,
such that
\begin{equation}\label{eq_symm_supp5}
\begin{aligned}
&\frac{1}{L}  \sum_{ i=1}^{L}
\mathbf{E}_{ \overline{\eta}_{\sigma_{i}}}
\bigg[
\int_{0}^{T} 
\mathds{1}_{ \eta_{1} (t) >A}
\ud t
\bigg\vert 
\inf_{t \in [0,T]} \gamma_{N} (\eta (t))> \delta
\bigg]\\
&\leqslant 
\frac{1}{L} 
\sum_{ i =1}^{L} \mathds{1}_{ \overline{\eta}_{i} \leqslant  M}
\sup_{ \substack{ \overline{\eta}'\in \Omega_{L,N}\\ \overline{\eta}'_{1}
\leqslant M\,, \ \gamma_{N}( \overline{\eta}'  ) > \delta}}
\mathbf{E}_{ \overline{\eta}'}
\bigg[
\int_{0}^{T} \mathds{1}_{\eta_{1} (t)>A} \ud t
\bigg\vert
 \inf_{t \in [0,T]} \gamma_{N}( \eta (t)) > \delta
\bigg]
+
\frac{T}{M} \frac{N}{L} \,,
\end{aligned}
\end{equation}
where we used that
$ \sum_{i=1}^{L} \mathds{1}_{ \overline{\eta}_{i}> M} \leqslant
\tfrac{N}{M}$.
Hence, \eqref{eq_symm_supp5}, and thus \eqref{eq_symm_supp4}, vanish when first taking the
thermodynamic limit $\tdlim$, before $ M \to \infty$, using
Lemma~\ref{lem_first_site}.
This concludes convergence in $ L^{1}( \mathbf{P})$, and therefore in
probability.
\end{proof}

\subsection{A priori estimates for $ \gamma_{N}$}\label{sec_aprio}

In the following, we prove Lemma~\ref{lem_apriori}.
Namely, that with high
probability as $\tdlim$
\begin{equation*}
\left.\begin{aligned}
  \liminf_{\tdlim} \gamma_{N} ( \overline{\eta})>0\\
 \limsup_{\tdlim} \gamma_{N} ( \overline{\eta})=0
\end{aligned}\right\}
\quad\Longrightarrow\quad
\left\{\begin{aligned}
  \inf_{t \in [0,T]} \gamma_{N}( \eta (t))>0\,,\\
  \sup_{t \in [0,T]} \gamma_{N}( \eta (t))\ll 1\,.
\end{aligned}\right.
\end{equation*}
Thus, provided we started with a positive mass in the fast phase, it will
remain uniformly positive with high probability.

\begin{proof}[Proof of Lemma~\ref{lem_apriori}]
First, we apply Dynkin's formula, to see that for any $ \overline{c}>0$
\begin{equation}\label{e_dynkin}
\begin{aligned}
e^{ \overline{c} t} \gamma_{N} ( \eta (t))
&= \gamma_{N} ( \overline{\eta})
+ \int_{0}^{t} \big( \partial_{s} + \mathfrak{L}_{L,N}\big) \big( e^{
\overline{c} \cdot} \gamma_{N} (\cdot)\big) (s, \eta (s)) \ud s
+ M_{L,N}(t)\\
& = 
 \gamma_{N} ( \overline{\eta})
+ \int_{0}^{t} 
e^{ \overline{c} s} 
\big( \overline{c}\,  \gamma_{N}( \eta(s)) + \mathfrak{L}_{L,N} \gamma_{n}( \eta
(s))\big) \ud s
+ M_{L,N}(t)
\end{aligned}
\end{equation}
where  $M_{L,N}$ is a martingale with predictable quadratic variation
\begin{equation}\label{eq_qv}
\begin{aligned}
\langle M_{L,N}\rangle_{t}
= \int_{0}^{t} \mathfrak{L}_{L,N} \gamma_{N}^{2}( \eta (s))- 2
\gamma_{N}( \eta (s)) \mathfrak{L}_{L,N} \gamma_{N} ( \eta (s)) \ud s
\,.
\end{aligned}
\end{equation}
The predictable quadratic variation $\langle M_{L,N}\rangle_{t}
$ vanishes, 
because uniformly in $ \eta \in \Omega_{L,N}$ 
\begin{equation}\label{eq_cdchamp}
\begin{aligned}
 &0 \leqslant \mathfrak{L}_{L,N} \gamma_{N}^{2}( \eta)- 2
\gamma_{N}( \eta) \mathfrak{L}_{L,N} \gamma_{N} ( \eta)\\
&\qquad= 
\sum_{i ,j =1}^{L}
u_{1}( \eta_{i}) u_{2}( \eta_{j})
\big[
\gamma_{N}( \eta^{i,j}) - \gamma_{N}( \eta) 
\big]^{2}\\
& \qquad  \leqslant 
\frac{1}{N^{2}} 
\sum_{i ,j =1}^{L}
u_{1}( \eta_{i}) u_{2}( \eta_{j})
\big[
 \mathds{1}_{\eta_{i}>A} \mathds{1}_{\eta_{j}< A}
+ \mathds{1}_{\eta_{i} \leqslant A} \mathds{1}_{\eta_{j} \geqslant
A}
\big]
= \mO \big( \tfrac{1}{L})\,,
\end{aligned}
\end{equation}
where we used Assumption~\ref{ass} in the last equality. The inequality in \eqref{eq_cdchamp} is a consequence of 
\begin{equation}\label{eq_gam_diff}
\begin{aligned}
\gamma_{N} ( \eta^{i,j})
- \gamma_{N}( \eta) =
-\frac{1}{N} \mathds{1}_{\eta_{i}>A} \mathds{1}_{\eta_{j}< A}
+\frac{1}{N} \mathds{1}_{\eta_{i} \leqslant A} \mathds{1}_{\eta_{j} \geqslant
A}\,,
\end{aligned}
\end{equation}
because for $ \gamma_{N}$ to change, a particle has to jump either from
the slow to the fast phase or vice versa.
Moreover, \eqref{eq_gam_diff} implies 
\begin{equation}\label{eq_lb_Lgam}
\begin{aligned}
\mathfrak{L}_{L,N} \gamma_{N}( \eta) 
& \geqslant  
-
\frac{1}{N} 
\sum_{i=1}^{L}
u_{1}( \eta_{i})  \mathds{1}_{\eta_{i}>A}
\sum_{j=1}^{L}
 u_{2}( \eta_{j})
 \mathds{1}_{\eta_{j}< A}
\geqslant 
- 2 ( \overline{r} +1) \gamma_{N} ( \eta)
\,.
\end{aligned}
\end{equation}
where we used  Assumption~\ref{ass} in the second inequality, specifically that $
u_{1}( \eta_{i}) \leqslant 2 ( \eta_{i}-A)$ whenever $ \eta_{i}>A$.\\

We first consider the case $ \liminf_{\tdlim} \gamma_{N} ( \overline{\eta})
>0$.
For every $ T \in (0, \infty)  $ and $ \delta>0 $, we can upper bound
\begin{equation}\label{eq_prob_gamma_est}
\begin{aligned}
\mathbf{P} \Big( 
\inf_{t \in [0,T]} \gamma_{N}( \eta (t)) < \delta
\Big) 
&\leqslant 
\mathbf{P} \Big( 
\inf_{t \in [0,T]} \gamma_{N}( \eta (t)) < \delta
\ \Big\vert
\sup_{t \in [0,T]}
|M_{L,N}(t)| \leqslant 
 \varepsilon\Big) \\
&\quad+ 
\mathbf{P} \Big( 
\sup_{t \in [0,T]}
|M_{L,N}(t)| > 
\varepsilon\Big) \,,
\end{aligned}
\end{equation}
for every $ \varepsilon >0$. 
The second term on the right--hand side vanishes as $\tdlim$, which is a
consequence of Doob's martingale inequality, see \eqref{eq_doob} below.
On the other hand, conditioned on $\sup_{t \in [0,T]}
|M_{L,N}(t)| \leqslant 
 \varepsilon$, \eqref{e_dynkin} (with $ \overline{c}=2 ( \overline{r}+1 ) $)
and \eqref{eq_lb_Lgam} yield the inequality 
\begin{equation}\label{e_gam_estimate}
\begin{aligned}
e^{ \overline{c}t} \gamma_{N} ( \eta(t))
\geqslant
\gamma_{N}( \overline{\eta} )
+
\int_{0}^{t} e^{ \overline{c} s} \big( \overline{c} \, \gamma_{N} (
\eta (s)) - 2 ( \overline{r}+1) \gamma_{N} ( \eta (s))\big) \ud s
- \varepsilon
=
\gamma_{N} ( \overline{\eta}) - \varepsilon
\,,
\end{aligned}
\end{equation}
for every $ t \in [0,T]$. Therefore, for all $N$ and $L$ large enough
\begin{equation*}
\begin{aligned}
\inf_{t \in [0,T]}\gamma_{N} ( \eta (t)) 
\geqslant 
\tfrac{1}{2}
e^{- 2 ( \overline{r}+1) T}
\big( \liminf_{\tdlim}  \gamma_{N} ( \overline{\eta}) - \varepsilon\big) >0\,,
\end{aligned}
\end{equation*}
such that the first term on the right--hand side of \eqref{eq_prob_gamma_est}
equals zero, provided $ \delta$ is small enough.\\

Now, consider the case $ \limsup_{\tdlim} \gamma_{N} ( \overline{\eta})=0$.
Then \eqref{e_dynkin} yields (with $ \overline{c}=0$)
\begin{equation*}
\begin{aligned}
\gamma_{N} ( \eta (t))
&= \gamma_{N} ( \overline{\eta}) 
+ \int_{0}^{t} \mathfrak{L}_{L,N} \gamma_{N}( \eta (s)) \ud s
+ M_{L,N}(t)\\
& 
\leqslant 
 \gamma_{N} ( \overline{\eta}) 
+ 2 ( \overline{q}+1 )  \int_{0}^{t} \gamma_{N}( \eta (s)) \ud s
+ \sup_{s \in [0,T]} |M_{L,N}(s)|
\end{aligned}
\end{equation*}
where we used 
\begin{equation*}
\begin{aligned}
\mathfrak{L}_{L,N} \gamma_{N} ( \eta)
\leqslant \frac{1}{N} 
\sum_{i=1}^{L}
u_{1}( \eta_{i})  \mathds{1}_{\eta_{i} \leqslant A}
\sum_{j=1}^{L}
 u_{2}( \eta_{j})
 \mathds{1}_{\eta_{j} \geqslant  A}
\leqslant 2 ( \overline{q}+1) \gamma_{N} ( \eta) \,,
\end{aligned}
\end{equation*}
which follows from \eqref{eq_gam_diff}, similarly as in \eqref{eq_lb_Lgam}. 
Therefore, Gr\"onwall's inequality yields
\begin{equation*}
\begin{aligned}
\gamma_{N} ( \eta (t))
\leqslant 
\Big( 
 \gamma_{N} ( \overline{\eta}) 
+
 \sup_{s \in [0,T]} |M_{L,N}(s)|
\Big)
e^{ 2 ( \overline{q}+1 )  t}\,,
\end{aligned}
\end{equation*}
which in turn gives the uniform $L^{1}( \mathbf{P}) $--bound
\begin{equation*}
\begin{aligned}
\mathbf{E}\Big[\sup_{t \in [0,T]} 
\gamma_{N} ( \eta (t))\Big]
\leqslant 
\Big( 
 \gamma_{N} ( \overline{\eta}) 
+
\mathbf{E}\Big[\sup_{t \in [0,T]} |M_{L,N}(t)|\Big]
\Big)
e^{ 2 ( \overline{q}+1 )  T}\,.
\end{aligned}
\end{equation*}
This upper bound vanishes, since by Doob's martingale inequality
\cite[Theorem~9.17]{Kallenberg2021}
and \eqref{eq_cdchamp} give
\begin{equation}\label{eq_doob}
\begin{aligned}
\mathbf{E}\Big[ \sup_{t \in [0,T]} |M_{L,N}(t)|^{2}\Big]
\leqslant
4 \mathbf{E}\big[ |M_{L,N} (T)|^{2}\big]
= 4 \mathbf{E}\big[
\langle M_{L,N}\rangle_{T}
\big]\to 0\,, \quad \text{as } \tdlim\,,
\end{aligned}
\end{equation}
and $ \limsup_{\tdlim} \gamma_{N} ( \overline{\eta})=0$ by assumption.
This concludes the proof.
\end{proof}

\subsection{Proof of Lemma~\ref{lem_first_site}}\label{sec_first_site}

The  inclusion process projected onto its first variable $ ( \eta_{1} (t))_{t \geqslant 0}$ 
can interpreted 
as a random walk on $\{0,\ldots, N\}$.
Its transition rates are inhomogeneous, in the sense that
they depend explicitly on the evolution of the whole
inclusion process $ ( \eta (t))_{t \geqslant 0}$: For $ f ( \eta) := \eta_{1} $
\begin{equation*}
\begin{aligned}
\mathfrak{L}_{L,N}f ( \eta) 
&= 
\sum_{i , j =1}^{L} u_{1} ( \eta_{i}) u_{2}( \eta_{j} ) \big[ f (
\eta^{i,j}) - f ( \eta) \big]\\
& = 
- u_{1}( \eta_{1}) 
\sum_{j =2}^{L}u_{2}( \eta_{j}) 
 + u_{2} ( \eta_{1}) \sum_{j =2}^{L} u_{1} ( \eta_{j} )\,.
\end{aligned}
\end{equation*}
In other words, 
$ ( \eta_{1} (t))_{t \geqslant 0}$
is an inhomogeneous continuous--time random walk, which at time $t$
\begin{equation*}
\begin{aligned}
\begin{cases} 
\text{moves up }+1\,, \ \text{with rate } &  u_{2}( \eta_{1}(t)) \sum_{i =2}^{L}
u_{1}( \eta_{i}(t))= p ( \eta(t)) c ( \eta (t))  \\
\text{moves down }-1\,,  \  \text{with rate } & u_{1} ( \eta_{1}(t)) \sum_{j =2}^{L}
u_{2} ( \eta_{j}(t))= (1- p (\eta(t))) c ( \eta(t))  \,.
\end{cases}
\end{aligned}
\end{equation*}
Here we conveniently introduced the total jump rate
$ c : \Omega_{L,N} \to [0, \infty)$
\begin{equation}\label{eq_c_totalrate}
\begin{aligned}
c ( \eta) := 
 u_{1}( \eta_{1}) 
\sum_{j =2}^{L}u_{2}( \eta_{j}) 
+ u_{2} ( \eta_{1}) \sum_{j =2}^{L} u_{1} ( \eta_{j} )\,,
\end{aligned}
\end{equation}
and the probability $ p ( \eta) : \Omega_{L,N} \to [0,1]$
\begin{equation}\label{eq_def_p}
\begin{aligned}
p ( \eta) :=
\frac{1}{c ( \eta)} 
 u_{2} ( \eta_{1}) \sum_{j =2}^{L} u_{1} ( \eta_{j} )\,,
\end{aligned}
\end{equation}
of performing an upward jump.\\

If we were to find rates $ c'_{L}, c''_{L}>0$ and $ p'_{L} \in [0,1]$ such that 
\begin{equation}\label{eq_supp_bndrates}
\begin{aligned}
\sup_{ \substack{\eta \in \Omega_{L,N}\\ \eta_{1}> A}} p ( \eta) \leqslant
p'_{L}
\,, \quad \inf_{ \substack{\eta \in \Omega_{L,N}\\ \eta_{1}>A}} c ( \eta)
\geqslant   c'_{L}
\quad \text{and}\quad 
\sup_{ \substack{\eta \in \Omega_{L,N}\\ \eta_{1} \leqslant A}} c ( \eta)
 \leqslant   c''_{L}\,,
\end{aligned}
\end{equation}
we could construct a new continuous--time random walk $ (\mZ (t))_{t \geqslant
0} = (\mZ^{(L)} (t))_{t \geqslant 0}$  on $
\{A, \ldots, N\}$ (started from $ \max\{\eta_{1}(0) ,A \}$), which
\begin{equation}\label{eq_def_ratesZ}
\begin{aligned}
\begin{cases} 
\text{moves up }+1\,, \ \text{with rate } & p'_{L}  c'_{L}  \mathds{1}_{Z>A} +
c''_{L} \mathds{1}_{\mZ =A}\\
\text{moves down }-1\,,  \  \text{with rate } & ( 1- p'_{L}) c'_{L}
\mathds{1}_{\mZ>A}   \,.
\end{cases}
\end{aligned}
\end{equation}
This new random walk is more likely to jump upwards than $ \eta_{1}(t)$, and
takes longer to perform jumps when $ \mZ (t) >A$. On the other hand, once $ \mZ
(t) $ reaches $ A$, it spends less time there than $ \eta_{ 1}(t)$ before being
reflected, see Figure~\ref{fig_rw_bound}. 
Consequently,  
\begin{equation}\label{eq_coupl_supp1}
\begin{aligned}
\mathbf{P}\bigg(
\int_{0}^{T} \mathds{1}_{ \eta_{1}(t)>A} \ud t > \varepsilon\bigg)
\leqslant 
\mathbf{P}\bigg(
\int_{0}^{T} \mathds{1}_{ \mZ (t)>A} \ud t > \varepsilon\bigg)\,,
\end{aligned}
\end{equation}
and it suffices to study the behaviour of the homogeneous random walk $ \mZ $,
which is much easier.

\begin{figure}[H]
\centering
\begin{tikzpicture}
\fill[gray!20] (0,0) rectangle (13.4,1.); 
\draw[very thick] (0,-0.4) -- (13.4,-.4);
\draw[very thick] (0,-0.55) -- (0,-0.25); 
\draw[very thick] (1.1,-0.55) -- (1.1,-0.25); 
\draw[very thick] (5.12,-0.55) -- (5.12,-0.25); 
\draw[very thick] (8.05,-0.55) -- (8.05,-0.25); 
\draw[very thick] (12.03,-0.55) -- (12.03,-0.25); 

\draw[very thick, purple] (0,2.5) -- (13.4,2.5);
\draw[very thick, purple] (0,2.65) -- (0,2.35); 
\draw[very thick, purple] (1.52,2.65) -- (1.52,2.35); 
\draw[very thick, purple] (2.07,2.65) -- (2.07,2.35); 
\draw[very thick, purple] (9.45,2.65) -- (9.45,2.35); 
\draw[very thick, purple] (8.96,2.65) -- (8.96,2.35); 
\draw[very thick, purple] (10.92,2.65) -- (10.92,2.35); 
\draw[very thick, purple] (12.5,2.65) -- (12.5,2.35);

    \begin{axis}[
        width=15cm, height=5cm,
        axis lines=middle,
        xlabel={\scriptsize time $t$},
        xmin=0, xmax=11,
        ymin=0.0, ymax=1.7,
        xtick=\empty, ytick=\empty,
    ]
    
    \addplot[black, very thick] coordinates {
(0.00,0.70) (0.05,0.70) (0.10,0.75) (0.15,0.80) (0.20,0.75) (0.25,0.80)
(0.30,0.75) (0.35,0.75) (0.40,0.7) (0.45,0.70) (0.50,0.75) (0.55,0.75)
(0.60,0.70) (0.65,0.65) (0.70,0.70) (0.75,0.65) (0.80,0.60) (0.85,0.55)
(0.90,0.5) (0.95,0.5) (1.00,0.5) (1.05,0.50) (1.10,0.5) 
(1.15,0.50) (1.20,0.50) (1.25,0.50) (1.30,0.50) (1.35,0.50)  
(1.40,0.45) (1.45,0.45) (1.50,0.45) (1.55,0.45) (1.60,0.45) (1.65,0.45) (1.70,0.45) (1.75,0.45) (1.80,0.45) (1.85,0.45)
(1.90,0.40) (1.95,0.40) (2.00,0.40) (2.05,0.40) (2.10,0.40) (2.15,0.40)
(2.20,0.40) (2.25,0.40) (2.30,0.40) (2.35,0.40) (2.40,0.40) (2.45,0.40)
(2.50,0.40) (2.55,0.40)
(2.60,0.40) (2.65,0.40) (2.70,0.40) (2.75,0.40) (2.80,0.40) (2.85,0.40)
(2.90,0.40)
(2.95,0.40) (3.00,0.40) (3.05,0.40) (3.10,0.40) (3.15,0.40) (3.20,0.40)
(3.25,0.40) (3.30,0.40) (3.35,0.45) (3.40,0.45) (3.45,0.45) (3.50,0.45)
(3.55,0.45) (3.60,0.45)
(3.65,0.45) (3.70,0.45) (3.75,0.45) (3.80,0.45) (3.85,0.45) (3.90,0.5)
(3.95,0.5) (4.00,0.5) (4.05,0.5) (4.10,0.5) (4.15,0.5) (4.20,0.5)
(4.20,0.50) (4.25,0.55) (4.30,0.55) (4.35,0.60) (4.40,0.60) (4.45,0.65)
(4.50,0.60) (4.55,0.65) (4.60,0.65) (4.65,0.70) (4.70,0.65) (4.75,0.65)
(4.80,0.65) (4.85,0.70) (4.90,0.70) (4.95,0.65) (5.00,0.65) (5.05,0.60)
(5.10,0.55) (5.15,0.55) (5.20,0.60) (5.25,0.60) (5.30,0.65) (5.35,0.65)
(5.40,0.70) (5.45,0.75) (5.50,0.80) (5.55,0.80) (5.60,0.80) (5.65,0.75)
(5.70,0.70) (5.75,0.65) (5.80,0.70) (5.85,0.75) (5.90,0.75) (5.95,0.75)
(6.00,0.70) (6.05,0.65) (6.10,0.65) (6.15,0.65) (6.20,0.65) (6.25,0.60)
(6.30,0.55) (6.35,0.55) (6.40,0.60) (6.45,0.60) (6.50,0.60) (6.55,0.55)
 (6.60,0.50) (6.65,0.50) (6.70,0.50) (6.75,0.50) (6.80,0.50) (6.85,0.50)
(6.90,0.50) (6.95,0.50) (7.00,0.50) (7.05,0.50) (7.10,0.50) (7.15,0.50)
(7.20,0.50) (7.25,0.50) (7.30,0.50) (7.35,0.50) (7.40,0.50) (7.45,0.50)
(7.50,0.50) (7.55,0.50) (7.60,0.50) (7.65,0.50) (7.70,0.50) (7.75,0.50)
(7.80,0.50) (7.85,0.50) (7.90,0.50) (7.95,0.50) (8.00,0.50) (8.05,0.50)
(8.10,0.50) (8.15,0.50) (8.20,0.45) (8.25,0.450) (8.30,0.450) (8.35,0.450)
(8.40,0.450) (8.45,0.450) (8.50,0.450) (8.55,0.45) (8.60,0.450) (8.65,0.450)
(8.70,0.450) (8.75,0.450) (8.80,0.450) (8.85,0.450) (8.90,0.450) (8.95,0.450)
(9.00,0.450) (9.05,0.45) (9.10,0.45) (9.15,0.45) (9.20,0.45) (9.25,0.45)
(9.30,0.45) (9.35,0.45) (9.40,0.45) (9.45,0.450) (9.50,0.450) (9.55,0.450)
(9.60,0.450) (9.65,0.450) (9.70,0.450) (9.75,0.50) (9.80,0.50)
(9.85,0.50) (9.90,0.55) (9.95,0.60) (10.00,0.65) (10.05,0.60) (10.10,0.55)
(10.15,0.60) (10.20,0.60) (10.25,0.60) (10.30,0.65) (10.35,0.70) (10.40,0.65)
(10.45,0.70) (10.50,0.65) (10.55,0.70) (10.60,0.65) (10.65,0.70) (10.70,0.65)
(10.75,0.70) (10.80,0.70) (10.85,0.75) (10.90,0.75) (10.95,0.70) (11.00,0.65)
(11.05,0.70)  };
    
    \addplot[purple, very thick] coordinates {
(0.00,0.70) (0.05,0.75) (0.10,0.75) (0.15,0.80) (0.20,0.85) (0.25,0.85)
(0.30,0.90) (0.35,0.85) (0.40,0.8) (0.45,0.8) (0.50,0.85) (0.55,0.85)
(0.60,0.9) (0.65,0.95) (0.70,0.95) (0.75,0.90) (0.80,0.85) (0.85,0.80)
(0.90,0.75) (0.95,0.8) (1.00,0.75) (1.05,0.65) (1.10,0.65) 
(1.15,0.60) (1.20,0.55) (1.25,0.51) (1.30,0.51) (1.35,0.51) (1.40,0.51) (1.45,0.51) (1.50,0.51) (1.55,0.51) (1.60,0.51) (1.65,0.51) (1.70,0.51) 
(1.75,0.55) (1.80,0.55) (1.85,0.6) (1.90,0.65) (1.95,0.6) (2.00,0.65)
(2.05,0.7) (2.10,0.7) (2.15,0.65) (2.20,0.60) (2.25,0.55) (2.30,0.55)
(2.35,0.55) (2.40,0.6) (2.45,0.6) (2.50,0.65) (2.55,0.65) (2.60,0.6)
(2.65,0.65) (2.70,0.75) (2.75,0.75) (2.80,0.8) (2.85,0.85) (2.90,0.85)
(2.95,0.8) (3.00,0.75) (3.05,0.8) (3.10,0.8) (3.15,0.75) (3.20,0.7)
(3.25,0.65) (3.30,0.65) (3.35,0.7) (3.40,0.65) (3.45,0.7) (3.50,0.7)
(3.55,0.75) (3.60,0.7) (3.65,0.7) (3.70,0.65) (3.75,0.6) (3.80,0.6)
(3.85,0.6) (3.90,0.65) (3.95,0.65) (4.00,0.60) (4.05,0.55) (4.10,0.55)
(4.15,0.60) (4.20,0.65) (4.25,0.60) (4.30,0.65) (4.35,0.70) (4.40,0.75)
(4.45,0.80) (4.50,0.75) (4.55,0.75) (4.60,0.80)
(4.65,0.85) (4.70,0.85)
(4.75,0.8) (4.80,0.85) (4.85,0.8) (4.90,0.85) (4.95,0.9) (5.00,0.9)
(5.05,0.85) (5.10,0.85) (5.15,0.90) (5.20,0.85) (5.25,0.9) (5.30,0.95)
(5.35,0.95) (5.40,0.90) (5.45,0.95) (5.50,0.95) (5.55,0.9) (5.60,0.85)
(5.65,0.9) (5.70,0.85) (5.75,0.80) (5.80,0.80) (5.85,0.85) (5.90,0.80)
(5.95,0.85) (6.00,0.90) (6.05,0.85) (6.10,0.85) (6.15,0.8)
(6.20,0.75)
(6.25,0.75) (6.30,0.80) (6.35,0.85) (6.40,0.85) (6.45,0.85) (6.50,0.85)
(6.55,0.80) (6.60,0.75) (6.65,0.70) (6.70,0.65) (6.75,0.70) (6.80,0.75)
(6.85,0.70) (6.90,0.70) (6.95,0.65) (7.00,0.65) (7.05,0.70) (7.10,0.70)
(7.15,0.65) (7.20,0.65) (7.25,0.60) (7.30,0.55) (7.35,0.51) (7.40,0.51)
(7.45,0.51) (7.50,0.51)
(7.55,0.51) (7.60,0.51) (7.65,0.51) (7.70,0.51) (7.75,0.51) (7.80,0.55)
(7.85,0.6) (7.90,0.60) (7.95,0.55) (8.00,0.60) (8.05,0.650) (8.10,0.650)
(8.15,0.70) (8.20,0.750) (8.25,0.750) (8.30,0.70) (8.35,0.650) (8.40,0.60)
(8.45,0.60) (8.50,0.650) (8.55,0.60) (8.60,0.550) (8.65,0.55) (8.70,0.6)
(8.75,0.65) (8.80,0.65) (8.85,0.55) (8.90,0.55) (8.95,0.51) (9.00,0.51) (9.05,0.51) (9.10,0.51) (9.15,0.51) (9.20,0.51) (9.25,0.51) (9.30,0.51) (9.35,0.51) (9.40,0.51) (9.45,0.51) (9.50,0.51) (9.55,0.51) (9.60,0.51) (9.65,0.51) (9.70,0.51) (9.75,0.51) (9.80,0.51)
(9.85,0.51) (9.90,0.51) (9.95,0.51) (10.00,0.51) (10.05,0.51) (10.10,0.51)
(10.15,0.51) (10.20,0.51) (10.25,0.51) (10.30,0.55) (10.35,0.550) (10.40,0.6)
(10.45,0.65) (10.50,0.65) (10.55,0.6) (10.60,0.65) (10.65,0.70) (10.70,0.75)
(10.75,0.75) (10.80,0.8) (10.85,0.8) (10.90,0.85) (10.95,0.90) (11.00,0.950)
           };

\foreach \x in {0.9,4.2,6.6,9.85} {
        \addplot[dashed, black] coordinates { (\x,0) (\x,.5) };
    }

    \foreach \x in {1.25,1.7,7.35,7.75,8.95,10.25} {
        \addplot[dashed, black] coordinates { (\x,0.5) (\x,1.2) };
    }
    
    \end{axis}

    \node at (-.2, 1) {\scriptsize $A$};

    \node at (3, -.7) {\scriptsize $\mR_1$};
    \node at (10, -.7) {\scriptsize $\mR_2$};

 \node at (0.55, -.7) {\scriptsize $\mT_0$};
    \node at (6.6, -.7) {\scriptsize $\mT_1$};
    \node at (13, -.7) {\scriptsize $\mT_2$};

    \node at (1.8, 2.7) {\scriptsize $\mR_1'$};
\node at (9.23, 2.7) {\scriptsize $\mR_2'$};
    \node at (11.7, 2.7) {\scriptsize $\mR_3'$};

\node at (.6, 2.7) {\scriptsize $\mT_0'$};
    \node at (5.45, 2.7) {\scriptsize $\mT_1'$};
    \node at (10.15, 2.7) {\scriptsize $\mT_2'$};
    \node at (13, 2.7) {\scriptsize $\mT_3'$};

   \node at (0,3.7) {\scriptsize $ \eta_{1}(t)$, $ \mZ (t)$};

\end{tikzpicture}
\caption{Evolution of $ \eta_{1}(t) $ (in black) and $\mZ (t) $ (in red). For
rates satisfying \eqref{eq_supp_bndrates}, there exists a coupling of the two
processes such that $\mT_{n} \leqslant \mT_{n}' $ and $ \mR_{n} \geqslant
\mR_{n}' $, i.e.~the duration of $ \eta_{1}$ spent away from $ A$ is bounded
by the duration of $\mZ $ away from $A$, see \eqref{eq_coupl_supp1}. }
\label{fig_rw_bound} 
\end{figure}
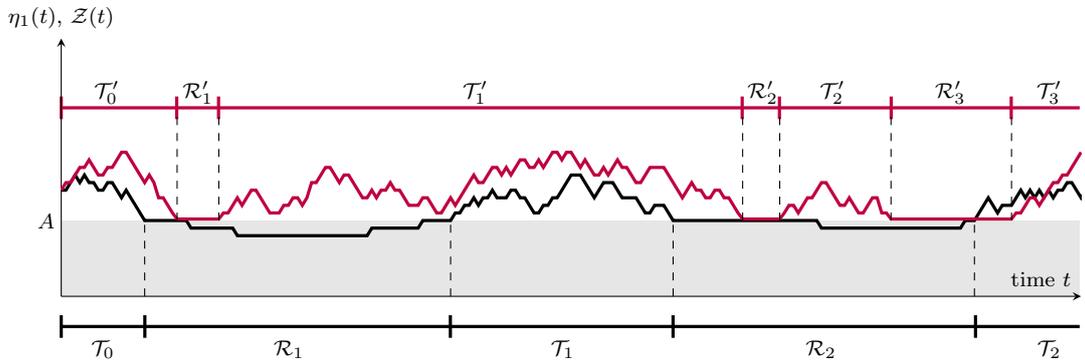

Now,  to conclude Lemma~\ref{lem_first_site} from \eqref{eq_coupl_supp1}, it is necessary that $
c''_{L} \ll c'_{L}$, as well as $1 \ll c'_{L}$, so that the excursions of $
\mZ $ away from $A$ only last over a 
vanishing time--period.
However, this will only be true if $ \gamma_{N} >0$, because the particle
exchange is driven by the mass in the fast phase.
To proceed, we must therefore condition on the event that $
\gamma_{N} $ is uniformly positive, which allows us to construct a random walk
$\mZ$ with rates analogous to \eqref{eq_supp_bndrates}.

\begin{lemma}\label{lem_coupling}
For every $ \delta \in (0,1)$,
there exist constants $ a, c'' >0$, independent of $N$ and~$L$, such that
\begin{equation*}
\begin{aligned}
 \inf_{ \substack{\eta \in \Omega_{L,N}\\ \eta_{1}>A\,, \ \gamma_{N} ( \eta) >
\delta}} c ( \eta)
\geqslant a  L =: c'_{L}
\quad \text{and}\quad 
\limsup_{\tdlim}\sup_{ \substack{\eta \in \Omega_{L,N}\\ \eta_{1} \leqslant A\,, \ \gamma_{N} ( \eta) >
\delta}} c ( \eta)
 \leqslant   c''\,.
\end{aligned}
\end{equation*}
Moreover, for $L$ large enough
\begin{equation*}
\begin{aligned}
\sup_{ \substack{ \eta \in \Omega_{L,N} \\ \eta_{1}>A\,, \  \gamma_{N} ( \eta) > \delta}} p ( \eta) 
\leqslant \tfrac{1}{2}+15 \zeta_{L}=: p'_{L}\,,
\end{aligned}
\end{equation*}
where $ \zeta_{L} =\mO(\tfrac{1}{L})$ was introduced in Assumption~\ref{ass}.
\end{lemma}

A coupling argument then concludes comparison of $ \eta_{1}(t)$ and $ \mZ
(t)$:

\begin{corollary}\label{cor_coupling}
For every initial condition $ \overline{\eta} \in
\Omega_{L,N}$ the inclusion process satisfies
\begin{equation*}
\begin{aligned}
\mathbf{P}_{ \overline{\eta}}\bigg(
\int_{0}^{T} \mathds{1}_{ \eta_{1}(t)>A} \ud t > \varepsilon \bigg\vert 
\inf_{t \in [0,T]} \gamma_{N}( \eta (t)) > \delta
\bigg)
\leqslant 
\mathbf{P}_{ \overline{\mZ}}\bigg(
\int_{0}^{T} \mathds{1}_{ \mZ (t)>A} \ud t > \varepsilon\bigg)\,,
\end{aligned}
\end{equation*}
where $ \mZ $  is the random walk \eqref{eq_def_ratesZ} started from $
\overline{\mZ}:=\max\{
A, \eta_{1}(0) \}$, with rates $c'_{L},c''$ and $p_{L}'$ from Lemma~\ref{lem_coupling}.
\end{corollary}


The proof of Corollary~\ref{cor_coupling}, and Lemma~\ref{lem_coupling}, is presented in Section~\ref{sec_coup}.
Corollary~\ref{cor_coupling} allows us to bound the duration of excursions of $
\eta_{1}(t)$ away
from $\{0,\ldots,A\}$ in terms of a homogeneous continuous--time random walk $\mZ$.
The explicit estimate on transition rates of $\mZ$ allows us to infer
Lemma~\ref{lem_first_site}.

We are now ready to state the proof of Lemma~\ref{lem_first_site}.

\begin{proof}[Proof of Lemma~\ref{lem_first_site}]
Let $ M \in \mathbf{N}$, $ T \in (0, \infty) $ and $ \delta >0$. Moreover, let $ \overline{\eta} = \overline{\eta}^{(L,N)} \in
\Omega_{L,N}$ be a sequence of initial conditions such that $
\overline{\eta}_{1} \leqslant M $ and $ \gamma_{N}( \overline{\eta} ) > \delta$.

We apply Corollary~\ref{cor_coupling}, which yields the upper bound
\begin{equation}\label{e_lem_first_site_supp2}
\begin{aligned}
\mathbf{P}_{ \overline{\eta}}\bigg(
 \int_{0}^{T} \mathds{1}_{\eta_{1} (t)>A} \ud t
\ \bigg \vert
\inf_{t \in [0,T]} \gamma_{N}( \eta (t)) > \delta
\bigg)
\leqslant
\mathbf{P}_{ \overline{M}}\bigg(
\int_{0}^{T} \mathds{1}_{ \mZ (t)>A} \ud t > \varepsilon\bigg)\,,
\end{aligned}
\end{equation}
where $ \mZ = \mZ^{(L)}$ is the continuous--time random walk 
from \eqref{eq_def_ratesZ} starting in $ \overline{M}:= \max \{M,A\}$, with $ c'_{L}, c'' $ and $ p_{L}' $ from
Lemma~\ref{lem_coupling}.

The random walk $ \mZ $ can be interpreted as a renewal process that restarts
every time it hits level $A$.
To this end, we define the stopping times $ \tau_{0} :=0$ and 
\begin{equation}\label{eq_stoptime_Z}
\begin{aligned}
\sigma_{k}' &:= \inf \{ t> \tau_{k-1}' \,:\, \mZ(t) = A\}\\
\tau_{k}' &:= \inf \{ t> \sigma_{k}' \,:\, \mZ(t) = A+1\}
\end{aligned}
\end{equation} 
for $ k \in \mathbf{N}$.
We then partition $[0, \infty)$ into intervals $[ \tau_{k}',
\sigma_{k+1}' )$ where $ \mZ (t) > A $, and intervals $ [ \sigma_{k +1}', \tau_{k
+1}' )$ where $ \mZ (t)=A$.
In view of \eqref{e_lem_first_site_supp2}, it
will suffice to study the lengths of these renewal intervals, because
\begin{equation}\label{e_lem_first_site_supp3}
\begin{aligned}
\mathbf{P}_{ \overline{M}}
\bigg(
\int_{0}^{T} \mathds{1}_{ \mZ (t)>A} \ud t > \varepsilon
\bigg)
\leqslant 
\mathbf{P}_{ \overline{M}}
\bigg(
\sum_{k =0}^{\mK_{T}'} \mT_{k}' 
 > \varepsilon
\bigg)\,,
\end{aligned}
\end{equation}
where $ \mT_{k}' := \sigma_{k+1}' - \tau_{k}' $ denotes the duration of
the $k$--th excursion away from $ A$, see also Figure~\ref{fig_rw_bound}, and 
\begin{equation*}
\begin{aligned}
\mK_{T}' :=
\inf \{ k \in \mathbf{N}\,: \, \tau_{k}' \geqslant T\}
\end{aligned}
\end{equation*}
is an upper bound of the number of renewals within $[0,T]$. 
We also define $ \mR_{k}' :=
\sigma_{k}' - \tau_{k}' $, the time $\mZ$ spends at level $ A$ when hitting
it for the $k$--th time.
Note that $\mR_{k}' = \mathrm{Exp}(c'')$ in law, independently for every $k$.
Together with the fact that $ \tau_{k}' = \sum_{i=1}^{k} ( \mT_{i}' +
\mR_{i}')$, this implies
\begin{equation}\label{eq_estimateK}
\begin{aligned}
\mathbf{P}_{ \overline{M}}( \mK_{T}' >K ) 
=
\mathbf{P}_{ \overline{M}} ( \tau_{K}' <T)
\leqslant 
\mathbf{P}_{ \overline{M}}
\Big( 
\sum_{i=1 }^{K} \mR_{i}' <T
\Big)
\leqslant \frac{c'' T}{K-1}
 \,, \quad \text{for all } K \in \mathbf{N}\,,
\end{aligned}
\end{equation}
where made use of the explicit density of a $ \mathrm{Gamma} (K, c'')$ random
variable:
\begin{equation*}
\begin{aligned}
\mathbf{P}_{ \overline{M}}
\Big( 
\sum_{i=1 }^{K} \mR_{i}' <T
\Big)
=
\int_{0}^{T} 
\frac{(c'')^{K} t^{K-1}}{(K-1)!} 
 e^{- c'' t}
\ud t
\leqslant 
\frac{c'' T}{K-1} 
\int_{0}^{T} 
\frac{(c'')^{K-1} t^{K-2}}{(K-2)!} 
 e^{- c'' t} \ud t 
\leqslant 
\frac{c'' T}{K-1} 
\,.
\end{aligned}
\end{equation*}
On the other hand, by \(\sigma\)--additivity 
\begin{equation}\label{e_lem_first_site_supp4}
\begin{aligned}
\mathbf{P}_{ \overline{M}}\Big(
\sum_{k =0}^{K} \mT_{k}' > \varepsilon
\Big)
\leqslant 
\mathbf{P}_{ \overline{M}}\Big( 
\bigcup_{k =0}^{K}\big\{ \mT_{k}' > \tfrac{ \varepsilon}{K+1}\big\}\Big)
\leqslant 
(K+1) 
\mathbf{P}_{ \overline{M}} \big( \mT_{0}' > \tfrac{ \varepsilon}{K+1}\big)\,.
\end{aligned}
\end{equation}
Thus, from \eqref{e_lem_first_site_supp3} together with \eqref{eq_estimateK}
and \eqref{e_lem_first_site_supp4}, we get the bound 
\begin{equation*}
\begin{aligned}
\mathbf{P}_{ \overline{M}}
\bigg(
\int_{0}^{T} \mathds{1}_{ \mZ (t)>A} \ud t > \varepsilon
\bigg)
& \leqslant 
\mathbf{P}_{ \overline{M}}\Big(
\sum_{k =0}^{K} \mT_{k}' > \varepsilon
\, \Big\vert\,
\mK_{T}' \leqslant K
\Big)
+
\mathbf{P}_{ \overline{M}}( \mK_{T}' >K ) 
\\
&\leqslant 
(K+1) 
\mathbf{P}_{ \overline{M}} \big( \mT_{0}' > \tfrac{ \varepsilon}{K+1}\big)
+
\frac{c'' T}{K-1} \,.
\end{aligned}
\end{equation*}
In particular, if $ \lim_{\tdlim} \mathbf{P}_{ \overline{M}} \big(
\mT_{0}' > \tfrac{ \varepsilon}{K+1}\big)=0$ then 
the proof is finished, since we can take the limit $K \to \infty$ thereafter.

It is convenient to consider the  discrete--time embedded random walk $ S'=
S'^{(L)}$ of $\mZ$.
We define the first hitting time $\mathfrak{n}'=\mathfrak{n}'(L) $ of the
level $A$
\begin{equation*}
\begin{aligned}
\mathfrak{n}' := \inf \{ n \in \mathbf{N}\,: S'_{n}= A\}\,,
\end{aligned}
\end{equation*}
and note that until $ \mathfrak{n}'$, $S'$ is a simple random walk that jumps up with
probability $ p_{L} '$ and down with probability $ 1- p_{L}' $.
We can then bound for any $ \overline{n} \in \mathbf{N}$ 
\begin{equation}\label{e_lem_first_site_supp7}
\begin{aligned}
\mathbf{P}_{ \overline{M}} \big(
\mT_{0}' > \tfrac{ \varepsilon}{K+1}\big)
& \leqslant 
\sum_{n =1}^{ \overline{n}}
\mathbf{P}_{ \overline{M}} \big(
\mT_{0}' > \tfrac{ \varepsilon}{K+1}\big\vert \mathfrak{n}' = n\big)
\mathbf{P}_{ \overline{M}} ( \mathfrak{n}' = n ) 
+ \mathbf{P}_{ \overline{M}}( \mathfrak{n}' > \overline{n})\,.
\end{aligned}
\end{equation}
 Using that $ \mT '_{0}$ decomposes into a sum of exponential random
variables with rate $ c'_{L} = aL$, yields for every $n \in \mathbf{N}$ 
\begin{equation*}
\begin{aligned}
\mathbf{P}_{ \overline{M}} \big(
\mT_{0}' > \tfrac{ \varepsilon}{K+1}\big\vert \mathfrak{n}' = n\big)
=
\mathbf{P} \Big(
\mathrm{Gamma} (n, a L) > \tfrac{ \varepsilon}{K+1}\Big)
=
\mathbf{P} \Big(
\mathrm{Gamma} (n, a ) > \tfrac{ \varepsilon L}{K+1} \Big)\to 0\,,
\end{aligned}
\end{equation*}
as $\tdlim$, where we used the scaling relationship for the Gamma distribution in the last
equality. 
In particular, taking limits of \eqref{e_lem_first_site_supp7}, we see 
\begin{equation*}
\begin{aligned}
\lim_{\tdlim} 
\mathbf{P}_{ \overline{M}} \big(
\mT_{0}' > \tfrac{ \varepsilon}{K+1}\big)
\leqslant 
\lim_{ \overline{n} \to \infty} 
\lim_{\tdlim} 
\mathbf{P}_{ \overline{M}}\big( \mathfrak{n}' (L) > \overline{n}\big) 
=
\lim_{ \overline{n} \to \infty} 
\mathbf{P}_{ \overline{M}}\big( \mathfrak{n}' ( \infty) > \overline{n}\big) 
=0\,,
\end{aligned}
\end{equation*}
where $  \mathfrak{n}' ( \infty)$ denotes the hitting time for the simple symmetric
random walk, which is almost surely finite, see for example \cite[Chapter~XIV;
display~(2.8)]{fellerbook}.
\end{proof}

\subsection{Construction of the excursion--bounding random walk}\label{sec_coup}

In this section, we prove Corollary~\ref{cor_coupling}, namely that the duration of
excursions of $ \eta_{1}(t)$ away from $ \{0, \ldots,A\}$ are bounded by the
ones of a
homogeneous continuous--time random walk $\mZ (t)$.
The idea of the proof has already been outlined in the previous section,
preceding the statement of Lemma~\ref{lem_coupling}.

\subsubsection{Estimates on the transition rates}

First, we prove that there exist suitable rates $ c'_{L}, c''$ and $
p_{L}' $ of a random walk $\mZ$, satisfying the bounds stated in Lemma~\ref{lem_coupling}.
This will allow us to bound the excursion durations of
$ \eta_{1}(t)$ by
the ones of $\mZ(t)$.

\begin{proof}[Proof of Lemma~\ref{lem_coupling}]
To conclude the statement, it suffices to derive lower and upper bounds for 
\begin{equation}\label{eq_coup_supp1}
\begin{aligned}
u_{2}( \eta_{1}) \sum_{i =2}^{L} u_{1}( \eta_{i})
\quad \text{and} \quad 
 u_{1} ( \eta_{1}) \sum_{j =2}^{L}
u_{2} ( \eta_{j})
\,.
\end{aligned}
\end{equation}
We begin with the case of
 $ \eta \in \Omega_{L,N}$ when $ \eta_{1}> A$ and $ \gamma_{N}( \eta ) >
\delta$.
Then, 
\begin{equation}\label{eq_num_ub2}
\begin{aligned}
u_{2}( \eta_{1}) \sum_{i =2}^{L} u_{1}( \eta_{i})
&\leqslant 
( \eta_{1}^{+} + \zeta_{L} )
\Big(
\sum_{i=2}^{L} 
\mathds{1}_{\eta_{i}>A} ( \eta_{i}^{+} + \zeta_{L})
+
\sum_{i =2}^{L}
\mathds{1}_{0< \eta_{i} \leqslant A}
\frac{q_{\eta_{i}}+ \zeta_{L}}{L} 
\Big)\\
& \leqslant 
( \eta_{1}^{+} + \zeta_{L} )
\Big(
\gamma_{N}( \eta) N -  \eta_{1}^{+} + (\#_{> A} \eta -1) \zeta_{L}
+
\sum_{i =2}^{L}
\mathds{1}_{0< \eta_{i} \leqslant A}
\frac{q_{\eta_{i}}}{L}
+ \zeta_{L} \Big)\\
& \leqslant 
( \eta_{1}^{+} + \zeta_{L} )
\Big(
\big(  \gamma_{N} ( \eta) N - \eta_{1}^{+}\big)(1 +  \zeta_{L})
+
\sum_{i =2}^{L}
\mathds{1}_{0< \eta_{i} \leqslant A}
\frac{q_{\eta_{i}}}{L}
+ \zeta_{L}\Big)
\,.
\end{aligned}
\end{equation}
where $ \eta_{i}^{+}= ( \eta_{i}- A)_{+}$
and we used the crude estimate
$ \#_{>A} \eta-1 \leqslant \gamma_{N}( \eta) N - \eta_{1}^{+}   $  in the last step.
Similarly, we derive the lower bound 
\begin{equation}\label{eq_num_lb1}
\begin{aligned}
u_{2}( \eta_{1}) \sum_{i =2}^{L} u_{1}( \eta_{i})
\geqslant 
( \eta_{1}^{+} -\zeta_{L} )
 \Big(
(  \gamma_{N} ( \eta) N - \eta_{1}^{+}) (1 -  \zeta_{L})
+
\sum_{i =2}^{L}
\mathds{1}_{0< \eta_{i} \leqslant A}
\frac{q_{\eta_{i}}}{L}
- \zeta_{L} 
\Big)\,,
\end{aligned}
\end{equation}
as well as a lower bound for the second term in \eqref{eq_coup_supp1}:
\begin{equation}\label{eq_den2_lb1}
\begin{aligned}
u_{1} ( \eta_{1}) \sum_{j =2}^{L}
u_{2} ( \eta_{j})
& \geqslant 
( \eta_{1}^{+} - \zeta_{L} )
\Big(
\big(  \gamma_{N} ( \eta) N - \eta_{1}^{+}\big)(1 -  \zeta_{L})
+
\sum_{i =2}^{L}
\mathds{1}_{0 \leqslant  \eta_{i} \leqslant A}
\frac{r_{\eta_{i}}}{L}
- \zeta_{L} 
\Big)\\
& \geqslant 
( \eta_{1}^{+} - \zeta_{L} )
\Big(
\big(  \gamma_{N} ( \eta) N - \eta_{1}^{+}\big) (1 -  \zeta_{L})
+
\sum_{i =2}^{L}
\mathds{1}_{0 <  \eta_{i} \leqslant A}
\frac{q_{\eta_{i}}}{L}
- \zeta_{L} 
\Big)\,,
\end{aligned}
\end{equation}
where we used  $ r_{k} \geqslant q_{k}$, $k=1,\ldots, A$, and $ q_{0}=0 $ in
the second
inequality.

Combining \eqref{eq_num_ub2} and \eqref{eq_num_lb1}, \eqref{eq_den2_lb1} yields
\begin{equation}\label{e_p_ub1}
\begin{aligned}
p ( \eta) 
&\leqslant 
\frac{1}{2} \, \frac{ \eta_{1}^{+}+ \zeta_{L}}{ \eta_{1}^{+}- \zeta_{L}} 
\, \cdot\, 
\frac{(  \gamma_{N} ( \eta) N - \eta_{1}^{+})(1 +  \zeta_{L})
+
\sum_{i =2}^{L}
\mathds{1}_{0< \eta_{i} \leqslant A}
\frac{q_{\eta_{i}}}{L}
+ \zeta_{L}}{(  \gamma_{N} ( \eta) N - \eta_{1}^{+}) (1 -  \zeta_{L})
+
\sum_{i =2}^{L}
\mathds{1}_{0 <  \eta_{i} \leqslant A}
\frac{q_{\eta_{i}}}{L}
- \zeta_{L}} \\
& = 
\frac{1}{2} \Big( 1 + \frac{ 2 \zeta_{L}}{ \eta_{1}^{+}- \zeta_{L}} \Big)
\bigg( 1 +
\frac{ 2 \zeta_{L} (  \gamma_{N} ( \eta) N - \eta_{1}^{+} +1)}{( \gamma_{N} (
\eta) N - \eta_{1}^{+}) ( 1- \zeta_{L}) +\sum_{i =2}^{L}
\mathds{1}_{0 <  \eta_{i} \leqslant A}
\frac{q_{\eta_{i}}}{L}
- \zeta_{L}} 
\bigg)\,.
\end{aligned}
\end{equation}
For $ L$ large enough, we have $ \zeta_{L} \leqslant \tfrac{1}{3}$, such
that the first parenthesis on the right--hand side of
\eqref{e_p_ub1} is upper bounded by $1 + 3 \zeta_{L}$, because $
\eta_{1}^{+} \geqslant  1$.
The fraction within the second parenthesis can be estimated
as follows: If $ \eta_{1}^{+}<  {\gamma}_{N} N $, then
\begin{equation*}
\begin{aligned}
& \frac{
2 \zeta_{L}(
\gamma_{N}( \eta ) N - \eta_{ 1}^{+} + 1)
}{(\gamma_{N}( \eta) N- \eta_{1}^{+})  (1 -  \zeta_{L})
+
\sum_{i =2}^{L}
\mathds{1}_{0< \eta_{i} \leqslant A}
\frac{q_{\eta_{i}}}{L}
- \zeta_{L}} 
\leqslant 
2 \zeta_{L}
 \frac{
\gamma_{N}( \eta ) N - \eta_{ 1}^{+} + 1
}{
\tfrac{2}{3}(\gamma_{N}( \eta) N- \eta_{1}^{+}) - \tfrac{1}{3}}
\leqslant 12 \zeta_{L}
\end{aligned}
\end{equation*}
where we  used again  $ \zeta_{L}  \leqslant \tfrac{1}{3}$ and dropped the sum
over $ q_{ k}$'s in the denominator.
On the other hand, if $ \eta_{1}^{+}=  \gamma_{N} N $ (i.e.~the only site with
fast particles is $ \eta_{1} $), then 
\begin{equation*}
\begin{aligned}
  \frac{
2 \zeta_{L}(
\gamma_{N}( \eta ) N - \eta_{ 1}^{+} + 1)
}{(\gamma_{N}( \eta) N- \eta_{1}^{+})  (1 -  \zeta_{L})
+
\sum_{i =2}^{L}
\mathds{1}_{0< \eta_{i} \leqslant A}
\frac{q_{\eta_{i}}}{L}
- \zeta_{L}} 
\leqslant 
 \frac{
2 \zeta_{L} 
}{
\tfrac{L-1}{L}
\min_{k=1, \ldots, A} q_{k}
- \zeta_{L}} 
\leqslant 12 \zeta_{L} \,,
\end{aligned}
\end{equation*}
for $ L$ large enough.
Thus, we have 
\begin{equation*}\label{eq_est_p_eta}
\begin{aligned}
p ( \eta) 
\leqslant 
\frac{1}{ 2} ( 1+ 3 \zeta_{L}) ( 1+ 12 \zeta_{L})
\leqslant \frac{1}{2} ( 1+ 27 \zeta_{L})
\leqslant \frac{1}{2} + 15 \zeta_{L}\,,
\end{aligned}
\end{equation*}
uniformly over all $ \eta$ with $
\eta_{1}> A$,  for large enough $L$.

Similarly, by combining \eqref{eq_num_lb1} and \eqref{eq_den2_lb1}, we get 
\begin{equation}\label{eq_c_lb1}
\begin{aligned}
c ( \eta) 
&\geqslant 
 {\eta}_{1}^{+} \Big(
\big(  \gamma_{N} ( \eta) N - \eta_{1}^{+}\big) (1 -  \zeta_{L})
+
\sum_{i =2}^{L}
\mathds{1}_{0< \eta_{i} \leqslant A}
\frac{q_{\eta_{i}}}{L}
- \zeta_{L} 
\Big)
\geqslant 
\frac{1}{4} ( \underline{q} \vee 1) \delta N 
\,,
\end{aligned}
\end{equation}
for $N,L$ large enough, which can be shown by considering the same two cases as
for $ p ( \eta) $ above. 
Thus, $ c ( \eta) \geqslant
c'_{L}$ (uniformly over all $ \eta$ such that $
\eta_{1}> A$ and $ \gamma_{N}( \eta) > \delta$) where $ c'_{L} := aL $ with $ a = \tfrac{N}{4L}  ( \underline{q} \vee 1)
\delta$.\\

Finally, we show the upper bound for $ c ( \eta) \leqslant  c ''$. Let $
\eta \in \Omega_{L,N}$ such that $ \eta_{1} \leqslant  A$, then for all $N,L$ large enough
\begin{equation*}
\begin{aligned}
c ( \eta) 
&\leqslant 
\frac{ \overline{r} + \zeta_{L}}{L}
\sum_{i =2}^{L} u_{1}( \eta_{i})
+
\frac{ \overline{q} + \zeta_{L}}{L} 
\sum_{j =2}^{L} u_{2}( \eta_{i})\\
& \leqslant 
\frac{2}{L} (\overline{r} + \zeta_{L}) 
\big(
\gamma_{N}( \eta ) N( 1 + \zeta_{L})
+
\overline{r}
+ \zeta_{L}
\big) 
\leqslant 4 \rho ( \overline{r} + 1)=: c''
\end{aligned}
\end{equation*} 
where we used $ \gamma_{N} \leqslant 1 $ and the form of the rates $u_{1}, u_{2}$ from
Assumption~\ref{ass}, in particular, $ q_{k} \leqslant r_{k} \leqslant
\overline{r}$ for all $k=0,
\ldots, A$. This finishes the proof.
\end{proof}

\subsubsection{The coupling argument}

Finally, we conclude Corollary~\ref{cor_coupling} from Lemma~\ref{lem_coupling}
using an explicit coupling of $ \eta_{1}(t)$ and $ \mZ (t)$.
Similar as in \eqref{eq_stoptime_Z}, we define stopping times $ \tau_{0}:=0$
and
\begin{equation*}
\begin{aligned}
\sigma_{k} &:= \inf \{ t> \tau_{k-1} \,:\, \eta_{1}(t) \leqslant  A\}\,,\\
\tau_{k} &:= \inf \{ t> \sigma_{k} \,:\, \eta_{1}(t) =  A+1\}\,.
\end{aligned}
\end{equation*}
 Again this gives rise to the duration $\mT_{k}:=
\sigma_{k +1}- \tau_{k}$ of the $k$--th excursion above level $A$, and duration
$ \mR_{k}:= \sigma_{k}- \tau_{k}$ spent in the slow phase beforehand.
See also Figure~\ref{fig_rw_bound}. 
The goal of this section is to prove that there exists a coupling between $
\eta(t)$ and $ \mZ(t)$, such that $ \mT_{k} \leqslant \mT_{k}' $ and $
\mR_{k} \geqslant \mR_{k}' $, for all $k$, conditioned on the event that $ \inf_{t \in[0,T]} \gamma_{N}( \eta (t))
> \delta$.
Once this has been established, Corollary~\ref{cor_coupling} follows.\\

The idea is to couple the inclusion process $ \eta (t)$ with $ \mZ (t)$ on
each of the ``renewal intervals'', $[ \tau_{k}, \tau_{k +1}) $ and $ [
\tau'_{k}, \tau_{k +1}' )$.
First, we will only focus on the case $k \geqslant 1 $, because the initial renewal
interval ($k=0$) differs in terms of its initial condition. This will be treated as a
special case at the end of the section.

Let $( \xi (m))_{m \in \mathbf{N}}\subset \Omega_{L,N}$ be the
embedded discrete--time evolution $ \eta (t) $ on $[ \tau_{k}, \sigma_{k +1})$,
when $k \geqslant 1$, which satisfies $ \xi_{1}(0)=A+1$.
Because not every performed particle jump yields a change of $ \xi_{1}$, we define the
subsequence $( m_{n})_{n}$ of all ``prejumps'' that satisfy 
\begin{equation*}
\begin{aligned}
\xi_{1}( m_{n}) \neq \xi_{1}( m_{n}-1 )\,.
\end{aligned}
\end{equation*}
This yields a random walk $ S_{n}:= \xi_{1}(m_{n})$ on $ \mathbf{N}$,
representing the evolution of the discrete--time embedded random walk of $
\eta_{1}(t)$, which terminates once it reaches level $ A$ after $ \mathfrak{n} \in
\mathbf{N} \cup \{ \infty\}$ steps.

Now, define the random walk $ S' = S'^{(k)}$ as follows: First, $S_{0} ' =
A+1$.
For $ n \in \mathbf{N}$: 
\begin{enumerate}
\item If $ S $ performs an upwards jump, i.e.~$ S_{n} = S_{n -1}+1 $, then so
does $ S' $ and we set $ S'_{n} = S'_{n -1}+1$.

\item If $ S $ performs a downwards jump, i.e.~$ S_{n} = S_{n -1}-1 $, then 
 \begin{equation*}
\begin{aligned}
S'_{n} = S'_{n -1} 
+
\begin{cases}
+1\,,& \text{ with probability }  p'(\xi (m_{n-1}) )  \\
-1\,, & \text{ otherwise}\,, 
\end{cases}
\end{aligned}
\end{equation*}
independently at every step, where 
\begin{equation*}
\begin{aligned}
p' ( \xi) 
:= 
\frac{ \frac{1}{2} - p ( \xi) + 15 \zeta_{L}}{1- p ( \xi)} \,,
\end{aligned}
\end{equation*}
with $p( \xi) $ defined in \eqref{eq_def_p}.
Note that $ p' \in [0,1]$ as a consequence of Lemma~\ref{lem_coupling}.

\end{enumerate}

By construction, we have 
$ S_{n} \leqslant S '_{n}$ for all $n \leqslant \mathfrak{n}$, see also
Figure~\ref{fig_couple} below.
 Moreover, $S'$ evolves according to a simple random walk with probability
$p'_{L}= \tfrac{1}{2} +15
\zeta_{L}$ of jumping upwards, until hitting $ A$ after a total of
$\mathfrak{n}' \geqslant \mathfrak{n}$ steps.
Therefore, in combination with \eqref{eq_def_ratesZ}, the above coupling yields that $
(S'_{n})_{n =0, \ldots, \mathfrak{n}'}$
is a discrete--time embedded random walk of the continuous time process $ \mZ
(t)$. \\

Next, we construct exponential waiting times to lift the discrete--time random
walk $ S'$ to a suitably coupled continuous--time random walk $\mZ$ with law
\eqref{eq_def_ratesZ}.
Assume we have constructed $ \mZ (t)$ already on $[0, \tau_{k}' )$. 
On the interval $[ \tau_{k}, \sigma_{k +1})$, the process $ \eta_{1}(t)$ jumps according to an
inhomogeneous Poisson process $\Lambda$ with intensity $ c ( \eta (t))$, see
\eqref{eq_c_totalrate}. 
Recall from Lemma~\ref{lem_coupling}, that $ c ( \eta (t)) \geqslant c_{L}' $
for every $ t \in [ \tau_{k}, \sigma_{k+1})$. 
Hence, we can couple the waiting times $\Lambda $ to a
 Poisson process $\Lambda' = \Lambda'^{(k)}$ with constant rate $ c_{L}' $ such that
\begin{equation*}
\begin{aligned}
\Lambda'(t) \leqslant \Lambda (t+ \tau_{k})- \Lambda ( \tau_{k}) \,, \quad \forall t \in [0,
\sigma_{k+1} - \tau_{k})\,.
\end{aligned}
\end{equation*}
Then  define
$ \mZ (t) $ on the interval $[ \tau_{k}',
\sigma_{k+1}']$ as 
\begin{equation*}
\begin{aligned}
\mZ (t) := 
S'_{\Lambda' (t - \tau_{k}')}\,,
\end{aligned}
\end{equation*}
which by construction has law \eqref{eq_def_ratesZ} and satisfies
\begin{equation*}
\begin{aligned}
 \mT_{k}'=
 \inf \{ t >0 \,:\, \Lambda'(t) = \mathfrak{n}'\} \geqslant 
 \inf \{ t > \tau_{k} \,:\, \Lambda(t)- \Lambda ( \tau_{k}) = \mathfrak{n}\}  
=\mT_{k}\,.
\end{aligned}
\end{equation*}
Moreover, we can define  $\sigma_{k +1}' := \tau_{k}' + \mT_{k}$.

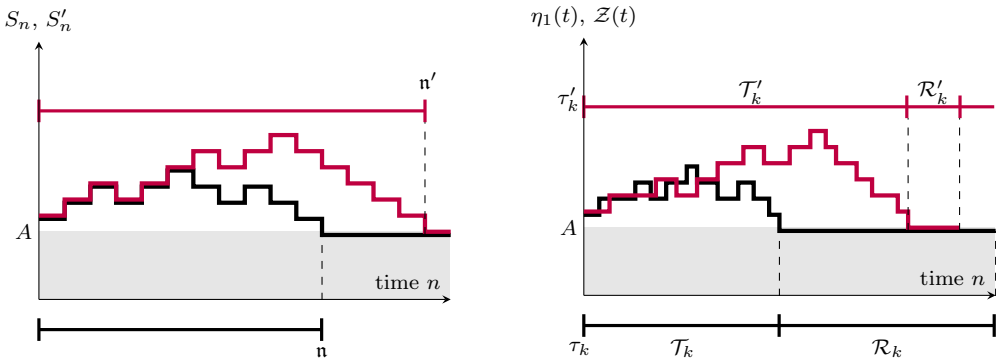
\begin{figure}[H]
\centering
\begin{subfigure}[t]{0.5\textwidth}
        \centering
        \begin{tikzpicture}
\fill[gray!20] (0,0) rectangle (5.4,.9); 
\draw[very thick] (0,-0.4) -- (3.72,-.4);
\draw[very thick] (0,-0.55) -- (0,-0.25); 
\draw[very thick] (3.72,-0.55) -- (3.72,-0.25);

\draw[very thick, purple] (0,2.5) -- (5.08,2.5);
\draw[very thick, purple] (0,2.65) -- (0,2.35); 
\draw[very thick, purple] (5.08,2.65) -- (5.08,2.35);

    \begin{axis}[
        width=7cm, height=5cm,
        axis lines=middle,
        xlabel={\scriptsize time $n$},
        xmin=0, xmax=4,
        ymin=0.3, ymax=1.1,
        xtick=\empty, ytick=\empty,
    ]
    
    \addplot[black, ultra thick] coordinates {
(0.00,0.55)(0.25,0.55)
(0.25,0.6)(0.5,0.6)
(0.5,0.65)(0.75,0.65)
(0.75,0.6)(1,0.6)
(1,0.65)(1.25,0.65)
(1.25,0.7)(1.5,0.7)
(1.5,0.65)(1.75,0.65)
(1.75,0.6)(2,0.6)
(2,0.65)(2.25,0.65)
(2.25,0.60)(2.5,0.60)
(2.5,0.55)(2.75,0.55)
(2.75,0.50) (4,0.50)};
    
    \addplot[purple, ultra thick] coordinates {
(0.00,0.56)(0.25,0.56)
(0.25,0.61)(0.5,0.61)
(0.5,0.66)(0.75,0.66)
(0.75,0.61)(1,0.61)
(1,0.66)(1.25,0.66)
(1.25,0.71)(1.5,0.71)
(1.5,0.76)(1.75,0.76)
(1.75,0.71)(2,0.71)
(2,0.76)(2.25,0.76)
(2.25,0.81)(2.5,0.81)
(2.5,0.76)(2.75,0.76)
(2.75,0.71)(3,0.71)
(3,0.66)(3.25,0.66)
(3.25,0.61)(3.5,0.61)
(3.5,0.56)(3.75,0.56)
(3.75,0.51)(4,0.51)
};

\foreach \x in {2.75} {
        \addplot[dashed, black] coordinates { (\x,0) (\x,.5) };
    }

    \foreach \x in {3.75} {
        \addplot[dashed, black] coordinates { (\x,0.5) (\x,.85) };
    }
    
    \end{axis}

    \node at (-.2, .9) {\scriptsize $A$};

    \node at (3.72, -.7) {\scriptsize $\mathfrak{n}$};
\node at (5.125, 2.9) {\scriptsize $\mathfrak{n}'$};

   \node at (0,3.7) {\scriptsize $ S_{n}$, $ S_{n}'$};

\end{tikzpicture}
    \end{subfigure}%
    ~ 
    \begin{subfigure}[t]{0.5\textwidth}
	\begin{tikzpicture}
\fill[gray!20] (0,0) rectangle (5.4,.9); 
\draw[very thick] (0,-0.4) -- (5.4,-.4);
\draw[very thick] (0,-0.55) -- (0,-0.25); 
\draw[very thick] (5.4,-0.55) -- (5.4,-0.25); 
\draw[very thick] (2.57,-0.55) -- (2.57,-0.25); 

\draw[very thick, purple] (0,2.5) -- (5.4,2.5);
\draw[very thick, purple] (0,2.65) -- (0,2.35); 
\draw[very thick, purple] (4.25,2.65) -- (4.25,2.35); 
\draw[very thick, purple] (4.95,2.65) -- (4.95,2.35);

    \begin{axis}[
        width=7cm, height=5cm,
        axis lines=middle,
        xlabel={\scriptsize time $n$},
        xmin=0, xmax=4,
        ymin=0.3, ymax=1.1,
        xtick=\empty, ytick=\empty,
    ]
    
    \addplot[black, ultra thick] coordinates {
(0.00,0.55)(0.15,0.55)
(0.15,0.6)(0.5,0.6)
(0.5,0.65)(0.65,0.65)
(0.65,0.6)(.8,0.6)
(.8,0.65)(1,0.65)
(1,0.7)(1.1,0.7)
(1.1,0.65)(1.3,0.65)
(1.3,0.6)(1.5,0.6)
(1.5,0.65)(1.65,0.65)
(1.65,0.60)(1.8,0.60)
(1.8,0.55)(1.9,0.55)
(1.9,0.5)(4,0.5)

};
    
    \addplot[purple, ultra thick] coordinates {
(0.00,0.56)(0.25,0.56)
(0.25,0.61)(0.7,0.61)
(0.7,0.66)(0.9,0.66)
(0.9,0.61)(1.1,0.61)
(1.1,0.66)(1.3,0.66)
(1.3,0.71)(1.55,0.71)
(1.55,0.76)(1.75,0.76)
(1.75,0.71)(2,0.71)
(2,0.76)(2.2,0.76)
(2.2,0.81)(2.35,0.81)
(2.35,0.76)(2.45,0.76)
(2.45,0.71)(2.65,0.71)
(2.65,0.66)(2.8,0.66)
(2.8,0.61)(3.05,0.61)
(3.05,0.56)(3.15,0.56)
(3.15,0.51)(3.65,0.51)

         };

\foreach \x in {1.9,4} {
        \addplot[dashed, black] coordinates { (\x,0) (\x,.5) };
    }

    \foreach \x in {3.65,3.15} {
        \addplot[dashed, black] coordinates { (\x,0.5) (\x,.85) };
    }
    
    \end{axis}

    \node at (-.2, .9) {\scriptsize $A$};

    \node at (4, -.7) {\scriptsize $\mR_k$};
    \node at (1.3, -.7) {\scriptsize $\mT_k$};
\node at (4.6, 2.7) {\scriptsize $\mR_k'$};
    \node at (2.2, 2.7) {\scriptsize $\mT_k'$};

    \node at (-.2, 2.6) {\scriptsize $\tau_k'$};
    \node at (-.05, -.7) {\scriptsize $\tau_k$};

   \node at (0,3.7) {\scriptsize $ \eta_{1}(t)$, $ \mZ(t)$};

\end{tikzpicture}        
    \end{subfigure}

\caption{The left figure shows the coupling of the embedded discrete--time
random walks $ S $ (in black) and $ S'$ (in red). After coupling also the
exponential waiting times, we see that excursions of $ \mZ(t)$ (constructed from
$S'$) away from $ A$ take longer than those of $ \eta_{1}(t)$ (in black).}
\label{fig_couple}
\end{figure}

For the intervals $ [ \sigma_{k+1}' , \tau_{k+1}')$, we simply couple a family
of exponential random variables $ ( \mR_{k}')_{k}$ with rate $ c''$ to  $ (
\mR_{k})_{k}$ such that $ \mR_{k} \geqslant \mR'_{k}$. Again, this is possible
because of Lemma~\ref{lem_coupling}. 
We then set $ \tau_{k +1} ' = \sigma_{k +1}' + \mR'_{k+1} $,
together with $ \mZ (t) = A$ on $[ \sigma_{k +1}' , \tau_{k +1}') $.\\

The construction of $\mZ$ on the initial interval $[0,
\sigma_{1}' )$ works
precisely as above, however, the embedded random walks start in $ \mZ(0)=
\eta_{1}( 0)$ instead of $ A+1$, if $ \eta_{1}(0) >A$.
On the other hand, if $ \eta_{1}(0) \leqslant  A $, we set $ \mZ(0)=A$ and we begin
with the
construction of $ \mZ $ on $[0, \tau_{1}')$.

\section{A bound on generator differences}\label{sec_conv_gen}

In this section, we will prove Lemma~\ref{lem_gen_est} by directly comparing
the action of the generators $\mL$ and $\mathfrak{L}_{L,N}$.
The proof is standard but tedious, due to the
two--phase 
interactions in $\mathfrak{L}_{L,N}$ \eqref{e_IP_gen}.
It employs approximation through
Taylor's theorem, specifically, we shall make extensive use of the fact that for
every $f \in \mD$ 
\begin{equation}\label{e_taylor}
\begin{aligned}
f (x',y') - f(x,y)
= 
&\sum_{i =1}^{ \infty} (x_{i}'- x_{i}) \partial_{x_{i}}f(x,y)
+ \sum_{k =0}^{A-1} (y_{k}' - y_{k}) \partial_{y_{k}} f(x,y)\\
&+ \frac{1}{2}  \sum_{i,j =1}^{\infty} (x_{i}'- x_{i})  (x_{j}'- x_{j}) \partial_{x_{i}}
\partial_{x_{j}}f(x,y)\\
&+ \frac{1}{2}  \sum_{k,l =0}^{A-1} (y_{k}' - y_{k}) (y_{l}' - y_{l})  \partial_{y_{k}}
\partial_{y_{l}} f(x,y)\\
& + \frac{1}{2}  \sum_{i =1}^{\infty} \sum_{k =0}^{A-1} (x_{i}'- x_{i}) (y_{k}' - y_{k}) 
 \partial_{x_{i}} \partial_{y_{k}}  f(x,y)+ R^{f}_{x,y}(x',y')\,,
 \end{aligned}
\end{equation}
where $ x $ and $ x'$ only differ by finitely many entries. The error term $
R_{f}$ is uniformly estimated by
\begin{equation}\label{e_tayerr_est}
\begin{aligned}
| R^{f}_{x,y}(x',y')| \leqslant C_{f}\big( \max_{ \substack{ i=1,2,\ldots \\ k =0,
\ldots, A-1}}
\{
|x_{i}- x_{i}' | , | y_{k}- y_{k}'|
\}\big)^{3}\,,
\end{aligned}
\end{equation} 
with $ C_{f}$ being a finite constant only depending on (third order
derivatives of) $ f$.

\begin{proof}[Proof of Lemma~\ref{lem_gen_est}]
Let $ \eta \in \Omega_{L,N}$.
 We split the generator $\mathfrak{L}_{L,N}$ \eqref{e_IP_gen} into four parts, according to the four
possible interactions.\\

\noindent \hbox{\tiny $\blacksquare$} \underline{Slow to slow phase}, i.e.~$ \eta_{i} \leqslant A
$ and $ \eta_{j} \leqslant A$.
First, recall the embedding $\embd $ \eqref{e_def_embd_direct} of particle configurations into $
\overline{\nabla} \times S_{Y}$.
Then by \eqref{e_taylor}
\begin{equation}\label{e_f_taylor}
\begin{aligned}
\big|f\circ \embd  ( \eta^{i,j}) - f\circ \embd  (\eta)\big|
\leqslant 
C_{f}  \Big(\frac{1}{L} + \frac{1}{N} \Big)\,, 
\qquad i,j=1, \ldots, L \,, 
\end{aligned}
\end{equation}
where we absorbed first and second-order derivatives of $f$ into $
C_{f}$.
Therefore, the contribution to the generator of interactions within the slow
phase is estimated by
\begin{equation}\label{e_b2b_supp1}
\begin{aligned}
&\Big|\sum_{i,j=1}^{L} u_{1}( \eta_{i}) u_{2}( \eta_{j})
\mathds{1}_{ \eta_{i} \leqslant  A,\, \eta_{j} \leqslant A}
[f\circ \embd  ( \eta^{i,j}) - f\circ \embd  (\eta)]\Big|\\
& \leqslant  \frac{\overline{u}^{2}}{L^{2}}
\sum_{i,j=1}^{L} 
|f\circ \embd  ( \eta^{i,j}) - f\circ \embd  (\eta)|
 \leqslant  C_{f} \overline{u}^{2} \, \Big(\frac{1}{L} + \frac{1}{N} \Big) = o_{f}(1) \,,
\end{aligned}
\end{equation}
with $ \overline{u} := \sup_{L \in \NN}  \max_{k=0, \ldots, A} \{ 
u_{1,L}(k)\, L,
u_{2,L}(k)\, L  \}< \infty$ by Assumption~\ref{ass}.
Hence, the right--hand side of \eqref{e_b2b_supp1} vanishes which yields that
interactions within the slow phase are not captured on the considered time scale.\\

\noindent \hbox{\tiny $\blacksquare$} \underline{Slow to fast phase}, i.e.~$ \eta_{i} \leqslant
A$ and $ \eta_{j} > A$. 
We start by noting that in this case 
 \eqref{e_taylor} yields 
\begin{equation*}
\begin{aligned}
&f\circ \embd  ( \eta^{i,j}) - f\circ \embd
(\eta)
= 
 \frac{1}{N} \partial_{x_{j}}f\circ \embd  (\eta)
+ \frac{1}{L} 
\partial_{y_{ \eta_{i}-1 }} f\circ \embd  ( \eta)
-\frac{1}{L} 
\partial_{y_{ \eta_{i}}} f\circ \embd  (\eta)
+ o_{f}\big( \tfrac{1}{L} \big)\,.
\end{aligned}
\end{equation*}
Thus, we can rewrite the corresponding part of the generator
$\mathfrak{L}_{L,N}$ in terms of 
\begin{equation}\label{e_b2c_supp3}
\begin{aligned}
&\sum_{i,j=1}^{L}
 u_{1}( \eta_{i}) u_{2}( \eta_{j})
\mathds{1}_{ \eta_{i} \leqslant  A,\, \eta_{j} > A}
[f\circ \embd  (\eta^{i,j}) - f\circ \embd
(\eta)]\\
& = 
\frac{1}{L} 
\sum_{i,j=1}^{L} 
 u_{1}( \eta_{i}) u_{2}( \eta_{j})
\mathds{1}_{ \eta_{i} \leqslant  A,\, \eta_{j} > A}
[ 
\partial_{y_{ \eta_{i}-1 }} f
-\partial_{y_{ \eta_{i}}} f]\circ \embd  (\eta)\\
& \quad 
+
 \frac{1}{N}
\sum_{i,j=1}^{L} 
 u_{1}( \eta_{i}) u_{2}( \eta_{j})
\mathds{1}_{ \eta_{i} \leqslant  A,\, \eta_{j} > A}
 \partial_{x_{j}}f\circ \embd  (\eta)
+o_{f}(1)\\
& = 
\frac{N}{L} 
\Big(\sum_{j=1}^{L} 
\mathds{1}_{\eta_{j} > A}
\frac{u_{2}( \eta_{j})}{N} \Big)
\sum_{k=1}^{A} 
 u_{1}( k)\, L
\frac{\#_{k}\eta}{L} 
[ 
\partial_{y_{ k-1 }} f-\partial_{y_{ k}} f]\circ \embd  (\eta)\\
& \quad 
+ 
\Big( 
\sum_{k =1}^{A } u_{1}(k)  L\,  
\frac{\#_{k} \eta}{L} 
\Big)
\sum_{j=1}^{L} 
\mathds{1}_{\eta_{j} > A}
 \frac{u_{2}( \eta_{j})}{N}
 \partial_{x_{j}}f\circ \embd  (\eta)
+o_{f}(1)
\end{aligned}
\end{equation}
where we performed the change of index $ k = \eta_{i}$ in the last step,
before rearranging terms. 
Note that $\partial_{y_{A}}f =0$ since $f$ does not
depend explicitly on a $ y_{A}$ component.
Finally, using Assumption~\ref{ass}, we have that 
\begin{equation}\label{e_b2c_supp4}
\begin{aligned}
&\sum_{i,j=1}^{L}
 u_{1}( \eta_{i}) u_{2}( \eta_{j})
\mathds{1}_{ \eta_{i} \leqslant  A,\, \eta_{j} > A}
[f\circ \embd  (\eta^{i,j}) - f\circ \embd
(\eta)]\\
& = 
\frac{N}{L} 
\gamma_{N}( \eta) \sum_{k=1}^{A} 
 q_{k}\frac{\#_{k}\eta}{L} 
[ 
\partial_{y_{ k-1 }} f-\partial_{y_{ k}} f]\circ \embd  (\eta)\\
& \quad 
+ 
\Big( 
\sum_{k =1}^{A } q_{k}\,  
\frac{\#_{k} \eta}{L} 
\Big)
\sum_{j=1}^{L} 
 \frac{( \eta_{j}-A )_{+}}{N}
 \partial_{x_{j}}f\circ \embd  (\eta)
+o_{f}(1)\,,
\end{aligned}
\end{equation}
where we  used 
in particular that
\begin{equation}\label{eq_gamma_cf_rates}
\begin{aligned}
\bigg|
 \sum_{i =1}^{L}
\frac{ u_{\iota}( \eta_{i})}{N}\mathds{1}_{ \eta_{i}>A}
- \gamma_{N}( \eta) 
\bigg|
=
\bigg|
 \sum_{i =1}^{L}
\frac{ u_{\iota}( \eta_{i})- ( \eta_{i}-A ) }{N}\mathds{1}_{ \eta_{i}>A}
\bigg|
\leqslant \frac{ \zeta_{L} \, \#_{>A} \eta}{N} \to 0 \,,
\end{aligned}
\end{equation}
as $\tdlim $, for $\iota =1,2$. 
\\

\noindent \hbox{\tiny $\blacksquare$} \underline{Fast to slow phase}, i.e.~$ \eta_{i}>A $ and $
\eta_{j} \leqslant  A$. Analogously to the
previous step,
 \eqref{e_taylor} yields 
\begin{equation*}
\begin{aligned}
f\circ \embd  ( \eta^{i,j}) - f\circ \embd
(\eta)
&= 
- \frac{1}{N} \partial_{x_{i}}f\circ \embd  (\eta)
+ \frac{1}{N} \partial_{x_{j}} f \circ \embd ( \eta) 
\mathds{1}_{\eta_{j}=A}\\
&\quad + \frac{1}{L} 
\partial_{y_{ \eta_{j}+1 }} f\circ \embd  ( \eta)
-\frac{1}{L} 
\partial_{y_{ \eta_{j}}} f\circ \embd  (\eta)
+ o_{f}\big( \tfrac{1}{L} \big)\,.
\end{aligned}
\end{equation*}
Therefore,
\begin{equation}\label{e_c2b_supp4}
\begin{aligned}
&\sum_{i,j=1}^{L}
 u_{1}( \eta_{i}) u_{2}( \eta_{j})
\mathds{1}_{ \eta_{i} >  A,\, \eta_{j} \leqslant  A}
[f\circ \embd  (\eta^{i,j}) - f\circ \embd
(\eta)]\\
&=
\frac{N}{L} 
\gamma_{N}( \eta) 
\sum_{k=0}^{A-1} 
r_{k}
\frac{\#_{k}\eta}{L} 
[ 
\partial_{y_{ k+1 }} f
-\partial_{y_{ k}} f]\circ \embd  (\eta)\\
& \quad 
-
\Big( 
\sum_{k =0}^{A} r_{k}\frac{\#_{k} \eta}{L} 
\Big)
\sum_{i=1}^{L} 
 \frac{( \eta_{i}-A )_{+}}{N}
 \partial_{x_{i}}f\circ \embd  (\eta)\\
&\quad +
\gamma_{N}( \eta) 
\sum_{j=1}^{L} 
 \frac{r_{ A}}{L} 
\mathds{1}_{ \eta_{j}=A}
 \partial_{x_{j}}f\circ \embd  (\eta)
+o_{f}(1)
\,,
\end{aligned}
\end{equation}
where we dropped the term $k=A$ in the first sum, since $
\partial_{y_{A+1}}f = \partial_{y_{A}} f=0$.
Note that the second--to--last term in \eqref{e_c2b_supp4} equals zero, because 
$\partial_{x_{j}} \varphi_{m}\big( \tfrac{ \eta_{j}-A}{N} \big)
\mathds{1}_{\eta_{j}=A}=0$, for all $ m \geqslant 2$. 
\\

\noindent \hbox{\tiny $\blacksquare$} \underline{Fast to fast phase}, i.e.~$ \eta_{i}> A $ and $
\eta_{j}> A $.
The treatment of the this part of the generator is more subtle,
as we have to observe exact cancellations. To this end, 
Taylor's approximation theorem  \eqref{e_taylor}, yields
\begin{equation}\label{e_c2c_supp1}
\begin{aligned}
&f \circ \embd ( \eta^{i,j}) - f\circ \embd  (\eta)
= 
\frac{1}{N} 
\big[\partial_{x_{j}}f
- \partial_{x_{i}}f\big]\circ \embd (\eta)\\
&\quad + \frac{1}{2N^{2}} \big[ 
\partial_{x_{i}}^{2}f
+
\partial_{x_{j}}^{2}f
- 2 \partial_{x_{i}}\partial_{x_{j}}f\big]\circ \embd (\eta)
+ o_{f}\big ( \tfrac{1}{L^{2}}\big)\,.
\end{aligned}
\end{equation}
$\bullet$ We start by showing that contributions of first--order derivatives in
\eqref{e_c2c_supp1} to the generator $\mathfrak{L}_{L,N}$ are controlled by a
uniform constant, using the approximate
symmetry of rates.
 First, we
rearrange the expression to
\begin{equation}\label{e_c2c_sym_supp1}
\begin{aligned}
&\sum_{i,j=1}^{L} 
u_{1} ( \eta_{i}) u_{2}( \eta_{j}) \mathds{1}_{ \eta_{i}>A,\, \eta_{j}> A}
\frac{1}{N} 
\big[\partial_{x_{j}}f
- \partial_{x_{i}}f\big]\circ \embd ( \eta) \\
&=
\sum_{j =1}^{L}
\partial_{x_{j}}f\circ \embd ( \eta) 
\mathds{1}_{ \eta_{j}> A}
\bigg\{
u_{2} ( \eta_{j})  \sum_{i =1}^{L}
\frac{ u_{1}( \eta_{i})}{N}\mathds{1}_{ \eta_{i}>A}
-
u_{1}( \eta_{j} )  \sum_{i =1}^{L}
\frac{  u_{2}( \eta_{i})}{N}
\mathds{1}_{ \eta_{ i}>A}
 \bigg\}\,.
\end{aligned}
\end{equation}
Note that the term inside the parenthesis can be bounded in terms of
\begin{equation*}
\begin{aligned}
&\bigg\vert
u_{2} ( \eta_{j})  \sum_{i =1}^{L}
\frac{ u_{1}( \eta_{i})}{N}\mathds{1}_{ \eta_{i}>A}
-
u_{1}( \eta_{j} ) \sum_{i =1}^{L}
\frac{  u_{2}( \eta_{i})}{N}
\mathds{1}_{ \eta_{ i}>A}
 \bigg\vert\\
&\leqslant
\bigg\vert
u_{2} ( \eta_{j})  -
u_{1}( \eta_{j} )   \bigg\vert
\sum_{i =1}^{L}
\frac{ u_{1}( \eta_{i})}{N}\mathds{1}_{ \eta_{i}>A}
  + 
\bigg\vert
\sum_{i =1}^{L}
\frac{ u_{1}( \eta_{i})}{N}\mathds{1}_{ \eta_{i}>A}
-
 \sum_{i =1}^{L}
\frac{  u_{2}( \eta_{i})}{N}
\mathds{1}_{ \eta_{ i}>A}
 \bigg\vert
u_{1}( \eta_{j} ) 
\\
& \leqslant 
2 \zeta_{L} \Big( \gamma_{N} ( \eta) + \frac{\zeta_{L} L}{N}\Big)
+ 
2 \zeta_{L}\, \#_{>A} \eta \Big(\frac{ \eta_{j}-A}{N} + \frac{\zeta_{L}
}{N}\Big)\,,
\end{aligned}
\end{equation*}
where we applied the triangle inequality first, and then 
used  Assumption~\ref{ass}, \eqref{eq_gamma_cf_rates} and that $ u_{1}(
\eta_{j}) \leqslant \eta_{j} -A+ \zeta_{L}$,
in the last inequality.
Thus, \eqref{e_c2c_sym_supp1} is upper bounded in terms of 
\begin{equation}\label{e_c2c_sym0}
\begin{aligned}
&\bigg\vert\sum_{i,j=1}^{L} 
u_{1} ( \eta_{i}) u_{2}( \eta_{j}) \mathds{1}_{ \eta_{i}>A,\, \eta_{j}> A}
\frac{1}{N} 
\big[\partial_{x_{j}}f
- \partial_{x_{i}}f\big]\circ \embd ( \eta)\bigg\vert \\
& \leqslant 
2\sum_{j =1}^{L}
\big|\partial_{x_{j}}f\circ \embd ( \eta) \big|
\mathds{1}_{ \eta_{j}>A}
\bigg\vert \zeta_{L} \Big( \gamma_{N} ( \eta) + \frac{\zeta_{L} L}{N}\Big)
+ 
 \zeta_{L} \, \#_{>A} \eta\Big(\frac{ \eta_{j}-A}{N} + \frac{\zeta_{L}}{ N}\Big)
\bigg\vert\\
& \leqslant 4
\zeta_{L}
\sum_{j =1}^{L}
\big|\partial_{x_{j}}f\circ \embd ( \eta)\big|
+
2
\zeta_{L}
\, \#_{>A} \eta
\sum_{j =1}^{L}
 \frac{\eta_{j}-A}{N}
\mathds{1}_{ \eta_{j}>A}
\big|\partial_{x_{j}}f\circ \embd ( \eta)\big|\,,
\end{aligned}
\end{equation}
where the last inequality holds uniformly for $L,N$ large enough.
Lastly, we use that $ \sum_{i= 1 }^{\infty} (1+ x_{j}) |\partial_{j}f (x) | \leqslant
C_{f}< \infty$ uniformly, 
for $ f \in \mD$. Then 
\eqref{e_c2c_sym0} can be conveniently expressed in terms of  
\begin{equation}\label{e_c2c_sym}
\begin{aligned}
&\bigg\vert\sum_{i,j=1}^{L} 
u_{1} ( \eta_{i}) u_{2}( \eta_{j}) \mathds{1}_{ \eta_{i}>A,\, \eta_{j}> A}
\frac{1}{N} 
\big[\partial_{x_{j}}f
- \partial_{x_{i}}f\big]\circ \embd ( \eta)\bigg\vert
\lesssim C_{f} \, \zeta_{L} \Big(1 +  \#_{>A} \eta \Big)\,.
\end{aligned}
\end{equation}
$\bullet$ We now proceed to the contribution of the second--order terms in
\eqref{e_c2c_supp1} to $\mathfrak{L}_{L,N}$, which can be conveniently
rewritten into
\begin{equation}\label{eq_supp_c2c_5}
\begin{aligned}
&\frac{1}{2 N^{ 2}} 
\sum_{\substack{i,j=1\\ i \neq j }}^{L} 
u_{1} ( \eta_{i}) u_{2}( \eta_{j}) \mathds{1}_{ \eta_{i}>A,\, \eta_{j}> A}
\big[ 
\partial_{x_{i}}^{2}f
+
\partial_{x_{j}}^{2}f
- 2 \partial_{x_{i}}\partial_{x_{j}}f\big]\circ \embd (\eta)\\
& = 
\sum_{i=1}^{L}
\frac{ ( \eta_{i} -A)_{+}}{N}
\Big(
\gamma_{N}( \eta)
-
\frac{ ( \eta_{i} -A)_{+}}{N}
\Big)
\partial_{x_{i}}^{2}f\circ \embd  (\eta)\\
& \qquad -
\sum_{j=1}^{L}
\sum_{\substack{i=1\\ i \neq j}}^{L}
\frac{ ( \eta_{j} -A)_{+}}{N}
\frac{ ( \eta_{i} -A)_{+}}{N}
\partial_{x_{i}} \partial_{x_{j}} f\circ \embd  (\eta)+ o_{f}(1)\,,
\end{aligned}
\end{equation}
where we also used Assumption~\ref{ass} and \eqref{eq_gamma_cf_rates}.\\

\noindent \hbox{\tiny $\blacksquare$} \underline{Combining the estimates}
 \eqref{e_b2b_supp1}, \eqref{e_b2c_supp4}, \eqref{e_c2b_supp4},
\eqref{e_c2c_sym}, and
\eqref{eq_supp_c2c_5},
 we overall just proved that
\begin{equation}\label{e_supp1_approx_L}
\begin{aligned}
\mathfrak{L}_{L,N} ( f \circ \embd) (\eta)
&=
\sum_{i,j =1}^{L}
\frac{ ( \eta_{i} -A)_{+}}{N}
\Big(
\gamma_{N}( \eta)
\delta_{i,j}
-
\frac{ ( \eta_{j} -A)_{+}}{N}
\Big)
\partial_{x_{i}}\partial_{x_{j}}f\circ \embd  (\eta)
\\
& \qquad 
-
\Big( 
\sum_{k =0}^{A } ( r_{k}-  q_{k})\frac{\#_{k} \eta}{L} 
\Big)
\sum_{i=1}^{L} 
 \frac{( \eta_{i}- A )_{+}}{N}
 \partial_{x_{i}}f\circ \embd  (\eta)\\
&\quad+
\frac{N}{L} 
\gamma_{N} ( \eta) 
\sum_{k=1}^{A} 
q_{k}
\frac{\#_{k}\eta}{L} 
(
\partial_{y_{ k-1 }} f
-\partial_{y_{ k}} f)\circ \embd  (\eta)
\\
&\qquad+
\frac{N}{L} 
\gamma_{N} ( \eta) 
\sum_{k=0}^{A-1} 
r_{k}
\frac{\#_{k}\eta}{L} 
( 
\partial_{y_{ k+1 }} f
-\partial_{y_{ k}} f)\circ \embd  (\eta)\\
& \quad 
+
\mO_{f} \Big(
\zeta_{L}   \#_{>A} \eta \Big)
+o_{f}(1)\,,
\end{aligned}
\end{equation}
uniformly over all $ \eta \in \Omega_{L,N}$.
The expression \eqref{e_supp1_approx_L} is already reminiscent of $ \mL f 
 (\embd (\eta))$. 
All that is left is to replace $\#_{A} \eta$ and  the $ \gamma_{N} ( \eta)$ in \eqref{e_supp1_approx_L} with suitable
expressions in terms of $ \embd ( \eta)$.
To this end, we first note 
\begin{equation}\label{eq_A_occ}
\begin{aligned}
\frac{\#_{A} \eta}{L} = 1 - \sum_{k =0}^{A-1}
\frac{\#_{k} \eta}{ L} - \frac{\#_{>A} \eta}{L} \,,
\end{aligned}
\end{equation}
such that
\begin{equation*}
\begin{aligned}
 \gamma_{N}( \eta ) 
&=
\sum_{i =1}^{L} \frac{ \eta_{i}- A}{N} \mathds{1}_{\eta_{i}>A}
=
1 - \frac{L}{N} \Big(\sum_{k=0}^{A} k \frac{\#_{k} \eta}{L}  
+ A \frac{\#_{> A} \eta}{L}\Big) \\
&=
1 - \frac{L}{N} \Big( A+
\sum_{k =0}^{A-1} (k-A)  \frac{\#_{k} \eta}{L}
\Big)
=  \gamma \big( (
\tfrac{\#_{k}}{L})_{k=0}^{A-1}\big) 
+ o(1)  \,,
\end{aligned}
\end{equation*}
where we used \eqref{eq_A_occ} in the second--to--last step, before rearranging the terms.
Therefore, we can replace all $ \gamma_{N}( \eta ) $ with $ \gamma \big( (
\tfrac{\#_{k}}{L})_{k=0}^{A-1}\big) = \gamma ( \embd ( \eta))$.

Finally, we replace $\#_{A} \eta$ in  \eqref{e_supp1_approx_L}
with $ \overline{\#}_{A} \eta := L - \sum_{k =0}^{A-1} \#_{k} \eta$.
According to \eqref{eq_A_occ}, the corresponding error can be absorbed
into an error term $\mO_{f}\big( \tfrac{\#_{>A} \eta }{L}\big)$.
Note that $\mO_{f}\big( \zeta_{L}\, \#_{>A} \eta )$ in
\eqref{e_supp1_approx_L} can be absorbed into this
new error term, because $ \zeta_{L} = \mO ( L^{-1})$.
For convenience, we will write $ \overline{\#}_{k} \eta = \#_{k} \eta  $, $ k=0,
\ldots, A-1$.

Overall, we  recover
\begin{equation}\label{e_supp2_approx_L}
\begin{aligned}
\mathfrak{L}_{L,N} ( f \circ \embd) (\eta)
&=
\sum_{i,j =1}^{L}
\frac{ ( \eta_{i} -A)_{+}}{N}
\Big(
\gamma (\embd( \eta))
\delta_{i,j}
-
\frac{ ( \eta_{j} -A)_{+}}{N}
\Big)
\partial_{x_{i}}\partial_{x_{j}}f\circ \embd  (\eta)
\\
& \qquad 
-
\theta ( \embd( \eta) )
\sum_{i=1}^{L} 
 \frac{( \eta_{i}- A )_{+}}{N}
 \partial_{x_{i}}f\circ \embd  (\eta)\\
&\quad+
\frac{N}{L} 
\gamma (\embd( \eta))
\sum_{k=1}^{A} 
q_{k}
\frac{ \overline{\#}_{k}\eta}{L} 
(
\partial_{y_{ k-1 }} f
-\partial_{y_{ k}} f)\circ \embd  (\eta)
\\
&\qquad+
\frac{N}{L} 
\gamma (\embd( \eta))
\sum_{k=0}^{A-1} 
r_{k}
\frac{ \overline{\#}_{k}\eta}{L} 
( 
\partial_{y_{ k+1 }} f
-\partial_{y_{ k}} f)\circ \embd  (\eta)\\
& \quad 
+
\mO_{f} \Big(
\frac{ \#_{>A} \eta}{L} \Big)
+o_{f}(1)\\
&=
\mL f ( \embd ( \eta))
+ 
\mO_{f} \Big(
\frac{ \#_{>A} \eta}{L} \Big)
+o_{f}(1)
\,.
\end{aligned}
\end{equation}
This concludes the proof.
\end{proof}

\appendix
\section{Multi-loci Wright--Fisher diffusions}\label{sec_app_WF}

In this appendix, we consider the general setting of multi-loci Wright--Fisher
diffusions, which are characterised by infinitesimal generators of the form 
\begin{equation}\label{e_gen_WF}
\begin{aligned}
\mathtt{L} g ( \mathbf{x})
= 
\sum_{k= 1}^{n}
\left( 
\sum_{i,j =1}^{d_{k}} a^{(k)}_{i,j}( \mathbf{x})
\partial_{ \mathbf{x}^{(k)}_{i}  }\partial_{ \mathbf{x}^{(k)}_{j} }g( \mathbf{x})
+ \sum_{i =1}^{d_{k}} b_{i}^{(k)}( \mathbf{x}) \partial_{  \mathbf{x}^{(k)}_{i}}g( \mathbf{x})\,,
 \right)\,,
\end{aligned}
\end{equation} 
where  $ g \in C^{2}( K_{\bf d}) $, with
\begin{equation}\label{e_K_prod}
\begin{aligned}
\mathbf{x} = \big( \mathbf{x}^{(1)} , \ldots, \mathbf{x}^{n}\big) \in
K_{d_{1}}\times \cdots \times K_{d_{n}}=:
 K_{ \bf
d} \,, 
\end{aligned}
\end{equation}
and $K_{d_{k}}$ defined in terms of
\begin{equation*}
\begin{aligned}
K_{d}:= \Big\{x \in [0, 1]^d\,:\, \sum_{i=1}^{d} x_{k} \leqslant 1 \Big\}\,.
\end{aligned}
\end{equation*}
We assume the diffusion part of the operator to be of the typical Wright--Fisher
form:
\begin{equation*}
\begin{aligned}
a^{(k)}_{i,j}( \mathbf{x}) = \alpha_{k}\,   \mathbf{x}_{i}^{(k)} ( \delta_{i,j}
-  \mathbf{x}_{j}^{(k)})\,, \quad \text{with } \ \alpha_{k} \in [0, \infty)\,.
\end{aligned}
\end{equation*}
Moreover, to guarantee that the diffusion does not exit the multi-simplex $
K_{\bf d}$, we require 
that
\begin{equation}\label{eq_drfit_coeff1}
\begin{aligned}
b^{(k)}_{i}( \mathbf{x} ) \geqslant 0 \,, \text{ for all  }
\mathbf{x} \text{ such that }
\mathbf{x}_{i}^{(k)}=0\,, 
\end{aligned}
\end{equation}
\begin{equation}\label{eq_drfit_coeff2}
\begin{aligned}
\text{and}
\quad
\sum_{i =1 }^{d_{k}} b^{(k)}_{i}(\mathbf{x} ) \leqslant  0\,,  \text{ for all }
\mathbf{x}  \text{ such that }
\, \sum_{i = 1}^{d_{k}} \mathbf{x}_{i}^{(k)} = 1\,.
\end{aligned}
\end{equation}
The closure of the operator $\mathtt{L}$ is the infinitesimal generator of a
Feller process supported on 
$C([0, \infty), K_{\bf d})$, provided $ b^{(k)}_{i} \in C^{ 4}( K_{\bf
d})$. 
To see this, one can follow for example the proof of \cite{Ethier_1976} in the
multi-loci setting, see also \cite{Ethier_1979} for a two-loci
setting.

In the following, we will only require that the martingale problem for
$\mathtt{L}$ admits solutions in $C([0, \infty), K_{\bf d})$, which holds
under the weaker assumption that the coefficients $b^{(k)}_{i}$'s are Lipschitz continuous,
see \cite[Proposition~1]{Ethier_1976}. Uniqueness of solutions to the
martingale problem should also extend to this case, see \cite[p.~406]{Ethier_1979}.
\\

To approximate solutions of the martingale problem for $\mL$ in
Section~\ref{sec_wp_pd}, we consider the two--loci case
with $\mathtt{L}^{(M)}$ as in \eqref{e_gen_WF} and $ d_{1}= M-1$,
$d_{2} = A$.
For convenience, we decompose $
\mathtt{L}^{(M)}$ into $\mathtt{L}^{(M)} =
\mathtt{A}^{(M)}_{ 1,\beta} + \mG$, with
 $\mG$ as in \eqref{eq_gen_G} and
\begin{equation}\label{eq_def_AM_op}
\begin{aligned}
\mathtt{A}^{(M)}_{1, \beta} g ( z,y) := 
\sum_{i,j =1}^{M-1} 
z_{i} ( \delta_{i,j} - z_{j} )
\partial_{z_{i} , z_{j} }^{2}g(z,y)
+ 
\overline{\beta}(y)
\sum_{i =1}^{M-1} \big( \tfrac{1 }{M-1} (1- z_{i} ) -  z_{i} \big)
\partial_{ z_{i}}g(z,y)
\,,
\end{aligned}
\end{equation} 
where
 $ \overline{\beta} \in C (S_{Y}, [0, \infty))$ is a
(Lipschitz) continuous extension of
\begin{equation}\label{e_def_beta}
\begin{aligned}
\beta( y) =
\sum_{k=0}^{A-1}
b_{k}(y) 
\frac{\partial_{y_{k}} \gamma(y)}{ \gamma(y)}
+
\theta (y) \geqslant 0 
 \,, \quad y \in \overline{\supp \gamma}\,,
\end{aligned}
\end{equation}
see Assumption~\ref{ass_def_L}.

The drift coefficients of this two--loci diffusion model are Lipschitz continuous, by
Assumption~\ref{ass_def_L}.(b)~and~\ref{ass_def_L}.(d).
Thus, the martingale problem for $\mathtt{L}^{(M)}$ admits solutions $(
Z^{(M)}(t), Y(t))_{t \geqslant 0}$ with sample paths in $ C (
[0, \infty), K_{M-1}\times S_{Y}) $.
Note that the boundary conditions \eqref{eq_drfit_coeff1} and
\eqref{eq_drfit_coeff2} remain satisfied for $
\overline{\beta} \geqslant 0 $, due to the specific form of the drift coefficients.

\subsection*{Assumption~\ref{ass_def_L} is satisfied for the PD--limit of the inclusion process}

For the inclusion process with non--trivial bulk, we don't rely on the
approximation of the limiting dynamics by Wright--Fisher diffusions, since the
inclusion process itself provides a construction of solutions to the martingale
problem for $ \mL $.

We still want to verify that the setting of Theorem~\ref{thm_main} provides an
example for which Assumption~\ref{ass_def_L} is satisfied.  
To this end, we check each assumption individually. 

\begin{enumerate}
\item[(a)]
The coefficients of the generator for the control process are of the form 
\begin{equation*}
\begin{aligned}
b_{k}(y)
=
\rho\, \gamma(y)
\big(
q_{k +1} y_{k +1} + r_{k -1} y_{k -1}
- (q_{k}+ r_{k}) y_{k}
\big)\,,
\end{aligned}
\end{equation*}
which clearly vanish whenever $ y \in S_{Y}^{*}$.

\item[(b)]
 The boundary conditions \eqref{e_cond_b1} and \eqref{e_cond_b2} hold,
since 
\begin{equation*}
\begin{aligned}
b_{k}(y)
=
\rho\, \gamma(y)
\big(
q_{k +1} y_{k +1} + r_{k -1} y_{k -1}
\big)
\geqslant 0\,, \quad \text{whenever } y_{k}=0\,,
\end{aligned}
\end{equation*}
and 
\begin{equation*}
\begin{aligned}
\sum_{k =0}^{ A-1} b_{k}(y) 
= 
\rho \, \gamma(y) 
\Big( q_{A} y_{A}
- r_{A-1} y_{A-1}\Big)
\leqslant 0\,,\quad \text{whenever } \sum_{k =0}^{A-1} y_{k}= 1 \,,
\end{aligned}
\end{equation*}
because in this case $y_{A}=1- \sum_{k=0}^{A-1} y_{k}= 0$.
Moreover, the $ b_{k}$'s are Lipschitz continuous.

\item[(c)]
 Lipschitz continuity of the control $ \theta(y) =
\sum_{k =0}^{A } (r_{k}- q_{k}) y_{k} $ is immediate, similar for $ \gamma$
because it is the positive part of the function
\begin{equation*}
\begin{aligned}
\gamma( y) 
= 1- \frac{1}{ \rho} \sum_{k=0}^{A} k y_{k}
=
1+
\frac{1}{\rho} 
\bigg(
\sum_{k =0}^{A-1} (A-k) y_{k} 
-A
\bigg)\,.
\end{aligned}
\end{equation*}

\item[(d)]
 Lastly, $ \beta$ on $ \supp \gamma$ takes the form 
\begin{equation*}
\begin{aligned}
\beta(y) 
&= \theta (y) 
+
\sum_{k =0}^{A-1} 
(A-k)
\big(
q_{k +1} y_{k +1} + r_{k -1} y_{k -1}
- (q_{k}+ r_{k}) y_{k}
\big)\\
& = 
\sum_{k =0}^{A } (r_{k}- q_{k}) y_{k} 
+
\left\{
\sum_{k =0}^{A}
q_{k} y_{k} - A\, q_{0} y_{0}
-
\sum_{k =0}^{A-1}
r_{k} y_{k}
\right\}
= r_{A}y_{A} \geqslant 0 \,,
\end{aligned}
\end{equation*}
where we used that $ \partial_{ y_{k}} \gamma (y)= \rho^{-1} (A-k)$ and the
explicit form of $ \theta( \cdot)$, before rearranging and cancelling the
majority of summands. Recall that $ q_{0} = 0$.

\end{enumerate}

\bibliography{cites}
\bibliographystyle{alpha}

\end{document}